\documentclass[10pt]{article}
\usepackage[ansinew]{inputenc}
\usepackage{amsmath}
\usepackage{amsfonts}
\usepackage{amssymb}
\usepackage{amsthm}
\usepackage{newlfont}
\usepackage{amsbsy}
\usepackage{bm}
\usepackage{color}
\usepackage{epsfig}
\usepackage{dsfont}
\usepackage{latexsym}
\usepackage{array}
\usepackage{longtable}
\usepackage{enumerate}
\usepackage{multicol}
\usepackage{mathrsfs}
\usepackage{wasysym}
\usepackage{graphics}
\usepackage{graphicx}
\DeclareMathOperator\supp{supp}
\usepackage{hyperref}
\hypersetup{
    colorlinks,
    citecolor=blue,
    filecolor=black,
    linkcolor=red,
    urlcolor=black
}

\numberwithin{equation}{section}

\newtheorem{lemma}{Lemma}[section]
\newtheorem{theorem}{Theorem}[section]
\newtheorem{corol}{Corollary}[section]
\newtheorem{defin}{Definition}[section]
\newtheorem{prop}{Proposition}[section]
\newtheorem*{rem}{\it Remark}
\newtheorem{claim}{Claim}
\newtheorem*{remarks}{Remarks}

\date{}
\title{\textbf{The IVP for a higher dimensional version of the Benjamin-Ono equation in weighted
Sobolev spaces}}
\author{Oscar G. Ria\~no  \thanks{IMPA - Instituto de Matem\'atica Pura e Aplicada,
E-mail: {\tt ogrianoc@impa.br}}}

\begin{document}

\maketitle 

\begin{abstract} 
We study the initial value problem associated to a higher dimensional version of the Benjamin-Ono equation. Our purpose is to establish local well-posedness results in weighted Sobolev spaces and to determinate according to them some sharp unique continuation properties of the solution flow. In consequence, optimal decay rate for this model is determined. A key ingredient is the deduction of a new commutator estimate involving Riesz transforms.
\end{abstract}



\section{Introduction}

This work is concerned with the initial value problem (IVP) for a higher dimensional version of the Benjamin-Ono equation;
\begin{equation}\label{HBO-IVP}\tag{HBO}
\begin{cases}
  \partial_t u - \mathcal{R}_1\Delta u+u\partial_{x_1} u=0,\hskip 15pt x\in \mathbb{R}^d,\,  t\in \mathbb{R}, \\
  u(x,0)= u_0.
  \end{cases}
\end{equation}
where $d\geq 2$, $\mathcal{R}_1=-\partial_{x_1}(-\Delta)^{-1/2}$ denotes the Riesz transform  with respect to the first coordinate defined by the Fourier multiplier operator with symbol $-i\xi_1 |\xi|^{-1}$, and $\Delta$ stands for the Laplace operator in the spatial variables $x\in\mathbb{R}^d$.

When $d=1$, the  Riesz  transform  coincides  with  the  Hilbert  transform, and so we recover the well-known Benjamin-Ono equation, see \cite{Ponce1991,molinetPilodBO,TaoBO,Kenig,FonPO} and the references therein. 
%
 When $d=2$, the \eqref{HBO-IVP} equation preserves its physical relevance, it describes the dynamics of three-dimensional slightly nonlinear disturbances in boundary-layer shear flows, without the assumption that the scale of the disturbance is smaller along than across the flow, see for instance~\cite{A,PS,VS}. Existence and decay rate of Solitary-wave solutions were studied in~\cite{M}. 

Some recent works have been devoted to establish that the IVP associated to \eqref{HBO-IVP} is locally well-posed (LWP) in the space $H^s(\mathbb{R}^d)$, $s\in \mathbb{R}$ and $d\geq 2$. Here we adopt Kato's notion of \emph{well-posedness}, which consists of  existence, uniqueness, persistence property (i.e.,  if the data $u_0\in X$ a function space, then the corresponding solution $u(\cdot)$ describes a continuous curve
in $X$, $u \in C([0,T ];X), T > 0$), and continuous dependence of the map data-solution. 
Regarding the IVP for \eqref{HBO-IVP}, in \cite{linaO} LWP was deduced for $s>5/3$ when $d=2$ and for $s>(d+1)/2$ when $d\geq 3$. In \cite{RobertS},  LWP  was improved to the range $s>3/2$ in the case $d=2$. Up to our knowledge there are no results concerning global well-posedness (GWP) in the current literature. It is worthwhile to mention that local well-posedness issues have been addressed by compactness methods, since one cannot solve the IVP related to \eqref{HBO-IVP} by a Picard iterative method implemented on its integral formulation for any initial data in the Sobolev space $H^s(\mathbb{R}^d)$, $d\geq 2$ and $s\in \mathbb{R}$. This is a consequence of the results deduced in \cite{linaO}, where it was established that the flow map data-solution $u_0\mapsto u$ for \eqref{HBO-IVP} is not of class $C^2$ at the origin from $H^s(\mathbb{R}^d)$ to $H^s(\mathbb{R}^d)$ $d\geq 2$. 
\\ \\
Real solutions of \eqref{HBO-IVP} formally satisfy  at least three conservation laws (time invariant quantities) 
\begin{equation}\label{Conservationlaws}
\begin{aligned}
    I(u)&=\int u(x,t) \, dx, \\
    M(u)&=\int u^2(x,t) \, dx, \\
    H(u)&=\int \left|(-\Delta)^{1/4} u(x,t)\right|^2- \frac{1}{3}u^3(x,t) \, dx. 
\end{aligned}
\end{equation}
This work is intended to determinate if for a given initial data in the Sobolev space $H^s(\mathbb{R}^d)$ with some additional decay at infinity (for instance polynomial), it is expected that the corresponding solution of \eqref{HBO-IVP} inherits this behavior.  Such matter has been addressed before for the Benjamin-Ono equation in \cite{FLinaPonceWeBO,FonPO}, showing that in general polynomial type decay is not preserved by the flow of this model. Here as a consequence of our results, we shall determinate that the same conclusion extends to the \eqref{HBO-IVP} equation.

 Let us now state our results. Our first consequence is motivated from the fact that the weight function $\langle x \rangle^{r}=(1+|x|^2)^{r/2}$ is smooth with bounded derivatives when $r\in [0,1]$. This property allows us to consider well-posedness issues for a more general class of weights.
\begin{prop}\label{wellposwei}
Let $\omega$ be a smooth weight with all its first and second derivatives bounded. Then, the IVP \eqref{HBO-IVP} is locally well-posed in $H^s(\mathbb{R}^d)\cap L^2(\omega^2 \, dx )$ for all $s>s_d$, where $s_2=5/3$ and $s_d=d/2+1/2$ for $d\ge 3$.
\end{prop}
The proof of Proposition \ref{wellposwei} is similar in spirit to that in \cite{CUNHA} for a two-dimension model. A remarkable difference is that our recent results in \cite{FonPO} enable us to prove Proposition \ref{wellposwei} in Sobolev spaces of lower regularity compared with those obtained by implementing a parabolic regularization argument. 

Next, we discuss LWP for the IVP \eqref{HBO-IVP} in weighted Sobolev spaces
\begin{equation}\label{weightespace}
Z_{s,r}(\mathbb{R}^d)=H^{s}(\mathbb{R}^d)\cap L^{2}(| x|^{2r} \, dx), \hspace{0.5cm} s,r \in \mathbb{R}
\end{equation}
and 
\begin{equation}\label{weightespacedot}
\dot{Z}_{s,r}(\mathbb{R}^d)=\left\{f\in H^{s}(\mathbb{R}^d)\cap L^{2}(| x|^{2r} \, dx):\, \widehat{f}(0)=0 \right\}, \hspace{0.5cm} s,r \in \mathbb{R}.
\end{equation}
In order to obtain a relation between differentiability and decay in the spaces \eqref{weightespace}, we notice that the linear part of the equation \eqref{HBO-IVP} $\mathcal{L}=\partial_t-\mathcal{R}_1 \Delta$ commutes with the operators $$\Gamma_l=x_l+t\delta_{1,l}(-\Delta)^{1/2}+t\partial_{x_l}\mathcal{R}_1, \hspace{0.2cm} l=1,\dots, d,$$
where $\delta_{1,l}$ denotes the Kronecker delta function with $\delta_{1,l}=1$ if $l=1$ and zero otherwise, thus one has
\begin{equation*}
[\mathcal{L},\Gamma_l]=\mathcal{L}\Gamma_l-\Gamma_l\mathcal{L}=0.
\end{equation*}
For this reason, it is natural to study well-posedness in weighted Sobolev spaces $Z_{s,r}(\mathbb{R}^d)$ where the balancing between decay and regularity satisfies the relation, $ r\leq s$. 
\begin{rem}
For the sake of brevity, from now on we shall state our results for the \eqref{HBO-IVP} equation only for dimensions two and three. Actually, it will be clear from our arguments that solutions of this model in the spaces \eqref{wellposwei} behave quite different in each of these dimensions. Nevertheless, following our ideas one can extend the ensuing conclusions to arbitrary even and odd dimensions.
\end{rem}
\begin{theorem}\label{localweigh}
Consider $d=2,3$. Let $s>s_d$ where $s_2=5/3$ and $s_3=2$.
\begin{itemize}
\item[(i)] If $r\in[0,d/2+2)$ with $r\leq s$, then the IVP associated to \eqref{HBO-IVP} is locally well-posed in $Z_{s,r}(\mathbb{R}^d)$.
\item[(ii)] If $r\in[0,d/2+3)$ with $r\leq s$, then the IVP associated to \eqref{HBO-IVP} is locally well-posed in $\dot{Z}_{s,r}(\mathbb{R}^d)$.
\end{itemize}
\end{theorem}

The proof of Theorem \ref{localweigh} is adapted from the arguments used by Fonseca and Ponce in \cite{FonPO} and Fonseca, Linares and Ponce in \cite{FLinaPioncedGBO}. Additional difficulties arise from extending these ideas to the \eqref{HBO-IVP} equation, since here we deal with a several variables model involving Riesz transform operators. Among them, the commutator relation between $\mathcal{R}_1$ and a polynomial of a certain higher degree requires to infer weighted estimates for derivatives of negative order. In this regard, as a further consequence of the proof of Theorem \ref{localweigh} we deduce.
\begin{corol}
Consider $d=2,3$ and $r_0 \in [0,d/2)$. Let $u\in C([0,T];\dot{Z}_{s,r}(\mathbb{R}^d))$ be a solution of the IVP \eqref{HBO-IVP} with $(d/2+2)^{-} \leq r \leq s$. Then 
\begin{equation*}
 |\nabla|^{-1}u \in C([0,T];L^2(|x|^{2r_0}\, dx)). 
\end{equation*}
\end{corol}
In the above display the operator $|\nabla|^{-1}$ is defined by the Fourier multiplier $|\xi|^{-1}=(\xi_1^2+\dots+\xi_d^2)^{-1/2}$. 
Next we state some continuation principles for the \eqref{HBO-IVP} equation.
\begin{theorem}\label{sharpdecay}
Assume that $d=2,3$. Let $u$ be a solution of the IVP associated to \eqref{HBO-IVP} such that $u\in C([0,T];Z_{2^{+},2}(\mathbb{R}^2))$ when $d=2$ and $u\in C([0,T];Z_{3,3}(\mathbb{R}^3))$ when $d=3$. If there exist two different times $t_1,t_2\in [0,T]$ for which
\begin{equation*}
u(\cdot,t_j)\in Z_{d/2+2,d/2+2}(\mathbb{R}^d), \, j=1,2 \hspace{0.3cm} \text{ then } \hspace{0.3cm} \widehat{u}_{0}(0)=0.
\end{equation*} 
\end{theorem}
In Theorem \ref{sharpdecay}, $u\in Z_{2^{+},2}(\mathbb{R}^2)$ means that $u\in H^{s^{+}}(\mathbb{R}^2) \cap L^{2}(|x|^4 dx)$, where there exists a positive number $\epsilon\ll 1$ such that $u\in H^{s+\epsilon}(\mathbb{R}^2)$. 
\begin{theorem}\label{threetimesharp} 
Suppose that $d=2,3$, $r_2=3$ and $r_3=4$. Let $u\in C([0,T];\dot{Z}_{r_d,r_d}(\mathbb{R}^d))$ be a solution of the IVP associated to \eqref{HBO-IVP}. If there exist three different times $t_1,t_2,t_3\in [0,T]$ such that
\begin{equation*}
u(\cdot,t_j)\in Z_{d/2+3,d/2+3}(\mathbb{R}^d), \, j=1,2,3 \hspace{0.3cm} \text{ then } \hspace{0.3cm} u(x,t)=0.
\end{equation*} 
\end{theorem}
It is worth pointing out that the deduction of Theorems \ref{sharpdecay} and \ref{threetimesharp} is more involving in the odd dimension case, where the decay rates $d/2+2$ and $d/2+3$ are not integer numbers. Roughly speaking, transferring decay to regularity in the frequency domain, on this setting one has to deal with an extra $1/2$-fractional derivative to achieve these conclusions. 

We remark that similar unique continuation properties have been established for the Benjamin-Ono equation in \cite{FonPO} and the dispersion generalized Benjamin-Ono equation in \cite{FLinaPioncedGBO}. A difference in the present work is that our proof of Theorems \ref{sharpdecay} and \ref{threetimesharp} incorporates an extra weight in the frequency domain, which allows us to consider less regular solutions of \eqref{HBO-IVP} to reach these consequences.

\begin{remarks}
\begin{itemize}
\item[(i)] When $d=1$, the conclusion of Theorem \ref{localweigh} coincides with the decay rates showed for the Benjamin-Ono equation in \cite[Theorem 1]{FonPO}. In this sense our results can be regarded as a generalization of those derived by the Benjamin-Ono equation. As a matter of fact, Theorem \ref{localweigh} tell us that an increment in the dimension allows a $1/2$ larger decay with respect to the preceding setting.
\item[(ii)] The restrictions on the Sobolev regularity  stated in Proposition \ref{wellposwei} and Theorem \ref{localweigh}
are imposed from our recent results in \cite{linaO}, which assure that under such considerations the solution $u(x,t)$ satisfies
\begin{equation}\label{normWinf}
    u\in L^1\big([0,T);W^{1,\infty}(\mathbb{R}^d)\big),
\end{equation}
where the Sobolev space $W^{1,\infty}(\mathbb{R}^d)$ is defined as usual with norm $\|f\|_{W^{1,\infty}}:=\|f\|_{L^{\infty}}+\|\nabla f\|_{L^{\infty}}$. The property  \eqref{normWinf} is essential to establish LWP in $Z_{s,r}(\mathbb{R}^d)$.
\item[(iii)] Theorem \ref{sharpdecay}  shows that the decay $r=(d/2+2)^{-}$ is the largest possible for arbitrary initial datum. In this regard Theorem \ref{localweigh} (i) is sharp. In addition, Theorem \ref{sharpdecay} shows that if $u_0 \in \mathbb{Z}_{s,r}(\mathbb{R}^d)$ with $d/2+2\leq r \leq s$ and $\widehat{u_0}(0)\neq 0$, then the corresponding solution $u=u(x,t)$ verifies
$$|x|^{(d/2+2)^{-}}u\in L^{\infty}([0,T];L^2(\mathbb{R}^d)), \hspace{0.2cm} T>0.$$
Although, there does not exists a non-trivial solution $u$ corresponding to data $u_0$ with $\widehat{u_0}(0) \neq 0$ with 
$$|x|^{d/2+2}u\in L^{\infty}([0,T'];L^2(\mathbb{R}^d)), \hspace{0.1cm} \text{ for some } T'>0.$$
\item[(iv)] Theorem \ref{threetimesharp} shows that the decay $r=(d/2+3)^{-}$ is the largest possible in the spatial $L^2$-decay rate. As a result, Theorem \ref{sharpdecay} (ii) is sharp. Apart from this, Theorem \ref{threetimesharp} tells us that there are non-trivial solutions $u=u(x,t)$ such that
$$|x|^{(d/2+3)^{-}} u\in L^{\infty}([0,T];L^2(\mathbb{R}^d)), \hspace{0.5cm} T>0$$
and it guarantees that there does not exist a non-trivial solution such that
$$|x|^{d/2+3} u\in L^{\infty}([0,T];L^2(\mathbb{R}^d)), \hspace{0.4cm} \text{for some } \, T'>0.$$
\end{itemize}
\end{remarks}

One may ask wherever the assumption in Theorem \ref{threetimesharp} can be reduced to two different times $t_1<t_2$. In this respect we have the following consequences.

\begin{theorem}\label{proptwotim} 
Suppose that $d=2,3$, $r_2=3$ and $r_3=4$. Let $u\in C([0,T];\dot{Z}_{r_d,r_d}(\mathbb{R}^d))$ be a solution of the IVP associated to \eqref{HBO-IVP}.
If there exist $t_1,t_2\in [0,T]$, $t_1\neq t_2,$ such that
\begin{equation*}
u(\cdot,t_j) \in Z_{d/2+3,d/2+3}(\mathbb{R}^d), \, j=1,2, 
\end{equation*}
and
\begin{equation*}
\int x_1u(x,t_1)\, dx=0 \hspace{0.5cm} \text{ or } \hspace{0.5cm} \int x_1u(x,t_2)\, dx=0,
\end{equation*}
then
\begin{equation*}
u\equiv 0.
\end{equation*}
\end{theorem}

\begin{theorem}\label{timesharp}
Suppose that $d=2,3$, $r_2=3$ and $r_3=4$. Let $u\in C([0,T];\dot{Z}_{s,r_d}(\mathbb{R}^d))$ with $s\geq d/2+4$ be a nontrivial solution of the IVP associated to \eqref{HBO-IVP} such that 
\begin{equation*}
u_0\in \dot{Z}_{d/2+3,d/2+3}(\mathbb{R}^d) \hspace{0.5cm} \text{ and } \hspace{0.5cm} \int x_1u_0(x)\, dx \neq 0.
\end{equation*}
Let
\begin{equation*}
t^{\ast}:=-\frac{4}{\left\|u_0\right\|_{L^2}^2}\int x_1 u_0(x)\, dx.
\end{equation*}
If $t^{\ast}\in(0,T]$, then
\begin{equation*}
u(t^{\ast})\in \dot{Z}_{d/2+3,d/2+3}(\mathbb{R}^2).
\end{equation*}
\end{theorem}

\begin{remarks}
\begin{itemize}
\item[(i)] Theorem \ref{proptwotim} tells us that the three times condition in Theorem \ref{threetimesharp} can be reduced to two times $t_1 \neq t_2$ provided that
\begin{equation*}
\int x_1u(x,t_1)\, dx=0 \hspace{0.5cm} \text{ or } \hspace{0.5cm} \int x_1u(x,t_2)\, dx=0.
\end{equation*}
\item[(ii)] Theorem \ref{timesharp} asserts that the condition of Theorem \ref{threetimesharp} in general cannot be reduced to two different times. In this sense the result of Theorem \ref{proptwotim} is optimal.
\item[(iii)] In view of Theorem \ref{timesharp},  we notice that the number of times involved in Theorems \ref{sharpdecay} and \ref{threetimesharp} is the same required to establish similar unique continuation properties for the Benjamin-Ono equation, see \cite[Theorem 1 and Theorem 2]{FonPO}. Therefore, our conclusions on the \eqref{HBO-IVP} equation are again regarded as a generalization of their equivalents for the Benjamin-Ono model.
\end{itemize}
\end{remarks}

Next we introduce the main ingredient behind the proof of Proposition \ref{wellposwei} and Theorem \ref{localweigh}. When dealing with energy estimates, motivated by the structure of the dispersion term in the \eqref{HBO-IVP} equation, it is reasonable to try to find a commutator relation involving the Riesz transform, in such a way that when applied to a differential operator it redistributes the derivatives lowering the order of the operator. In this direction, we provide  a new generalization of Calder\'on's first commutator estimate \cite{cald} in the context of the Riesz transform.
\begin{prop}\label{propconmu}
Let $\mathcal{R}_l$ be the usual Riesz transform in the direction $l=1,\dots,d$. For any $1<p<\infty$ and any multi-index $\alpha$ with $|\alpha|\geq 1$, there exists a constant $c$ depending on $\alpha$ and $p$ such that
\begin{equation}\label{conmuest}
\begin{aligned}
\Big\Vert \mathcal{R}_l(a\partial^{\alpha}f)-a \mathcal{R}_l\partial^{\alpha}f-\sum_{1\leq |\beta| < |\alpha|}\frac{1}{\beta!}\partial^{\beta}aD_{R_l}^{\beta}\partial^{\alpha}f \Big\Vert_{L^p} \leq c_{\alpha,p} \sum_{|\beta|=|\alpha|}\left\|\partial^{\beta}a\right\|_{L^{\infty}}\left\|f\right\|_{L^p}.
\end{aligned}
\end{equation}
The operator $D_{R_l}^{\beta}$ is defined via its Fourier transform as
\begin{equation}\label{diffeOperRie}
    \widehat{D_{R_l}^{\beta}g}(\xi)=i^{-|\beta|}\partial^{\beta}_{\xi}\left(\frac{-i\xi_l}{|\xi|}\right)\widehat{g}(\xi).
\end{equation}
\end{prop}
In Proposition \ref{propconmu} the convention for the empty summation (such as $\sum_{1\leq |\beta| <1}$) is defined as zero. Consequently,  when $|\alpha|=1$ we find
\begin{equation*}
    \left\|[\mathcal{R}_l,a]\partial^{\alpha}f\right\|_{L^p} \lesssim \left\|\nabla a\right\|_{L^{\infty}}\left\|f\right\|_{L^p}.
\end{equation*}
where
\begin{equation*}
    [\mathcal{R}_l,a]\partial^{\alpha}f=\mathcal{R}_l(a\partial^{\alpha}f)-a\mathcal{R}_l \partial^{\alpha}f.
\end{equation*}
Estimates of the form \eqref{conmuest} are of interest on their own in harmonic analysis, see \cite{Dli} for similar results and several applications dealing with homogeneous differential operators. The result of Proposition \ref{propconmu} may be of independent interest. Indeed, we believe that it could certainly be used to derive other properties for the \eqref{HBO-IVP} equation.  
\\ 
In the present work, \eqref{conmuest} is essential to transfer derivatives to some weighted functions. Additionally, the operators $D_{R_l}^{\beta}$ defined by \eqref{differform} are useful to symbolize commutator relations between the Riesz transform and polynomials.

We will begin by introducing some preliminary estimates to be used in subsequent sections. In Section 3 we prove Proposition \ref{wellposwei} and Theorems \ref{localweigh}, \ref{sharpdecay}, \ref{threetimesharp}, \ref{proptwotim} and \ref{timesharp} will be deduced in the following Sections 4, 5, 6, 7 and 8 respectively. We  conclude  the  paper  with  an  appendix  where  we  show  the commutator estimate stated in Proposition \ref{propconmu}.


\section{Notation and preliminary estimates}

We will employ the standard multi-index notation, $\alpha=(\alpha_1,\dots,\alpha_d) \in \mathbb{N}^d$, $\partial^{\alpha}=\partial^{\alpha_1}_{x_1}\cdots \partial_{x_d}^{\alpha_d}$, $|\alpha|=\sum_{j=1}^d \alpha_j$, $\alpha!=\alpha_1 ! \cdots \alpha_d !$ and $\alpha \leq \beta$ if $\alpha_j \leq \beta_j$ for all $j=1,\dots,d$.  As usual $e_k\in \mathbb{R}^d$ will denote the standard canonical vector in the $k$ direction.

For any two positive quantities $a$ and $b$, $a\lesssim b$ means that there exists $C>0$ independent of $a$ and $b$ (and in our computations of any parameter involving approximations) such that $a\leq Cb$. Similarly, we define $a\gtrsim b$, and $a\sim b$ states that $a\lesssim b$ and $b\gtrsim a$. $[A,B]$ denotes the commutator between the operators $A$ and $B$, that is
$$[A,B]=AB-BA.$$
Given $p\in [1,\infty]$, we define the Lebesgue spaces $L^p(\mathbb{R}^d)$ by its norm as $\left\|f\right\|_{L^p}=\left(\int_{\mathbb{R}^d} |f(x)|^p\,dx\right)^{1/p},$ with the usual modification when $p=\infty$. We denote by $C_c^{\infty}(\mathbb{R}^d)$ the spaces of smooth functions of compact support and $\mathcal{S}(\mathbb{R}^d)$ the space of Schwartz functions. The  Fourier transform is defined as $$\widehat{f}(\xi)=\mathcal{F}f(\xi)=\int_{\mathbb{R}^d} e^{-i x\cdot \xi} f(x) \, dx.$$
As usual, the operator $J^s=(1-\Delta)^{s/2}$ is defined by the Fourier multiplier with symbol $\langle \xi \rangle^s=(1+|\xi|^2)^{s/2}$, $s\in \mathbb{R}$. The norm in the Sobolev space $H^s(\mathbb{R})$ is given by
$$\left\|f\right\|_{H^s}=\left\|J^{s}f \right\|_{L^2}=\left\|\langle \xi \rangle^{s}\widehat{f}(\xi) \right\|_{L^2},$$
where $\langle \cdot \rangle=(1+|\cdot|^2)^{1/2}$. Similarly, the Homogeneous Sobolev space $\dot{H}^s(\mathbb{R}^d)$ is determined by its norm, $\left\|f\right\|_{\dot{H}^s}=\left\||\xi|^s\widehat{f}(\xi) \right\|_{L^2}$. 

A radial function $\phi\in C^{\infty}_c(\mathbb{R}^d)$, with $\phi(x)= 1$ when $|x|\leq 1$ and $\phi(x)=0$ if $|x| \geq 2$ will appear several times in our arguments.

Next, we introduce some notation that will be convenient to prove Theorems \ref{sharpdecay} and Theorems \ref{threetimesharp}. Given $k=1,\dots, d$ fixed, we define the operators $F_j^{k}$'s as being:
\begin{equation} \label{eqsharp1}
F_j^{k}(t,\xi,f)=\partial_{\xi_k}^{j}\big(e^{it\xi_1|\xi|}f(\xi)\big)
\end{equation}
for $j=1,2,3,4$. More precisely, 
\begin{equation}
\begin{aligned}\label{eqsharp1.1}
F_1^k(t,\xi,f)&=\partial_{\xi_k}\big(it\xi_1|\xi|\big)e^{it\xi_1|\xi|}f(\xi)+e^{it\xi_1|\xi|}\partial_{\xi_k}f(\xi), \\
F_2^{k}(t,\xi,f)&=\partial_{\xi_k}^2\big(it\xi_1|\xi|\big)e^{it\xi_1|\xi|}f(\xi)+\partial_{\xi_k}\big(it\xi_1|\xi|\big)F_1^{k}(t,\xi,f)+F_1^{k}(t,\xi,\partial_{\xi_k}f),\\
F_3^{k}(t,\xi,f)&=\partial_{\xi_k}^3\big(it\xi_1|\xi|\big)e^{it\xi_1|\xi|}f(\xi)+2\partial_{\xi_k}^2\big(it\xi_1|\xi|\big)F_1^{k}(t,\xi,f)+\partial_{\xi_k}(it\xi_1|\xi|)F_2^{k}(t,\xi,f)+F_2^{k}(t,\xi,\partial_{\xi_k}f), \\
F_4^{k}(t,\xi,f)&=\partial_{\xi_k}^4\big(it\xi_1|\xi|\big)e^{it\xi_1|\xi|}f(\xi)+3\partial_{\xi_k}^3\big(it\xi_1|\xi|\big)F_1^{k}(t,\xi,f)+3\partial_{\xi_k}^2(it\xi_1|\xi|)F_2^{k}(t,\xi,f)\\
&\hspace{0.3cm}+\partial_{\xi_k}(it\xi_1|\xi|)F_3^{k}(t,\xi,f) +F_3^{k}(t,\xi,\partial_{\xi_k}f).
\end{aligned}
\end{equation}
Additionally, the operators $\widetilde{F}_j^k$, $j=1,2,3,4$ are defined according to \eqref{eqsharp1.1} by the relations
\begin{equation}\label{primeF}
\begin{aligned}
&\widetilde{F}_j^k(t,\xi,f)=e^{-it\xi_1|\xi|}F_j^k(t,\xi,f).
\end{aligned}
\end{equation}
The following identities will be frequently considered in our arguments:
\begin{equation} \label{eqsharp2.1}
\begin{aligned}
\partial_{\xi_k}\big(\xi_1|\xi|\big)&=\delta_{1,k}|\xi|+\frac{ \xi_1 \xi_k}{|\xi|}, \hspace{1cm} \partial_{\xi_k}^2\big(\xi_1|\xi|\big)=2\delta_{1,k}\frac{\xi_k}{|\xi|}+\frac{\xi_1}{|\xi|}-\frac{\xi_1\xi_k^2}{|\xi|^3}, \\
\partial_{\xi_k}^3\big(\xi_1|\xi|\big)&=3\delta_{1,k}\frac{1}{|\xi|}-3\delta_{1,k}\frac{\xi_k^2}{|\xi|^3}-3\frac{\xi_1\xi_k}{|\xi|^3}+3\frac{\xi_1\xi_k^3}{|\xi|^5}, \\
\partial_{\xi_k}^4\big(\xi_1|\xi|\big)&=-12\delta_{1,k}\frac{\xi_k}{|\xi|^3}+12\delta_{1,k} \frac{\xi_k^3}{|\xi|^5}-3\frac{\xi_1}{|\xi|^3}+18\frac{\xi_1\xi_k^2}{|\xi|^5}-15\frac{\xi_1\xi_k^4}{|\xi|^7}. 
\end{aligned}
\end{equation}
Let $N\in \mathbb{Z}^{+}$. We introduce the truncated weights $\tilde{w}_N : \mathbb{R} \rightarrow \mathbb{R}$ satisfying 
 \begin{equation}
 \tilde{w}_{N}(x)=\left\{\begin{aligned} 
 &\langle x \rangle, \text{ if } |x|\leq N, \\
 &2N, \text{ if } |x|\geq 3N
 \end{aligned}\right.
 \end{equation}
in such a way that $\tilde{w}_N(x)$ is smooth and non-decreasing in $|x|$ with $\tilde{w}'_N(x) \leq 1$ for all $x>0$  and there exists a constant $c$ independent of $N$ from which $\tilde{w}''_N(x) \leq c\partial_x^2\langle x \rangle$. We then define the $d$-dimensional weights by the relation
\begin{equation}\label{intro2}
w_N(x)=\tilde{w}_N(|x|), \text{ where } |x|=\sqrt{x_1^2+\dots+x_d^2}.
\end{equation}
We require some point-wise bounds for the product between powers of the weight $w_N$ and a polynomial with variables in $\mathbb{R}^d$. More specifically, for a given $\theta \in (0,2]$ and multi-indexes $\alpha$ and $\beta$ with $1\leq |\alpha|\leq 2$, by the definition of $w_N$ one finds 
\begin{equation}\label{weighNprope}
    |\partial^{\alpha} w_{N}^{\theta}(x) x^{\beta}| \lesssim w_N^{\theta+|\beta|-|\alpha|}(x),
\end{equation}
where the implicit constant is independent of $N$ and $\theta$. In particular, when $\theta \leq |\alpha|$ and $\beta=0$, $ |\partial^{\alpha} w_{N}^{\theta} | \lesssim 1$.

Next we discuss some properties of the operators $D^{\beta}_{R_l}$ defined by \eqref{diffeOperRie}. The following lemma is useful to estimate the $L^2$-norm of these operators.
 
\begin{lemma}\label{lemmaRieszdeco} Let  $\alpha$ and $\beta$ be multi-indexes and  $f\in \dot{H}^{|\alpha|-|\beta|}(\mathbb{R}^d)$. Then there exist constants $c_{\sigma} \in \mathbb{R}$ such that
\begin{equation}
D_{R_1}^{\beta}(\partial^{\alpha}f)=\sum_{\sigma} c_{\sigma} \mathcal{R}_{\sigma}(|\nabla|^{|\alpha|-|\beta|}f),
\end{equation}
where the sum runs over all index $\sigma=(\sigma_{1},\dots,\sigma_{|\alpha|+|\beta|+1})$ with integer components such that $1\leq \sigma_j \leq d$, $j=1,\dots,|\alpha|+|\beta|+1$ and we denote by
$$\mathcal{R}_{\sigma}=\mathcal{R}_{\sigma_1}\cdots \mathcal{R}_{\sigma_{|\alpha|+|\beta|+1}}.$$
\end{lemma}
For instance, when  $\alpha=0$ and $|\beta|=1$, say $\beta=e_k$, one has 
\begin{equation}\label{identity1.1}
D_{R_1}^{e_k} f=-\delta_{1,k}|\nabla|^{-1}f-\mathcal{R}_1 \mathcal{R}_k \big(|\nabla|^{-1}f\big),
\end{equation}
and so, letting now $\alpha=e_j$,
\begin{equation}\label{identity1}
D_{R_1}^{e_k}\partial_{x_j} f=\delta_{1,k}\mathcal{R}_jf +\mathcal{R}_1 \mathcal{R}_k \mathcal{R}_jf .
\end{equation}

\begin{proof}[Proof of Lemma \ref{lemmaRieszdeco} ]
An inductive argument yields the following identity
\begin{equation}
\partial^{\beta}\left(\frac{\xi_1}{|\xi|}\right)=\frac{P_{\beta}(\xi)}{|\xi|^{2|\beta|+1}}, \hspace{0.5cm} \xi \neq 0,
\end{equation}
where $P_{\beta}(\xi)$ is a homogeneous polynomial with real coefficients of order $|\beta|+1$. Accordingly, we deduce the following point-wise estimate
\begin{equation}
\begin{aligned}
\mathcal{F}D_{R_1}^{\beta}(\partial^{\alpha}f)(\xi)=\frac{-1}{i^{|\beta|-|\alpha|-1}} \partial^{\beta}\left(\frac{\xi}{|\xi|}\right)\xi^{\alpha} \widehat{f}(\xi) =(-1)^{|\alpha|}\left(\frac{P_{\beta}(-i\xi)(-i\xi)^{\alpha}}{|\xi|^{|\alpha|+|\beta|+1}}\right) |\xi|^{|\alpha|-|\beta|}\widehat{f}(\xi).
\end{aligned}
\end{equation}
The proof is now a consequence of the fact that the inverse Fourier transform of $P_{\beta}(-i\xi)(-i\xi)^{\alpha}/|\xi|^{|\alpha|+|\beta|+1}$ can be written as a linear combination of the operators $\mathcal{R}_{\sigma}$, where $\sigma=(\sigma_{1},\dots,\sigma_{|\alpha|+|\beta|+1})$ with $1\leq \sigma_j \leq d$. 
\end{proof} As already mentioned, the operators $D^{\beta}_{R_l}$ are useful to express commutator relations between Riesz transforms and polynomials. More explicitly, for a given a multi-index $|\gamma|\geq 1$, we shall use the following point-wise estimate 
\begin{equation}\label{identity2}
\begin{aligned}
\left[\mathcal{R}_1,x^{\gamma}\right]f=\sum_{0<\beta\leq \gamma} \binom{\gamma}{\beta} (-1)^{|\beta|+1} D_{R_1}^{\beta}(x^{\gamma-\beta}f),
\end{aligned}
\end{equation}
valid for $f$ regular enough with appropriated decay and satisfying for instance  $$\int x^{\beta} f(x)\, dx=0, \hspace{0.2cm} \text{ for each }  |\beta|<|\gamma|.$$
In particular, taking $\gamma=e_k$, $k=1,\dots,d$ and recalling \eqref{identity1}, we obtain 
\begin{equation}\label{identity3}
[\mathcal{R}_1,x_k]\partial_{x_j}f=D_{R_1}^{e_k}\partial_{x_j} f=\delta_{1,k}\mathcal{R}_jf +\mathcal{R}_1 \mathcal{R}_k \mathcal{R}_jf.
\end{equation}

Now we state some preliminary results. The definition of the $A_p(\mathbb{R}^d)$ condition is essential in our analysis. 
\begin{defin}\label{def1}
 A non-negative function $w\in L^{1}_{loc}(\mathbb{R}^d)$ satisfies the $A_p(\mathbb{R}^d)$ inequality with $1<p<\infty$ if there exists a constant $C$ independent of the cube $Q$, such that
\begin{equation}\label{Apcon}
\sup_{Q} \left(\frac{1}{|Q|} \int_{Q} w(x)\, dx\right) \left(\frac{1}{|Q|} \int w(x)^{1-p'} dx\right)^{p-1}=Q_p(w) \leq C
\end{equation}
where the supremum runs over cubes in $\mathbb{R}^d$ and $1/p+1/p'=1$. 
\end{defin}

Since we are concerned with weighted energy estimates, we require some continuity properties of Riesz transforms in weighted spaces. 

\begin{theorem}(\cite{sharpRiesz}) \label{sharpRiezW}
For $1<p<\infty$ and $l=1,\dots,d$ there exists a constant $c$ depending on $p$ and $d$ so that for all weights $w\in A_p(\mathbb{R}^d)$ the Riesz transforms as operators in weighted space $\mathcal{R}_l: L^p(w(x)\, dx)\mapsto L^p(w(x)\, dx)$ satisfies
\begin{equation}
\left(\int_{\mathbb{R}^d} |\mathcal{R}_lf(x)|^{p} w(x)dx\right)^{1/p} \leq c \, Q_p(w)^{r} \left(\int_{\mathbb{R}^d} |f(x)|^{p} w(x)dx\right)^{1/p}
\end{equation}
where $Q_p(w)$ is defined by \eqref{Apcon}, $r = \max\{1,p'/p\}$. Moreover, this result is sharp.
\end{theorem}

One can verify that for fixed $\theta \in (-d,d)$, $w_N^{\theta}(x)$, $N\in \mathbb{Z}^{+}$, satisfies the $A_2(\mathbb{R}^d)$ inequality with a constant $Q_2(w_{N}^{\theta})$ independent of $N$. From this fact and Theorem \ref{sharpRiezW}, we infer:

\begin{prop}\label{propapcond}
For any $\theta \in (-d,d)$ and any $N\in \mathbb{Z}^{+}$, $w_N^{\theta}(x)$  satisfies the $A_2(\mathbb{R}^d)$ inequality \eqref{Apcon}.
Moreover, the $Riesz$ transform is bounded in $L^2(w_N^{\theta} (x)\, dx)$ with a constant depending on $\theta $ but independent of $N\in \mathbb{Z}^{+}$.
\end{prop}
Proposition \ref{propapcond} is helpful to show that our computations in the proof of Theorem \ref{localweigh} are independent of the parameter $N$ defining the weight $w_N$. We also require the following commutator relation.
\begin{prop}\label{prelimprop0}
Let $\theta \in (0,1)$ and $1\leq p_1,p_2 <\infty$ such that $\frac{3}{2}=\frac{1}{p_1}+\frac{1}{p_2}$. Then
\begin{equation}\label{prelimneq5}
\left\|[D^{\theta},g]f\right\|_{L^2} \lesssim  \left\||\cdot|^{\theta}\widehat{g}\right\|_{L^{p_1}}\left\|\widehat{f}\right\|_{L^{p_2}}.
\end{equation} 
\end{prop}
The following characterization of the spaces $L^p_s(\mathbb{R}^d)=J^{-s}L^p(\mathbb{R}^d)$ is fundamental in our considerations.
\begin{theorem}( \cite{SteinThe})\label{TheoSteDer}
Let $b\in (0,1)$ and $2d/(d+2b)<p<\infty$. Then $f\in L_b^p(\mathbb{R}^d)$ if and only if

\begin{itemize}
\item[(i)]  $f\in L^p(\mathbb{R}^d)$, 
\item[(ii)]$\mathcal{D}^bf(x)=\left(\int_{\mathbb{R}^d}\frac{|f(x)-f(y)|^2}{|x-y|^{d+2b}}\, dy\right)^{1/2}\in L^{p}(\mathbb{R}^d),$
\end{itemize}
with 
\begin{equation*}
\left\|J^b f\right\|_{L^p}=\left\|(1-\Delta)^{b/2} f\right\|_{L^p} \sim \left\|f\right\|_{L^p}+\left\|\mathcal{D}^b f\right\|_{L^p} \sim \left\|f\right\|_{L^p}+\left\|D^b f\right\|_{L^p}.
\end{equation*} 
\end{theorem}
Above we have introduced the notation $D^s=(-\Delta)^{s/2}$.
\\ \\
Next, we proceed to shows several consequences of Theorem \ref{TheoSteDer}.
When $p=2$ and $b\in (0,1)$ one can deduce that
\begin{equation} \label{prelimneq0.1} 
\left\|\mathcal{D}^b(fg)\right\|_{L^2} \lesssim \left\|f\mathcal{D}^b g\right\|_{L^2}+\left\|g\mathcal{D}^bf \right\|_{L^2},
\end{equation}
and it holds
\begin{equation} \label{prelimneq0.11}
\left\|\mathcal{D}^{b} h\right\|_{L^{\infty}} \lesssim \big(\left\|h\right\|_{L^{\infty}}+\left\|\nabla h\right\|_{L^{\infty}} \big).
\end{equation} 
In addition, we require the following result which is proved in much the same way as in \cite{NahPonc}.
\begin{prop}\label{prelimneq0.2}
Let $b\in (0,1)$. For any $t>0$ 
\begin{equation}
\mathcal{D}^b (e^{ix_1|x|})\lesssim (|t|^{b/2}+|t|^b|x|^b), \hspace{0.5cm} x\in \mathbb{R}^d.
\end{equation}
\end{prop}
The estimates \eqref{prelimneq0.1} and \eqref{prelimneq0.11} yield:
\begin{prop}\label{prelimprop1} 
Let $h\in L^{\infty}(\mathbb{R}^{d})$ with $\nabla h\in L^{\infty}(\mathbb{R}^{d})$. Then
\begin{equation*}
\left\|hf\right\|_{H^{1/2}} \lesssim \big(\left\|h\right\|_{L^{\infty}}+\left\|\nabla h\right\|_{L^{\infty}}\big)\left\|f\right\|_{H^{1/2}}.
\end{equation*} 
\end{prop}
As a further consequence of Theorem \ref{TheoSteDer} one has the following interpolation inequality.
\begin{lemma}
Let $a,b>0$. Assume that $J^af=(1-\Delta)^{a/2}f \in L^{2}(\mathbb{R}^d)$ and $\langle x \rangle^b f=(1+|x|^2)^{b/2}f\in L^{2}(\mathbb{R}^d)$. Then for any $\theta \in (0,1)$, 
\begin{equation}\label{prelimneq3}
\left\|J^{\theta a}(\langle x \rangle^{(1-\theta)b}f)\right\|_{L^2}\lesssim \left\|\langle x \rangle^{b} f\right\|_{L^2}^{1-\theta}\left\|J^a f\right\|_{L^2}^{\theta}.
\end{equation}
Moreover, the inequality \eqref{prelimneq3} is still valid with $w_N(x)$ instead of $\langle x \rangle$ with a constant $c$ independent of $N$.
\end{lemma}
\begin{proof}
The proof is similar to that in \cite[Lemma 1]{FonPO}.
\end{proof}
 Theorem \ref{TheoSteDer} also provides the following point-wise estimate:
\begin{lemma}\label{lemmapreli1}
Let $\theta \in (0,1)$, $l=0,1$ fixed and $P(x)$ be a homogeneous polynomial of degree $k \geq 0$ in $\mathbb{R}^d$. In addition, let $g\in L^{\infty}(\mathbb{R}^d)$ such that $ |\cdot|^{-l}g, \nabla g \in L^{\infty}(\mathbb{R}^d)$. Then,
\begin{equation}\label{prelimneq4}
\mathcal{D}^{\theta}\big(|\cdot|^{-k-l}P(\cdot) g\big)(\xi) \lesssim_{k} \big(\left\||\cdot|^{-l}g\right\|_{L^{\infty}}+\left\|\nabla g\right\|_{L^{\infty}}\big)\big(1+|\xi|^{-\theta}\big), 
\end{equation}
for all $\xi \neq 0$.
\end{lemma}
\begin{proof}
Let $l=0,1$, we write 
\begin{equation}
\begin{aligned}
\big(\mathcal{D}^{\theta}(|\cdot|^{-k-l}P(\cdot) g)\big)^2(\xi) &=\int \frac{\left||\xi|^{-k-l}P(\xi)g(\xi)-|\xi-\eta|^{-k-l}P(\xi-\eta) g(\xi-\eta)\right|^2}{|\eta|^{d+2\theta}} \, d\eta \\
& = \int_{|\eta|\leq \min\left\{|\xi|/2,1\right\}} (\cdots)\, d\eta+\int_{|\eta|>\min\left\{|\xi|/2,1\right\}} (\cdots)\, d\eta \\
&=: \mathcal{I}+\mathcal{II}.
\end{aligned}
\end{equation}
Given that $P(\xi)$ is a homogeneous polynomial of degree $k$, it is deduced
\begin{equation}\label{prelimneq4.1}
\begin{aligned}
\mathcal{II} &\lesssim \left\||\cdot|^{-l}g\right\|_{L^{\infty}}^2 \big(\int_{|\eta|>|\xi|/2} \frac{1}{|\eta|^{d+2\theta}}  \, d\eta+\int_{\min\left\{|\xi|/2,1\right\}<|\eta|\leq|\xi|/2} \frac{1}{|\eta|^{d+2\theta}}  \, d\eta \big)\\
& \lesssim  \left\||\cdot|^{-l}g\right\|_{L^{\infty}}^2\big(1+|\xi|^{-2\theta}\big).
\end{aligned}
\end{equation}
On the other hand, when $|\eta|\leq \min\left\{|\xi|/2,1\right\}$,  $|\eta-\xi|\sim |\xi|$ and so
\begin{equation}
\begin{aligned}
&\left||\xi|^{-k-l}P(\xi)g(\xi)-|\xi-\eta|^{-k-l}P(\xi-\eta) g(\xi-\eta)\right|  \\
& \hspace{0.9cm}\leq \left||\xi|^{-k-l}P(\xi)\big(g(\xi)-g(\xi-\eta)\big)\right|+\left||\xi|^{-k-l}P(\xi)-|\xi-\eta|^{-k-l}P(\xi-\eta) \right| |g(\xi-\eta)| \\
& \hspace{0.9cm} \lesssim  \left\|\nabla g\right\|_{L^{\infty}}|\xi|^{-l}|\eta|+\sum_{j=0}^{k+l-1}\frac{|\eta||\xi|^{k+l-1-j}|\xi-\eta|^{j}|\xi|^k}{|\xi|^{k+l}|\xi-\eta|^{k+l}}|g(\xi-\eta)| +\sum_{j=0}^{k-1}\frac{|\eta||\xi|^{k-1-j}|\xi-\eta|^{j}}{|\xi-\eta|^{k+l}}|g(\xi-\eta)|  \\
& \hspace{0.9cm} \lesssim \big(\frac{\left\|\nabla g\right\|_{L^{\infty}}}{|\xi|^l}+\frac{\left\||\cdot|^{-l} g\right\|_{L^{\infty}}}{|\xi|}\big)|\eta|.
\end{aligned}
\end{equation}
Hence we get
\begin{equation}\label{prelimneq4.3}
\begin{aligned}
\mathcal{I} &\lesssim \big(\frac{\left\|\nabla g\right\|_{L^{\infty}}^2}{|\xi|^{2l}}+\frac{\left\||\cdot|^{-l} g\right\|_{L^{\infty}}^2}{|\xi|^2}\big) \int_{|\eta|\leq \min\left\{|\xi|/2,1\right\} }  \frac{1}{|\eta|^{d-2+{2\theta}}} \, d\eta \lesssim \big(\left\||\cdot|^{-l}g\right\|_{L^{\infty}}^2+\left\|\nabla g\right\|_{L^{\infty}}^2 \big) \big(1+|\xi|^{-2\theta}\big).
\end{aligned}
\end{equation}
Gathering \eqref{prelimneq4.1} and \eqref{prelimneq4.3} we deduce \eqref{prelimneq4}.
\end{proof}
We are now in position to show the following result, which will be useful to deduce Theorems \ref{sharpdecay} and \ref{threetimesharp} in the three-dimensional setting.
\begin{prop}\label{lemmasharp3} 
Let $g \in C_c^{\infty}(\mathbb{R}^3)$ and $P(x)$ a homogeneous polynomial of degree $k \geq 1$ in $\mathbb{R}^3$. Then
\begin{equation}\label{eqsharp16.0}
\left\|\frac{P(\cdot)}{|\cdot|^k} f g\right\|_{H^{1/2}} \lesssim_{k,g} \left\| f\right\|_{H^{(1/2)^{+}}}.
\end{equation}
Furthermore, if $m$ is an integer with $0\leq m <k$, 
\begin{equation}\label{prelimneq4.4}
\left\|\frac{P(\cdot)}{|\cdot|^m} f g\right\|_{H^{1/2}} \lesssim_{k,m,g} \left\| f\right\|_{H^{1/2}}.
\end{equation}
\end{prop}\begin{proof}
Let us first prove \eqref{eqsharp16.0}. Consider a function $\tilde{g}\in C_c^{\infty}(\mathbb{R}^d)$ such that $\tilde{g}g=g$, then from \eqref{prelimneq4} with $l=0$, we have
\begin{equation}\label{eqsharp16.3}
\begin{aligned}
\left\|\frac{P(\cdot)}{|\cdot|^k} f g\right\|_{H^{1/2}}  &\lesssim \left\| f g\right\|_{L^2} +\left\|\mathcal{D}^{1/2}\big(|\cdot|^{-k}P(\cdot)f g\big)\right\|_{L^2} \\
 &\lesssim \left\| f g\right\|_{L^2} +\left\|\mathcal{D}^{1/2}\big(|\cdot|^{-k}P(\cdot)\tilde{g}\big)fg\right\|_{L^2} +\left\||\cdot|^{-k}P(\cdot)\tilde{g}\mathcal{D}^{1/2}(fg)\right\|_{L^2} \\
& \lesssim \left\| fg\right\|_{H^{1/2}}+ \left\||\cdot|^{-1/2}fg\right\|_{L^2}.
\end{aligned}
\end{equation}
Thus, the commutator relation \eqref{prelimneq5} with $p_1=1$ and $p_2=2$ yields
\begin{equation}\label{prelimneq5.1}
\begin{aligned}
\left\| fg\right\|_{H^{1/2}} &\lesssim \left\|fg\right\|_{L^2}+\left\|[D^{1/2},g]f\right\|_{L^2}+\left\|gD^{1/2}f\right\|_{L^2} \lesssim_g \left\|f\right\|_{H^{1/2}}. 
\end{aligned}
\end{equation}
On the other hand, taking $0<\epsilon< 1$, H\"older's inequality and Sobolev's embedding imply
\begin{equation}
\begin{aligned}
\left\||\cdot|^{-1/2}fg\right\|_{L^2} &\lesssim \left\|f\right\|_{L^{3/(1-\epsilon)}}\left\||\cdot|^{-1/2}g\right\|_{L^{6/(1+2\epsilon)}} \lesssim_g \left\|D^{1/2+\epsilon}f\right\|_{L^{2}} \leq \left\|f\right\|_{H^{1/2+\epsilon}},
\end{aligned}
\end{equation}
where we have used that $|\cdot|^{-1}\in L^{6/(1+2\epsilon)}_{loc}(\mathbb{R}^3)$. Thus incorporating the above estimates in \eqref{eqsharp16.3}, we get \eqref{eqsharp16.0}. To deduce \eqref{prelimneq4.4}, since $P(x)$ has degree $k$, there exist finite multi-indexes $\beta_1,\dots,\beta_l$ of order $k-m$ and homogeneous polynomials $P_{\beta_1}(x), \dots , P_{\beta_l}(x)$ of order $m$ such that
\begin{equation}
\frac{P(x)}{|x|^m}=\sum_{j=1}^l \frac{P_{\beta_j}(x)}{|x|^m}x^{\beta_j}.
\end{equation}
Therefore, since $k-m\geq 1$ and $x^{\beta_j}g$ is a smooth function of compact support for each $j$, arguing as in \eqref{eqsharp16.3} and \eqref{prelimneq5.1}, we obtain
\begin{equation*}
\begin{aligned}
\left\|\frac{P(\cdot)}{|\cdot|^m} f g\right\|_{H^{1/2}}  \lesssim &\sum_{j=1}^l \left\||x|^{-m}P_{\beta_j}(x) x^{\beta_j}fg\right\|_{H^{1/2}} \\
\lesssim &\sum_{j=1}^l\left\| x^{\beta_j}f g\right\|_{L^2}  +\left\||\cdot|^{-m}P_{\beta_j}(\cdot)\tilde{g}\mathcal{D}^{1/2}(x^{\beta_j}fg)\right\|_{L^2} +\left\|\mathcal{D}^{1/2}\big(|x|^{-m}P_{\beta_j}(x)\tilde{g}\big)x^{\beta_j}fg\right\|_{L^2}\\
 \lesssim & \sum_{j=1}^l \left\| x^{\beta_j}fg\right\|_{H^{1/2}}+ \left\||x|^{-1/2} x^{\beta_j}fg\right\|_{L^2} \lesssim  \left\|f\right\|_{H^{1/2}}.
\end{aligned}
\end{equation*}
The proof of the proposition is now completed.
\end{proof}


\subsection{Review local-well posedness in Sobolev spaces}\label{prelimsub}

The results concerning local well-posedness for the \eqref{HBO-IVP} equation in classical Sobolev spaces $H^s(\mathbb{R}^d)$ are fundamental in our arguments to extend these conclusions to weighted spaces. In this regard, we recall the following results derived in \cite{linaO}.  

\begin{theorem}\label{imprwellpos}
Let $s>s_d$ where $s_2=5/3$ and $s_d=d/2+1
/2$ when $d\ge 3$. Then, for any $u_0\in H^s(\mathbb{R}^d)$, there exist a time $T=T(\left\|u_0\right\|_{H^s})$ and a unique solution~$u$ to~\eqref{HBO-IVP} that belongs to
\begin{equation}\label{spaceLWP}
    C\big([0,T];H^s(\mathbb{R}^d)\big)\cap L^1\big([0,T];W^{1,\infty}(\mathbb{R}^d)\big).
\end{equation}
Moreover, the flow map $u_0 \mapsto u(t)$ is continuous from $H^s(\mathbb{R}^d)$ to $H^s(\mathbb{R}^d)$.
\end{theorem}
Part of the proof of Theorem \ref{imprwellpos} in \cite[Proposition 5.10 and Lemma 5.9]{linaO} guarantees existence of solutions for the \eqref{HBO-IVP} equation as the strong limit of smooth solutions in the class \eqref{spaceLWP}, whose initial data are mollified versions of $u_0$ in the sense of the Bona-Smith argument \cite{BS}. More precisely, for a given solution $u\in  C\big([0,T];H^s(\mathbb{R}^d)\big)\cap L^1\big([0,T];W^{1,\infty}(\mathbb{R}^d)\big)$ provided by Theorem \ref{imprwellpos}, there exists a sequence of smooth solutions of \eqref{HBO-IVP}, $u_n\in C([0,T];H^{\infty}(\mathbb{R}^{d}))$ where $H^{\infty}(\mathbb{R}^d)=\bigcap_{m\ge 0} H^{m}(\mathbb{R}^d)$ such that 
\begin{equation}\label{weigheq1}
    \sup_{t\in[0,T]}\left\|u_n(t)\right\|_{H^{s}} \leq 2 \left\|u_0\right\|_{H^{s}},
    \end{equation}
and
\begin{equation}\label{aproxres}
    u_n \to u \hspace{0.2cm} \text{ in the sense of } \hspace{0.2cm} C\big([0,T];H^s(\mathbb{R}^d))\cap L^1([0,T];W^{1,\infty}(\mathbb{R}^d)\big).
\end{equation}
In this manner, \eqref{aproxres} will be useful to perform rigorously weighted energy estimates at the $H^s(\mathbb{R}^d)$-level stated in Theorem \ref{imprwellpos}, and then take the limit $n\to \infty$ to deduce Proposition \ref{wellposwei} and Theorem \ref{localweigh}.

\section{Proof of Proposition \ref{wellposwei}}

In this section we establish local well-posedness in the space $H^{s}(\mathbb{R}^d)\cap L^{2}(\omega^2\, dx)$.  We require the following result.
\begin{lemma}\label{lemmwei}
 Let $\omega$ be a smooth weight with all its first and second derivatives bounded. Define
\begin{equation}
\omega_{\lambda}(x)=\omega(x)e^{-\lambda|x|^2}, \, x\in \mathbb{R}^d, \, \lambda \in (0,1).
\end{equation}
Then, there exists a constant $c>0$ independent of $\lambda$, such that
\begin{equation*}
\left\|\partial^{\alpha}\omega_{\lambda}\right\|_{\infty}\leq c,
\end{equation*}
where $\alpha$ is a multi-index of order $1\leq |\alpha|\leq 2$.
\end{lemma}
\begin{proof}
The proof is similar to that in \cite[Lemma 4.1]{CUNHA}.
\end{proof}
Now we proceed to prove Proposition \ref{wellposwei}. Given  $u_0\in H^{s}(\mathbb{R}^d)\cap L^{2}(\omega^2 \, dx)$, from Theorem \ref{imprwellpos}, there exist $T=T(\left\|u_0\right\|_{H^s})>0$, $u\in C([0,T];H^s(\mathbb{R}^d))$ solution of \eqref{HBO-IVP} with initial datum $u_0$ and a smooth sequence of solutions $u_n\in C([0,T];H^{\infty}(\mathbb{R}^{d}))$ with $u_n(0)\in  L^{2}(\omega^2 \, dx)$, satisfying \eqref{weigheq1} and \eqref{aproxres}. We shall prove the persistence property $u \in C([0,T];L^{2}(\omega^2 \, dx))$.
\\ \\
 We first perform energy estimates for the regularized solutions $u_n \in C([0,T];H^{\infty}(\mathbb{R}^d))$, $n\geq 1$. Let $\omega_{\lambda}$ defined as in Lemma \ref{lemmwei}. Since $\omega_{\lambda}$ is bounded and $u_n$ is smooth, we can multiply the equation \eqref{HBO-IVP} associated to $u_n$ by $\omega_{\lambda}^{2} u_n$ and then integrate on the spatial variable to deduce
\begin{equation}
\begin{aligned}
     \frac{d}{dt} \int (\omega_{\lambda}u_n)^2\, dx-\int \omega_{\lambda} \mathcal{R}_1\Delta u_n \omega_{\lambda} u_n \, dx+\int \omega_{\lambda} u_n \partial_{x_1}u_{n} \omega_{\lambda} u_n\, dx=0.
\end{aligned}
\end{equation}
The nonlinear term can be bounded as follows
\begin{equation*}
    \left|\int \omega_{\lambda} u_n \partial_{x_1}u_{n}  \omega_{\lambda} u_n\, dx\right| \leq \left\|\nabla u_n\right\|_{L^{\infty}}\left\|\omega_{\lambda}u_n\right\|^{2}_{L^2}.
\end{equation*}
To control the factor involving the dispersion, we write
\begin{equation}\label{weigheq2,1} 
    -\omega_{\lambda}\mathcal{R}_1\Delta u_n =[\mathcal{R}_1,\omega_{\lambda}]\Delta u_n-\mathcal{R}_1(\omega_{\lambda}\Delta u_n)=[\mathcal{R}_1,\omega_{\lambda}]\Delta u_n-\mathcal{R}_1([\omega_{\lambda},\Delta]u_n)-\mathcal{R}_1\Delta(\omega_{\lambda}u_n).
\end{equation}
Since the Riesz transform $\mathcal{R}_1$ defines an anti-symmetric operator it is seen that
\begin{equation*}
-\int \mathcal{R}_1\Delta(\omega_{\lambda}u_n) \omega_{\lambda} u_n \, dx=0.
 \end{equation*}
Thus, it remains to control the first two terms on the r.h.s of \eqref{weigheq2,1}. In light of the commutator estimate \eqref{conmuest}, Lemma \ref{lemmaRieszdeco} and \eqref{weigheq1}, we have  
\begin{equation*}
    \begin{aligned}
         \left\|[\mathcal{R}_1,\omega_{\lambda}]\Delta u_n\right\|_{L^2} \lesssim \sum_{j=1}^d \left\|[\mathcal{R}_1,\omega_{\lambda}]\partial_{x_j}^2 u_n\right\|_{L^2} & \lesssim \sum_{|\beta|=2}\left\|\partial^{\beta}\omega_{\lambda}\right\|_{L^{\infty}}\left\|u_n\right\|_{L^2}+\sum_{j=1}^d\sum_{|\beta|=1}\left\|\partial^{\beta}\omega_{\lambda}D_{R_1}^{\beta}\partial_{x_j}^2 u_n\right\|_{L^2}\\
         & \lesssim \left\|u_n\right\|_{L^2}+\sum_{j=1}^d\sum_{|\beta|=1}\left\|\partial^{\beta}\omega_{\lambda}\right\|_{L^{\infty}}\left\|D_{R_1}^{\beta}\partial_{x_j}^2 u_n\right\|_{L^2} \\
         & \lesssim \left\|u_n\right\|_{H^s} \lesssim \left\|u_0\right\|_{H^s},
    \end{aligned}
    \end{equation*}
where the implicit constant in the r.h.s of the above inequality is independent of $\lambda$ by virtue of Lemma \ref{lemmwei}. On the other hand, the identity
\begin{equation*}
    \begin{aligned}
         \left[\omega_{\lambda},\Delta \right] u_n=(\Delta \omega_{\lambda}) u_n-2 \nabla \omega_{\lambda}\cdot \nabla u_n
    \end{aligned}
\end{equation*}
and \eqref{weigheq1} yield 
\begin{equation*}
    \begin{aligned}
         \left\|\mathcal{R}_1([\omega_{\lambda},\Delta]u_n)\right\|_{L^2}&\lesssim  \left\|\Delta \omega_{\lambda}\right\|_{L^{\infty}}\left\|u_n\right\|_{L^2}+\left\|\nabla \omega_{\lambda}\right\|_{L^{\infty}}\left\|\nabla u_n\right\|_{L^2}\\
         & \lesssim \left(\left\|\Delta \omega_{\lambda}\right\|_{L^{\infty}}+\left\|\nabla \omega_{\lambda}\right\|_{L^{\infty}}\right)\left\|u_0\right\|_{H^s}.
    \end{aligned}
\end{equation*}
Gathering all these estimates, there exist constants $c_0$ and $c_1$ (depending on the $L^{\infty}$-norm of the weight $w$ and its derivatives, and independent of $\lambda$) such that
\begin{equation*}
    \frac{d}{dt}\left\|\omega_{\lambda} u_n\right\|_{L^2}^2 \leq c_0 \left\|u_0\right\|_{H^s}\left\|\omega_{\lambda}u_n\right\|_{L^2} +c_1\left\|\nabla u_n\right\|_{L^{\infty}}\left\|\omega_{\lambda}u_n\right\|_{L^2}^2. 
\end{equation*}
Consequently, in view of Gronwall's inequality we arrive at
\begin{equation}\label{weigheq3}
    \begin{aligned}
         \left\|\omega_{\lambda}u_n(t)\right\|_{L^2}&\leq (\left\|\omega_{\lambda}u_{n}(0)\right\|_{L^2}+c_0 \left\|u_0\right\|_{H^s} t)e^{c_1\int_{0}^t\left\|\nabla u_n(s)\right\|_{L^{\infty}}\, ds}.
    \end{aligned}
\end{equation}
From \eqref{aproxres} and the fact that $\omega_{\lambda}$ is bounded, one can take the limit $n\to \infty $ in \eqref{weigheq3} to find
\begin{equation*}
\begin{aligned}
    \left\|\omega_{\lambda} u(t)\right\|_{L^2} &\leq ( \left\|\omega_{\lambda}u_0\right\|_{L^2}+c_0\left\|u_0\right\|_{H^s}t)e^{c_1\int_{0}^t\left\|\nabla u(s)\right\|_{L^{\infty}}\, ds} \\
    & \leq ( \left\|\omega u_0\right\|_{L^2}+c_0\left\|u_0\right\|_{H^s}t)e^{c_1\int_{0}^t\left\|\nabla u(s)\right\|_{L^{\infty}}\, ds}.
\end{aligned}
\end{equation*}
The above inequality and Fatou's lemma imply
\begin{equation}\label{weigheq4}
\begin{aligned}
    \left\|\omega u(t)\right\|_{L^2} &\leq ( \left\|\omega u_0\right\|_{L^2}+c_0\left\|u_0\right\|_{H^s}t)e^{c_1\int_{0}^t\left\|\nabla u(s)\right\|_{L^{\infty}}\, ds}, \hspace{0.4cm} 0\leq t \leq T.
\end{aligned}
\end{equation}
This shows that $u \in L^{\infty}([0,T];L^{2}(\omega^2 \, dx))$. Let us prove that $u \in C([0,T];L^{2}(\omega^2 \, dx))$. Firstly, we claim that $u:[0,T]\mapsto L^2(\omega^2 \, dx dy)$ is weakly continuous. Indeed, for a given $g \in S(\mathbb{R}^d)$,
\begin{equation}\label{weakcont}
    \begin{aligned}
    \left| \int \omega(u(s)-u(t))\omega g\, dx\right|
         &\leq \left|\int \omega(u(s)-u(t))(\omega-\omega_{\lambda})g \, dx  \right|+\left|\int \omega(u(s)-u(t))\omega_{\lambda}g \, dx  \right| \\
         &\lesssim \sup_{t\in [0,T]}\left\|w u(t)\right\|_{L^2}\left\|(\omega-\omega_{\lambda})g\right\|_{L^2}+\left\|u(s)-u(t)\right\|_{L^2}\left\|\ \omega\omega_{\lambda}g\right\|_{L^2}.
    \end{aligned}
\end{equation}
Therefore, since $\left\|\omega \omega_{\lambda} g\right\|_{L^2} \leq \left\|\omega^2 g\right\|_{L^2}<\infty $, using that $g(\omega-\omega_{\lambda}) \to 0$ as $\lambda\to 0$ in $L^2(\mathbb{R}^d)$ (due to Lebesgue dominated convergence theorem), \eqref{weigheq4} and the fact that $u\in C([0,T];H^s(\mathbb{R}^d))$ with $s >0$, letting $\lambda \to 0$ in \eqref{weakcont}, we deduce weak continuity.
\\ \\
On the other hand, the estimate \eqref{weigheq4} yields 
\begin{equation}\label{weigheq5}
    \begin{aligned}
         \left\|\omega (u(t)-u_0)\right\|_{L^2}^2 & =\left\|\omega u(t)\right\|_{L^2}^2+\left\|\omega u_0\right\|_{L^2}^2-2\int \omega u(t) \omega u_0\, dx \\
         & \leq ( \left\|\omega u_0\right\|_{L^2}+c_0t)^2e^{2c_1\int_{0}^t\left\|\nabla u(s)\right\|_{L^{\infty}}\, ds}+\left\|\omega u_0\right\|_{L^2}^2-2\int \omega u(t) \omega u_0\, dx.
    \end{aligned}
\end{equation}
Clearly, weak continuity implies that the right-hand side of \eqref{weigheq5} goes to zero as $t\to 0^{+}$. This shows right continuity at the origin of the map $u:[0,T]\mapsto L^2(\omega^2 \, dx dy)$. Fixing $\tau \in (0,T)$ and using that \eqref{HBO-IVP} is invariant under the transformations, $(x,t)\mapsto (x, t+\tau)$ and  $(x,t)\mapsto (-x,\tau-t)$,  right continuity at the origin entails continuity in all the interval $[0,T]$, in other words $u \in C([0,T];H^{s}(\mathbb{R}^d) \cap L^{2}(\omega^2 \, dx))$. 
\\
The continuous dependence on the initial data can be deduced from its equivalent in $H^{s}(\mathbb{R}^d)$ and employing the above arguments. The proof of Proposition \ref{wellposwei} is now completed.


\section{Proof of Theorem \ref{localweigh}}

When the decay parameter $r\in [0,1]$, the weight $\langle x \rangle^r$ satisfies the hypothesis of Theorem \ref{wellposwei}. Thereby, we may assume that $1<r \leq s$.
\\
Let $u\in C([0,T];H^s(\mathbb{R}^d))$ be a solution of \eqref{HBO-IVP} with initial datum $u_0\in Z_{s,r}(\mathbb{R}^d)$ provided by Theorem \ref{imprwellpos}. We shall prove that $u\in L^{\infty}([0,T];L^2(|x|^{2r}\, dx))$. Once we have established this conclusion, the fact that $u\in C([0,T];L^2(|x|^{2r}\, dx))$ and the continuous dependence on the initial data follows by the same reasoning in the proof of Proposition \ref{wellposwei}. 
\\ \\
We begin by giving a brief sketch of the proof. Let $m$ be a non-negative integer, $0\leq \theta \leq 1$ and write $r=m+1+\theta$. Consider $k=1,2\dots, d$, multiplying \eqref{HBO-IVP} by $w_N^{2+2\theta}x_k^{2m} u$ (where $w_N$ is given by \eqref{intro2}) and integrating in $\mathbb{R}^d$ we obtain
\begin{equation}\label{eqlocalw0}
\begin{aligned}
\frac{1}{2} \frac{d}{dt} \int \big(w_N^{1+\theta}x_k^{m}u \big)^2 \, dx &-\int w_N^{1+\theta} x_k^{m}\mathcal{R}_1\Delta u w_N^{1+\theta} x_k^{m}u \,dx+\int w_N^{1+\theta}x_k^{m}u\partial_{x_1} u w_N^{1+\theta} x_k^{m}u \, dx=0.
\end{aligned}
\end{equation}
Arguing recursively on the size of the parameter $r=m+1+\theta$, starting with $m=0$, we will deduce from previous cases (decay $r\leq (m-1)+1+\theta$), that $u\in L^{\infty}([0,T];Z_{s, r-1}(\mathbb{R}^d))$ and satisfies
\begin{equation}\label{bounderi}
 \sup_{t\in [0,T]}\big( \left\|\langle x \rangle^{r-1}u(t)\right\|_{L^2}+ \sum_{1\leq |\beta| \leq m}\left\|\langle x \rangle^{r-|\beta|} \partial^{\beta} u(t)\right\|_{L^2} \big) \leq C_1
\end{equation}
where $C_1$\footnote{ Since we relay on Gronwall's lemma to attain our estimates, one may expect that $C_1$ depends on $\left\|\langle x \rangle^{r-|\beta|} \partial^{\beta} u_0\right\|_{L^2}$ for each multi-index $1\leq |\beta| \leq m$. However, the interpolation inequality \eqref{prelimneq3} shows that these expressions are bounded by $\left\|u_0\right\|_{H^s}$ and $\left\|\langle x \rangle^{r-1}u_0\right\|_{L^2}$. } depends on $T, \left\|u_0\right\|_{H^s}, \left\|\langle x \rangle^{r-1}u_0\right\|_{L^2}$  and $\int_0^T\left\|u(\tau)\right\|_{W^{1,\infty}(\mathbb{R}^d)}\, d \tau$. With the aim of \eqref{bounderi}, we proceed to estimate the last two term on the left-hand side of \eqref{eqlocalw0} to obtain a differential inequality, which after adding for $k=1,\dots,d$ has the form
\begin{equation}\label{differform}
\frac{d}{dt}\big(\sum_{k=1}^d \left\|w_N^{1+\theta}x_k^m u\right\|_{L^2}^2\big)\leq K_1\big(\sum_{k=1}^d \left\|w_N^{1+\theta}x_k^m u\right\|_{L^2}^2\big)^{1/2}+K_2\big(\sum_{k=1}^d \left\|w_N^{1+\theta}x_k^m u\right\|_{L^2}^2\big) 
\end{equation}
for some positive constants $K_1$ and $K_2$. Then Gronwall's lemma shows
\begin{equation*}
\sum_{k=1}^d \left\|w_N^{1+\theta}x_k^{m}u\right\|_{L^2} \leq C_2
\end{equation*}
and so letting $N\to \infty$, one gets
\begin{equation}\label{eqlocalw0.2}
 \sup_{r\in [0,T]} \left\|\langle x \rangle^{r}u(t)\right\|_{L^2}\lesssim  C_2,
\end{equation}
where $C_2$ is independent of $N$, depends on $T, \left\|u_0\right\|_{H^s}, \left\|\langle x \rangle^{r}u_0\right\|_{L^2}$ and $\int_0^T \left\|u(\tau)\right\|_{W^{1,\infty}(\mathbb{R}^d)}\, d \tau$.

Therefore, we continue in this fashion, increasing  $r=m+1+\theta$ and deducing \eqref{bounderi} in each step to conclude the proof of Theorem \ref{localweigh} (i).
This same procedure also provides a method to deduce Theorem \ref{localweigh} (ii). However, in this case the estimates for the integral equation \eqref{eqlocalw0} require of additional weighted bounds for derivatives of negative order, which will be deduced from the hypothesis $\widehat{u}(0)=0$. This discussion  encloses the scheme of the proof for Theorem \ref{localweigh}.

Next, we state the main considerations to get \eqref{bounderi}. As above, let $r=m+1+\theta$ with $m\geq 1$, consider a fixed integer $1\leq l \leq m$ and a multi-index $\gamma$ of order $l$. We use the \eqref{HBO-IVP} equation to obtain new equations
\begin{equation}\label{derivHBO}
\partial_{t}(\partial^{\gamma}u)-\mathcal{R}_1 \Delta \partial^{\gamma}u+\partial^{\gamma}(u u_{x_1})=0.
\end{equation}
After multiply \eqref{derivHBO} by $w^{2+2\theta}_N x_k^{2m-2|\gamma|}\partial^{\gamma}u$ and integrate over $\mathbb{R}^d$, it is deduced 
\begin{equation} \label{eqlocalw0.1}
\begin{aligned}
\frac{1}{2} \frac{d}{dt} \int \big(w_N^{1+\theta}x_k^{m-|\gamma|}\partial^{\gamma}u \big)^2 \, dx &-\int w_N^{1+\theta} x_k^{m-|\gamma|}\mathcal{R}_1\Delta \partial^{\gamma}u w_N^{1+\theta} x_k^{m-|\gamma|}\partial^{\gamma}u \,dx \\
&+\int w_N^{1+\theta}x_k^{m-|\gamma|} \partial^{\gamma}(u\partial_{x_1} u )w_N^{1+\theta} x_k^{m-|\gamma|}\partial^{\gamma}u \, dx=0.
\end{aligned}
\end{equation}
Estimating the above equivalences for all $k=1,\dots, d$ and each multi-index $\gamma$ with $|\gamma|=l$, we will deduce a closed differential inequality similar to \eqref{differform}, which yields $L^2(\langle x\rangle^{2r-2l} \, dx)$  bounds for all derivative of order $|\gamma|=l$. Then, adding for $l=1,\dots, m$, \eqref{bounderi} follows.

A first step to study \eqref{eqlocalw0} and \eqref{eqlocalw0.1} is to reduce our arguments to bound the dispersive terms corresponding to the second factors on the left-hand sides of these equations. Indeed, we first consider a fixed decay parameter $r=m+1+\theta$ for some nonnegative integer $m$ and $\theta \in [0,1]$. Then,  the nonlinear part of \eqref{eqlocalw0} can be controlled as
\begin{equation*}
\begin{aligned}
\left|\int w_N^{1+\theta}x_k^{m}u\partial_{x_1} u w_N^{1+\theta} x_k^{m}u \, dx\right| \leq \left\|\nabla u\right\|_{L^{\infty}}\left\|w_N^{1+\theta} x_k^{m}u \right\|_{L^2}^2.
\end{aligned}
\end{equation*}
Since our local theory in $H^{s}(\mathbb{R}^d)$ assures that $u\in L^{1}((0,T);W^{1,\infty}(\mathbb{R}^d))$, the above expression leads to an appropriated bound after Gronwall's Lemma. Now, we proceed to bound the nonlinearity in \eqref{eqlocalw0.1}. Here, $m\geq 1$ and we shall assume from previous steps that
\begin{equation}\label{bounderi1}
 \sup_{r\in [0,T]}\big( \left\|\langle x \rangle^{r-2}u(t)\right\|_{L^2}+ \sum_{1\leq |\beta| \leq m-1}\left\|\langle x \rangle^{r-1-|\beta|} \partial^{\beta} u(t)\right\|_{L^2} \big) \leq C_3,
\end{equation}
where the constant $C_3$ has the same dependence of $C_1$ in \eqref{bounderi}, after changing $r$ by $r-1$. We write
\begin{equation}\label{bounderi1.1}
\begin{aligned}
\int w_N^{1+\theta}x_k^{m-|\gamma|} \partial^{\gamma}(u\partial_{x_1} u )w_N^{1+\theta} x_k^{m-|\gamma|}\partial^{\gamma}u \, dx&=\sum_{\gamma_1+\gamma_2=\gamma} c_{\gamma_1,\gamma_2} \int w_N^{1+\theta}x_k^{m-|\gamma|} \partial^{\gamma_1}u\partial^{\gamma_2}\partial_{x_1} u w_N^{1+\theta} x_k^{m-|\gamma|}\partial^{\gamma}u \, dx \\
&=\sum_{\substack{\gamma_1+\gamma_2=\gamma \\ |\gamma_1|=0 \text{ or }|\gamma_1|=|\gamma|}}  (\cdots)  +\sum_{\substack{\gamma_1+\gamma_2=\gamma \\ |\gamma_1|=1}}  (\cdots)+\sum_{\substack{\gamma_1+\gamma_2=\gamma \\ 2\leq |\gamma_1|\leq |\gamma|-1}} (\cdots) \\
&=:B_1+B_2+B_3.
\end{aligned}
\end{equation} 
We proceed to estimate the terms $B_j$, $j=1,2,3$. Formally integrating by parts in the $x_1$ variable gives
\begin{equation*}
\begin{aligned}
B_1=&\frac{1}{2} \int w_N^{1+\theta}x_k^{m-|\gamma|} \partial^{\gamma}u\partial_{x_1} u w_N^{1+\theta} x_k^{m-|\gamma|}\partial^{\gamma}u \, dx -\int \partial_{x_1}\big(w_N^{1+\theta}x_k^{m-|\gamma|}\big) u\partial^{\gamma} u w_N^{1+\theta} x_k^{m-|\gamma|}\partial^{\gamma}u \, dx.
\end{aligned}
\end{equation*}
Then, when $|\gamma|=m$, using that $|\nabla w_N^{1+\theta}| \lesssim |w_N^{\theta}|$ with a constant independent of $N$,  we find
\begin{equation*}
\left\|B_1\right\|_{L^2} \lesssim \big(\left\|u\right\|_{L^{\infty}}+\left\|\nabla u\right\|_{L^{\infty}} \big)\left\|w_N^{1+\theta}\partial^{\gamma}u \right\|_{L^2}^2,
\end{equation*}
which is controlled by the local theory after Gronwall's lemma. Now, when $1\leq |\gamma|<m$, the inequality \eqref{weighNprope} reveals that
\begin{equation*}
|\partial_{x_1}\big(w_N^{1+\theta}x_k^{m-|\gamma|} \big)| \lesssim \langle x \rangle^{m+1+\theta-1-|\gamma|},
\end{equation*}
with a constant independent of $N$, and so
\begin{equation*}
\begin{aligned}
\left\|B_1\right\|_{L^2}\lesssim & \left\|u\right\|_{L^{\infty}}\left\| \langle x \rangle^{r-1-|\gamma|}\partial^{\gamma}u\right\|_{L^2}\left\|w_N^{1+\theta}x_k^{m-|\gamma|}\partial^{\gamma}u \right\|_{L^2} +\left\|\nabla u\right\|_{L^{\infty}}\left\|w_N^{1+\theta}x_k^{m-|\gamma|}\partial^{\gamma}u \right\|_{L^2}^2.
\end{aligned}
\end{equation*}
Our assumption \eqref{bounderi1} shows that the above expression is controlled since $1\leq |\gamma|<m$. This completes the estimate for $B_1$. Consider $B_2$, in this case $|\gamma_1|=1$, then $\partial^{\gamma_2}\partial_{x_1}$ has order $|\gamma|$ and so
\begin{equation*}
\left\|B_2\right\|_{L^2}\lesssim \left\|\nabla u\right\|_{L^{\infty}}\sum_{|\beta|=|\gamma|}\left\|w_N^{1+\theta}x_k^{m-|\gamma|}\partial^{\beta}u\right\|_{L^2}\left\|w_N^{1+\theta}x_k^{m-|\gamma|}\partial^{\gamma}u \right\|_{L^2}.
\end{equation*} 
The above estimate is part of the Gronwall's term collected after adding \eqref{eqlocalw0.1} for all multi-index of fixed order $|\gamma|$. To control the last term, we use that 
\begin{equation*}
w_N^{1+\theta}|x_k|^{m-|\gamma|} \lesssim \langle x \rangle^{1+\theta+m-1-(|\gamma_2|+1)},
\end{equation*}
whenever $\gamma=\gamma_1+\gamma_2$ and $ 2\leq |\gamma_1|$. Then Sobolev's embedding gives,
\begin{equation}\label{bounderi2}
\begin{aligned}
\left\|B_3\right\|_{L^2} &\lesssim \sum_{\substack{\gamma_1+\gamma_2=\gamma \\ 2\leq |\gamma_1|\leq |\gamma|-1}} \left\|\partial^{\gamma_1}u\right\|_{L^{\infty}}\left\|w_N^{1+\theta}x_k^{m-|\gamma|}\partial^{\gamma_2}\partial_{x_1}u\right\|_{L^2}\left\|w_N^{1+\theta}x_k^{m-|\gamma|}\partial^{\gamma}u\right\|_{L^2}\\
&\lesssim \sum_{\substack{\gamma_1+\gamma_2=\gamma \\ 2\leq |\gamma_1|\leq |\gamma|-1}} \left\|J^{d/2+|\gamma_1|+\epsilon}u\right\|_{L^{2}}\left\|\langle x \rangle^{r-1-(|\gamma_2|+1)}\partial^{\gamma_2}\partial_{x_1}u\right\|_{L^2}\left\|w_N^{1+\theta}x_k^{m-|\gamma|}\partial^{\gamma}u\right\|_{L^2},
\end{aligned}
\end{equation}
for any $\epsilon>0$. Since $|\gamma_1|\leq m-1$, taking $0<\epsilon<m+1+\theta-|\gamma_1|-d/2$ and recalling that the regularity $s\geq r=m+1+\theta$, we get
\begin{equation*}
\left\|J^{d/2+|\gamma_1|+\epsilon}u\right\|_{L^{2}} \lesssim \left\|u\right\|_{H^s},
\end{equation*}
for all $|\gamma_1|\leq m-1$. Plugging this information in \eqref{bounderi2} and using \eqref{bounderi1}, we get a controlled estimate for $B_3$. This completes the study of the non-linear term \eqref{bounderi1.1}.

Thus matters are reduced to control the second term on the left-hand sides of \eqref{eqlocalw0} and \eqref{eqlocalw0.1}. Since the estimate for the later can be obtained from the former by changing the roles of $u$ by $\partial^{\gamma}u$, we will mainly focus on the l.h.s of \eqref{eqlocalw0}. Whence we write
\begin{equation}\label{decomp1}
\begin{aligned}
w_N^{1+\theta}x_k^m \mathcal{R}_1 \Delta u=&w_N^{1+\theta}\mathcal{R}_1(x_k^m \Delta u)+w_N^{1+\theta}\left[x_k^m,\mathcal{R}_1\right]\Delta u \\
=&w_N^{1+\theta}\mathcal{R}_1 \Delta(x_k^m  u)+w_N^{1+\theta}\mathcal{R}_1([x_k^m,\Delta]u)+w_N^{1+\theta}\left[x_k^m,\mathcal{R}_1\right]\Delta u \\
=&\mathcal{R}_1(w_N^{1+\theta} \Delta(x_k^m  u))+[w_N^{1+\theta},\mathcal{R}_1]\Delta(x_k^m u)+w_N^{1+\theta}\mathcal{R}_1([x_k^m,\Delta]u) +w_N^{1+\theta}\left[x_k^m,\mathcal{R}_1\right]\Delta u \\
=&\mathcal{R}_1 \Delta(w_N^{1+\theta}x_k^m  u)+\mathcal{R}_1([w_N^{1+\theta},\Delta](x_k^m u))+[w_N^{1+\theta},\mathcal{R}_1]\Delta(x_k^m u)+w_N^{1+\theta}\mathcal{R}_1([x_k^m,\Delta]u) \\
&+w_N^{1+\theta}\left[x_k^m,\mathcal{R}_1\right]\Delta u \\
=:&\mathcal{R}_1 \Delta(w_N^{1+\theta}x_k^m  u)+ Q_1+Q_2+Q_3+Q_4.
\end{aligned}
\end{equation}
To simplify our arguments, the same notation $Q_j$ will be implemented for different parameters $r$ previously fixed. Inserting $\mathcal{R}_1 \Delta(w_N^{1+\theta}x_k^m  u)$ in \eqref{eqlocalw0}, one finds that its contribution is null since the Riesz transform defines a skew-symmetric operator. Accordingly, it remains to bound the $Q_j$-terms to deduce Theorem \ref{localweigh}. 

\subsection{LWP in \texorpdfstring{$Z_{s,r}(\mathbb{R}^d)$ with $r\in[0,3)$ if $d=2$, and $r\in[0,3]$ when $d=3$}{}}

We divide the proof in two main cases.

\emph{Case 1: $r\in [0,2]$}. As discussed, when $r\in [0,1]$, LWP is a consequence of Theorem \ref{wellposwei}. Suppose that $r\geq 1$, so our conclusion is obtained from \eqref{eqlocalw0} with $m=0$, $r=1+\theta\in [1,2]$ with $0\leq \theta \leq 1$. Notice that we do not require to deduce weighted estimates for derivatives. Besides, $Q_3=Q_4=0$ in \eqref{decomp1}, which reduce our arguments to handle the terms $Q_1$ and $Q_2$.
\\ 
We write
\begin{equation*}
\begin{aligned}
Q_1=-2\mathcal{R}_1(\Delta (w_N^{1+\theta})u+\nabla w_N^{1+\theta} \cdot \nabla u).
\end{aligned}
\end{equation*}
Then, the properties of the weight $w_N$ in \eqref{weighNprope} lead to the following estimate
\begin{equation}\label{weigheq6}
\begin{aligned}
\left\|Q_1\right\|_{L^2} \lesssim  \left\|w_N^{\theta}\nabla u\right\|_{L^2} &\lesssim  \left\|\nabla(w_N^{\theta}u)\right\|_{L^2}+\left\|\nabla w_N^{\theta}u\right\|_{L^2} \lesssim \left\|\nabla(w_N^{\theta}u)\right\|_{L^2}+\left\|u\right\|_{L^2}.
\end{aligned}
\end{equation} 
The interpolation inequality \eqref{prelimneq3} shows
\begin{equation}\label{weigheq7}
\left\|\nabla(w_N^{\theta}u)\right\|_{L^2} \lesssim \left\|J^1(w_N^{\theta}u)\right\|_{L^2}\lesssim \left\|w_N^{1+\theta} u\right\|_{L^2}^{\theta/(1+\theta)}\left\|J^{1+\theta}u\right\|_{L^2}^{1/(1+\theta)}.
\end{equation}
Note that this imposes the condition $r=1+\theta \leq s$. Applying in \eqref{weigheq7} Young's inequality and using \eqref{weigheq6}, we bound $Q_1$. To estimate $Q_2$, we apply  Proposition \ref{propconmu} to find 
\begin{equation}\label{weigheq7.5}
    \begin{aligned}
         \left\|Q_2\right\|_{L^2} &\leq \sum_{j=1}^d \left\|[w^{1+\theta}_N,\mathcal{R}_1]\partial_{x_j}^2u\right\|_{L^2} \lesssim \sum_{j=1}^d \sum_{|\beta|=2}\left\|\partial^{\beta} w_N^{1+\theta}\right\|_{L^{\infty}}\left\|u\right\|_{L^2}+\sum_{|\beta|=1}\left\|\partial^{\beta}w^{1+\theta}_N D_{R_1}^{\beta}\partial_{x_j}^2 u\right\|_{L^2}.
    \end{aligned}
\end{equation}
The second term on the r.h.s can be bounded combining Proposition \ref{propapcond}, \eqref{identity1} and \eqref{weighNprope} to obtain
\begin{equation}\label{weigheq8}
\begin{aligned}
\left\|\partial^{\beta}w^{1+\theta}_N D_{R_1}^{e_k}\partial_{x_j}^2 u\right\|_{L^2} &\lesssim \delta_{1,k}\left\|w_N^{\theta} \mathcal{R}_j\partial_{x_j}u \right\|_{L^2}+\left\|w_N^{\theta}\mathcal{R}_1\mathcal{R}_k\mathcal{R}_j\partial_{x_j}u\right\|_{L^2} \lesssim \left\|w^{\theta}_N \partial_{x_j}u\right\|_{L^2},
\end{aligned}
\end{equation}
 $0 < 2\theta <2 \leq d$, which is controlled as in \eqref{weigheq7}. Notice that the above argument fails when $\theta=1$ in dimension $d=2$ (since $w_N^2$ does not satisfies the $A_2(\mathbb{R}^2)$ condition), instead letting $\beta=e_l$, we use the identity \eqref{identity1} to write
\begin{equation}\label{weigheq8.0}
\begin{aligned}
w_ND_{R_1}^{e_l}\partial_{x_j}^2u=&\delta_{1,l}w_N\mathcal{R}_j\partial_{x_j}u +w_N\mathcal{R}_1 \mathcal{R}_l \mathcal{R}_j \partial_{x_j}u \\
=&\delta_{1,l}[w_N,\mathcal{R}_j]\partial_{x_j}u+\delta_{1,l}\mathcal{R}_j(w_N\partial_{x_j}u)+[w_N,\mathcal{R}_1]\partial_{x_j}\mathcal{R}_l \mathcal{R}_j u+\mathcal{R}_1([w_N,\mathcal{R}_l]\partial_{x_j}\mathcal{R}_ju) \\
&+\mathcal{R}_1 \mathcal{R}_l([w_N,\mathcal{R}_j]\partial_{x_j}u)+\mathcal{R}_1\mathcal{R}_l\mathcal{R}_j(w_N\partial_{x_j}u).
\end{aligned}
\end{equation}
Hence, the decomposition \eqref{weigheq8.0} allows us to apply Proposition \ref{propconmu} with one derivative to get
\begin{equation}\label{weigheq8.1}
\begin{aligned}
\sum_{|\beta|=1}\left\|\partial^{\beta}w^{2}_N D_{R_1}^{\beta}\partial_{x_j}^2 u\right\|_{L^2} &\lesssim \sum_{l=1}^2\left\|w_N D_{R_1}^{e_l}\partial_{x_j}^2 u\right\|_{L^2} \lesssim \left\|u\right\|_{L^2}+\left\|w_N \partial_{x_j}u\right\|_{L^2}.
\end{aligned}
\end{equation}
It is worth to notice that the above argument also establishes the bound \eqref{weigheq8} without the aim of Proposition \ref{propapcond}. In this manner, the right-hand side of \eqref{weigheq8} and \eqref{weigheq8.1} can be estimated as in \eqref{weigheq7}. Putting together these results in \eqref{weigheq7.5}, we bound $Q_2$ by Gronwall's terms. Finally, inserting the above information in \eqref{eqlocalw0} with $m=0$ the desire conclusion holds.


\emph{Case 2:} $r\in(2,3)$ if $d=2$ and $r\in(2,3]$ when $d=3$. Our conclusions are obtained from \eqref{eqlocalw0}, setting $m=1$ and $r=2+\theta$, with $0<\theta <1$ if $d=2$ and including $\theta=1$ if $d=3$. We first claim that
\begin{equation} \label{weigheq8.3}
\begin{aligned}
\sup_{t\in[0,T]} \left\|\langle x \rangle^{r-1} \nabla u(t)\right\|_{L^2} \leq M,
\end{aligned}
\end{equation}
with $M$ depending on $\left\|u_0\right\|_{H^s}$, $\left\|\langle x \rangle^{r}u_0\right\|_{L^2}$ and $T$. This estimate is derived from \eqref{eqlocalw0.1} with $m=1$ and $\gamma$ of order 1. Hence, \eqref{weigheq8.3} is established by reapplying the same arguments in the previous case, substituting $u$ by $\partial_{x_l}u$, $l=1,\dots,d$ in each estimate. Notice that in this case, \eqref{weigheq7} is given by
\begin{equation*}
\begin{aligned}
\left\|J^1(w_N^{\theta}\partial_{x_l}u)\right\|_{L^2}&\lesssim \left\|w_N^{1+\theta} \partial_{x_l}u\right\|_{L^2}^{\theta/(1+\theta)}\left\|J^{1+\theta}\partial_{x_l}u\right\|_{L^2}^{1/(1+\theta)} \lesssim \left\|w_N^{1+\theta} \partial_{x_l}u\right\|_{L^2}^{\theta/(1+\theta)}\left\|J^{2+\theta}u\right\|_{L^2}^{1/(1+\theta)},
\end{aligned}
\end{equation*}
which leads to a controlled expression after Young's inequality, since $\left\|w_N^{1+\theta} \partial_{x_l}u\right\|_{L^2}$ is part of the Gronwall's term to be estimated and $2+\theta \leq s$. It remains to study the factors $Q_j$ in \eqref{decomp1}.
To treat $Q_1$, we write
\begin{equation}\label{weigheq8.4}
\begin{aligned}
\left[w_N^{1+\theta},\Delta\right](x_ku)&=-\Delta(w_N^{1+\theta})x_ku-2\nabla(w_N^{1+\theta})\cdot \nabla (x_k u) \\
&=-\Delta(w_N^{1+\theta})x_ku-2\partial_{x_k}(w_N^{1+\theta})  u-2x_k \nabla(w_N^{1+\theta})\cdot \nabla u.
\end{aligned}
\end{equation}
This expression and \eqref{weighNprope} imply
\begin{equation}\label{weigheq8.5}
\begin{aligned}
\left\|Q_1\right\|_{L^2} &\lesssim \left\|\langle x \rangle u\right\|_{L^2}+\left\|w_N^{1+\theta}\nabla u\right\|_{L^2} \lesssim \left\|\langle x \rangle u\right\|_{L^2}+\left\|\langle x \rangle^{1+\theta} \nabla u\right\|_{L^2}.
\end{aligned}
\end{equation}
Notice that $\left\|\langle x \rangle u\right\|_{L^2}$ is bounded by the preceding case and $\left\|\langle x \rangle^{1+\theta} \nabla u\right\|_{L^2}$ by \eqref{weigheq8.3}. To deal with $Q_2$, we gather Proposition \ref{propconmu}, Lemma \ref{lemmaRieszdeco} and \eqref{weighNprope} to find
\begin{equation}\label{weigheq9.1}
\begin{aligned}
\left\|Q_2\right\|_{L^2} \lesssim \left\|x_k u\right\|_{L^2}+\sum_{|\beta|=1}\left\|\partial^{\beta}w^{1+\theta}_ND_{R_1}^{\beta}\Delta(x_k u)\right\|_{L^2} &\lesssim \left\|x_k u\right\|_{L^2}+\left\|w_N^{\theta}D_{R_1}^{\beta}\Delta(x_k u)\right\|_{L^2} \\
 & \lesssim \left\|\langle x \rangle u\right\|_{L^2}+\left\|w_N^{\theta}x_k\nabla u\right\|_{L^2},
\end{aligned}
\end{equation}
which is controlled due to \eqref{weigheq8.3}. To estimate $Q_3$ we employ the following point-wise inequality
\begin{equation}\label{weigheq9.2}  
\begin{aligned}
|Q_3|=|-2w_N^{1+\theta}\mathcal{R}_1\partial_{x_k}u| &\lesssim |w_N^{\theta}\mathcal{R}_1\partial_{x_k}u|+\sum_{l=1}^d|w_N^{\theta} x_l\mathcal{R}_1\partial_{x_k}u| \\
&\lesssim |w_N^{\theta}\mathcal{R}_1\partial_{x_k}u|+\sum_{l=1}^d|w_N^{\theta}[x_l,\mathcal{R}_1]\partial_{x_k}u|+\sum_{l=1}^d|w_N^{\theta}\mathcal{R}_1(x_l \partial_{x_k}u)|,
\end{aligned}
\end{equation}
which hold since $w_N^{1+\theta}\lesssim w_N^{\theta}+|x|w_N^{\theta}$. Thus, recalling \eqref{identity3} to handle the second term on the r.h.s of \eqref{weigheq9.2} and using Proposition \ref{propapcond} with $0\leq \theta <1$ when $d=2$, and with $0\leq \theta\leq 1$ when $d=3$, it is deduced that 
\begin{equation*}
\begin{aligned}
\left\|Q_3\right\|_{L^2} &\lesssim \left\|w_N^{\theta} \partial_{x_k}u\right\|_{L^2}+\left\|w_N^{\theta} u\right\|_{L^2}+\left\|w_N^{\theta}|x| \partial_{x_k}u\right\|_{L^2}  \lesssim \left\|\langle x \rangle^{\theta} u\right\|_{L^2}+\left\|\langle x \rangle^{1+\theta}\nabla u \right\|_{L^2}
\end{aligned}
\end{equation*}
which is controlled by previous cases and \eqref{weigheq8.3}. This complete the estimate for $Q_3$.

Next, we use the identity \eqref{identity3} to write $Q_4$ as 
\begin{equation*}
\begin{aligned}
Q_4= w_N^{1+\theta}[x_k,\mathcal{R}_1] \Delta u =-w_N^{1+\theta}D^{e_k}_{R_1}\Delta u.
\end{aligned}
\end{equation*}
Using again the inequality $w_N^{1+\theta} \lesssim w_N^{\theta}+|x|w_N^{\theta}$, we find
\begin{equation*}
\left\|Q_4\right\|_{L^2} \lesssim  \left\|w_{N}^{\theta}D_{R_1}^{e_k} \Delta u\right\|_{L^2}+\sum_{l=1}^d \left\|w_N^{\theta}x_l D_{R_1}^{e_k}\Delta u\right\|_{L^2}.
\end{equation*}
It is not difficult to see that for all $j=1,\dots,d$
\begin{equation*}
\begin{aligned}
    x_l D_{R_1}^{e_k} \partial_{x_j}^2 u&=-D^{e_l+e_k}_{R_1}\partial_{x_j}^2 u+D_{R_1}^{e_k}(x_l\partial_{x_j}^2  u)\\
    &=-D^{e_l+e_k}_{R_1}\partial_{x_j}^2 u+D_{R_1}^{e_k}\partial_{x_j}(x_l\partial_{x_j}u)-\delta_{j,l}D_{R_1}^{e_k}\partial_{x_j}u.
\end{aligned}
\end{equation*}
Thus, combining the above decomposition, Lemma \ref{lemmaRieszdeco} and Proposition \ref{propapcond} with $0\leq \theta <1$ if $d=2$ or $0\leq \theta \leq 1$ when $d=3$, we obtain
\begin{equation}\label{weigheq9.4}
\left\|Q_4\right\|_{L^2} \lesssim \left\|w_{N}^{\theta}  u\right\|_{L^2}+\left\|w_{N}^{\theta} \nabla u\right\|_{L^2}+\sum_{l=1}^d \left\|w_N^{\theta}x_l \nabla u\right\|_{L^2}.
\end{equation}
The above expression is controlled by previous cases and \eqref{weigheq8.3}. This concludes the estimates for the factors $Q_j$.
\\ 
Finally, gathering the above information in \eqref{eqlocalw0} with $m=1$ and recalling our previous discussions, we have deduced Theorem \ref{localweigh} (i) when $d=2$. In addition, when $d=3$, we have shown that $u\in C([0,T];Z_{r,s}(\mathbb{R}^3))$ with $r\in[0,3]$, $s\geq r$.

\subsection{  \texorpdfstring{LWP in $Z_{s,r}(\mathbb{R}^3)$ $r\in (3,7/2)$}{}}

In this part we complete the proof of Theorem \ref{localweigh} (i) when $d=3$. To obtain our estimates, we consider the differential equation \eqref{eqlocalw0} with $m=2$, $0\leq \theta <1/2$, $r=3+\theta$ and $r\leq s$.
We start deducing weighted estimate for derivatives of $u$. Considering \eqref{eqlocalw0.1} with $m=2$ and $\gamma$ of order $2$, we can reapply the argument when the decay parameter $r$ lies in the interval $ (1,2]$ to deduce
\begin{equation}\label{weigheq9}
\begin{aligned}
\sup_{t\in[0,T]} \sum_{|\beta|=2}\left\|\langle x \rangle^{r-2} \partial^{\beta}u(t)\right\|_{L^2} \leq M_0, 
\end{aligned}
\end{equation}
where $M_0$ depends on $\left\|u_0\right\|_{H^s}$, $\left\|\langle x \rangle^{r}u_0\right\|_{L^2}$ and $T$. 
 Therefore, setting $m=2$ and $\gamma$ of order $1$ in \eqref{eqlocalw0.1}, the inequality \eqref{weigheq9} allows us to argue exactly as in the previous subsection to deduce
\begin{equation}\label{weigheq12}
\begin{aligned}
\sup_{t\in[0,T]} \left\|\langle x \rangle^{r-1} \nabla u(t)\right\|_{L^2} \leq M_1,
\end{aligned}
\end{equation}
with $M_1$ depending on $\left\|u_0\right\|_{H^s}$, $\left\|\langle x \rangle^{r}u_0\right\|_{L^2}$ and $T$. Now we can proceed to estimate the terms $Q_j$ defined by \eqref{decomp1} with $m=2$. 

We can deduce a similar estimate to \eqref{weigheq8.4} dealing with $x_k^2$, and then bounding as in \eqref{weigheq8.5} with the aim of \eqref{weigheq12}, we control $Q_1$. 
The estimate for $Q_2$ is achieved as in \eqref{weigheq9.1} employing Proposition \ref{propconmu},
substituting $x_k$ by $x_k^2$ and controlling the resulting factor by \eqref{weigheq12}. The terms $Q_3$ and $Q_4$ can be controlled from the fact that $w_N^{2+2\theta}$ satisfies the hypothesis of Proposition \ref{propapcond} whenever $0\leq \theta<1/2$. Indeed, writing
\begin{equation}\label{weigheq9.3}
\begin{aligned}
Q_3=-2w_{N}^{1+\theta}\mathcal{R}_1 u-4w_N^{1+\theta}\mathcal{R}_1(x_k\partial_{x_k}u)
\end{aligned}
\end{equation}
and employing identity \eqref{identity2} with $\beta=2e_k$,
\begin{equation*}
\begin{aligned}
Q_4= w_N^{1+\theta}[x_k^2,\mathcal{R}_1] \Delta u =w_N^{1+\theta}D^{2e_k}_{R_1}\Delta u-2w_N^{1+\theta}D^{e_k}_{R_1}\Delta (x_ku)+4w_N^{1+\theta}D^{e_k}_{R_1}\partial_{x_k}u.
\end{aligned}
\end{equation*}
Then Lemma \ref{lemmaRieszdeco} and Proposition \ref{propapcond} imply
\begin{equation}
\begin{aligned}
\left\|Q_3\right\|_{L^2}+\left\|Q_4\right\|_{L^2} &\lesssim \left\|w_N^{1+\theta} u\right\|_{L^2}+\left\|w_N^{1+\theta} x_k \nabla u\right\|_{L^2}
&\lesssim \left\|\langle x \rangle^{3/2} u\right\|_{L^2}+\left\|\langle x \rangle^{r-1} \nabla u \right\|_{L^2},
\end{aligned}
\end{equation}
which is bounded by previous cases and \eqref{weigheq12}. Whence inserting this bound in \eqref{eqlocalw0} yields to the proof of Theorem \ref{localweigh} (i).

\subsection{LWP in \texorpdfstring{$\dot{Z}_{s,r}(\mathbb{R}^2)$, $r\in [3,4)$.}{}}

Here we restrict our arguments to dimension $d=2$. Our conclusions are achieved from \eqref{eqlocalw0}  by setting $m=2$, $0\leq\theta<1$ and so $r=3+\theta$. 
When the initial datum $u_0 \in Z_{s,r}(\mathbb{R}^2)$, $3 \leq r< 4$ and $r\leq s$, we can repeat the comments leading to \eqref{weigheq9} and \eqref{weigheq12} in dimension 3 to deduce
\begin{equation}\label{weigheq10} 
\begin{aligned}
\sup_{t\in[0,T]} \sum_{1\leq |\gamma|\leq 2}\left\|\langle x \rangle^{r-|\gamma|} \partial^{\gamma}u(t)\right\|_{L^2} \leq M_0, 
\end{aligned}
\end{equation}
where $M_0$ depends on $\left\|u_0\right\|_{H^s}$, $\left\|\langle x \rangle^{r}u_0\right\|_{L^2}$ and $T$.   On the other hand, when $\widehat{u}(0,t)=\widehat{u_0}(0)=0$ in $\mathbb{R}^2$, we claim  
\begin{equation}\label{weigheq11}
\begin{aligned}
\sup_{t\in[0,T]} \left\|\langle x \rangle^{\theta} |\nabla|^{-1} u(t)\right\|_{L^2(\mathbb{R}^2)} \lesssim M_1  
\end{aligned}
\end{equation}
for all $0\leq \theta <1$ and $M_1$ depending on $\left\|u_0\right\|_{H^s}$, $\left\|\langle x \rangle^{r}u_0\right\|_{L^2}$ and $T$. Indeed, let $\phi \in C_c^{\infty}(\mathbb{R}^d)$ with $\phi \equiv 1$  when $|\xi|\leq 1$ and write
\begin{equation}
 D_{\xi}^{\theta}\big(|\xi|^{-1}\widehat{u}(\xi)\big) =D_{\xi}^{\theta}\big(|\xi|^{-1}\widehat{u}(\xi) \phi\big)+D_{\xi}^{\theta}\big(|\xi|^{-1}\widehat{u}(\xi)(1-\phi)\big).
\end{equation} 
In sight of the zero mean assumption and Sobolev's embedding 
\begin{equation}\label{weigheq12.1}
||\xi|^{-1}\widehat{u}(\xi)|\lesssim \left\|\nabla \widehat{u}\right\|_{L^{\infty}} \lesssim \left\|\langle x \rangle^{2+\epsilon}u\right\|_{L^2}
\end{equation} 
for all $\epsilon>0$. Hence, from \eqref{prelimneq0.1} and Lemma \ref{lemmapreli1} one deduces
\begin{equation}\label{weigheq12.2}
\begin{aligned}
\left\|D_{\xi}^{\theta}\big(|\xi|^{-1}\widehat{u}(\xi)\big)\right\|_{L^{2}} &\lesssim \left\|D_{\xi}^{\theta}\big(|\xi|^{-1}\widehat{u}(\xi) \phi\big)\right\|_{L^2}+\left\||\xi|^{-1}\widehat{u}(\xi)(1-\phi)\right\|_{H^{1}_{\xi}} \\
&\lesssim \left\||\xi|^{-1}\widehat{u}(\xi) \phi\right\|_{L^2}+\left\|\mathcal{D}_{\xi}^{\theta}\big(|\xi|^{-1}\widehat{u}(\xi) \phi\big)\right\|_{L^2}+\left\| \widehat{u}\right\|_{H^{1}_{\xi}}\left\|\nabla \phi\right\|_{L^{\infty}} \\
&\lesssim \left\|\nabla \widehat{u}\right\|_{L^{\infty}}\left\|\phi\right\|_{L^2}+\left\|\mathcal{D}_{\xi}^{\theta}\big(|\xi|^{-1}\widehat{u}(\xi)\big)\phi\right\|_{L^2}+\left\||\xi|^{-1}\widehat{u}(\xi)\mathcal{D}_{\xi}^{\theta}\phi\right\|_{L^2}+\left\|\widehat{u}\right\|_{H^{1}_{\xi}} \\
&\lesssim \left\|\nabla \widehat{u}\right\|_{L^{\infty}}\left\|\phi\right\|_{L^2}+\big(\left\|\nabla\widehat{u}\right\|_{L^{\infty}}+\left\||\xi|^{-1}\widehat{u}\right\|_{L^{\infty}}\big)\big(\left\||\xi|^{-\theta}\phi\right\|_{L^2}+\left\|\phi\right\|_{L^2}\big)+\left\|\widehat{u}\right\|_{H^{1}_{\xi}}.
\end{aligned}
\end{equation}
Consequently, the above estimate and \eqref{weigheq12.1} yield 
\begin{equation}\label{weigheq14}
\sup_{t\in [0,T]} \left\| \langle x \rangle^{\theta} |\nabla|^{-1}u\right\|_{L^2}\lesssim \sup_{t\in [0,T]}\left\|\langle x \rangle^{2+\epsilon}u(t)\right\|_{L^{2}}.
\end{equation}
Since the right-hand side of the above inequality is bounded by previous cases whenever $\epsilon<1$, the proof of \eqref{weigheq12.2} is now completed. In this manner, with the aim of \eqref{weigheq10} and \eqref{weigheq11} we proceed to estimate the terms $Q_j$ given by \eqref{decomp1} with $m=2$. 
\\ \\
The analysis of $Q_1$ and $Q_2$ is obtained by implementing the same ideas leading to \eqref{weigheq8.5} and \eqref{weigheq9.1} respectively. To estimate $Q_3$, we write
\begin{equation}\label{weigheq14.01}
\begin{aligned}
Q_3&=-2w_N^{1+\theta}\mathcal{R}_1 u-4w_N^{1+\theta}\mathcal{R}_1(x_k\partial_{x_k}u)=2w_N^{1+\theta}\mathcal{R}_1u-4w_N^{1+\theta}\mathcal{R}_1(\partial_{x_k}(x_k u)),
\end{aligned}
\end{equation}
then using that $w_N^{1+\theta}\lesssim w_N^{\theta}+w_N^{\theta}|x|$, it is not difficult to deduce a similar estimate to \eqref{weigheq9.2} to find
\begin{equation*}
    \left\|Q_3\right\|_{L^2}\lesssim \left\|\langle x \rangle^{\theta} |\nabla|^{-1}u\right\|_{L^2}+\left\|\langle x\rangle^{1+\theta} u\right\|_{L^2}+\left\|\langle x \rangle^{2+\theta}\nabla u\right\|_{L^2}.
\end{equation*}
Let us detail which estimate involves the negative derivative in the above expression. Arguing as in \eqref{weigheq9.2} to study the first factor on the r.h.s of \eqref{weigheq14.01}, we have
\begin{equation*}
    |w_N^{1+\theta}\mathcal{R}_1 u| \lesssim  |w^{\theta}_N \mathcal{R}_1 u|+\sum_{l=1}^d |w^{\theta}[x_l,\mathcal{R}_1]u|+|w^{\theta}\mathcal{R}_1(x_l u)|.
\end{equation*}
Since $w^{\theta}[x_l,\mathcal{R}_1]u=-w^{\theta}D_{R_1}^{e_l}(u)$,  Lemma \ref{lemmaRieszdeco} shows that this expression is bounded by $\left\|\langle x \rangle^{\theta} |\nabla|^{-1}u\right\|_{L^2}$.
To study $Q_4$, we consider the identity
\begin{equation}\label{weigheq14.02}
\begin{aligned}
Q_4=w_N^{1+\theta}\left[x_k^2,\mathcal{R}_1\right]\Delta u= w_N^{1+\theta} D_{R_1}^{2e_k}\Delta u-2w_N^{1+\theta}D_{R_1}^{e_k}\Delta(x_k u)+4w_N^{1+\theta}D_{R_1}^{e_k}\partial_{x_k}u.
\end{aligned}
\end{equation}
Then using that $w_{N}^{1+\theta}\lesssim w_N^{\theta}+w_N^{\theta}|x|$, by a similar reasoning to the deduction of \eqref{weigheq9.4} we find
\begin{equation}
    \left\|Q_4\right\|_{L^2}\lesssim \left\|\langle x \rangle^{\theta}  |\nabla|^{-1} u\right\|_{L^2}+ \left\|\langle x\rangle^{1+\theta}u\right\|_{L^2}+\left\|\langle x\rangle^{2+\theta}\nabla u\right\|_{L^2}.
\end{equation}
Once again, it is worth pointing out which expressions require to consider negative  derivatives following the ideas behind \eqref{weigheq9.4} to control $Q_4$. Indeed, this procedure yields the identities
\begin{equation}
 x_lD^{2e_k}_{R_1}\Delta u=D_{R_1}^{2e_k+e_l}\Delta u+D_{R_1}^{2e_k}\Delta(x_l u)-2D_{R_1}^{2e_k}\partial_{x_l}u
\end{equation}
and
\begin{equation}
\begin{aligned}
x_lD_{R_1}^{e_k}\partial_{x_k}u&=[x_l,D_{R_1}^{e_k}]\partial_{x_k} u+D_{R_1}^{e_k}(x_l\partial_{x_k}u) \\
&=-D_{R_1}^{e_l+e_k}\partial_{x_k}u+D_{R_1}^{e_k}\partial_{x_k}(x_lu)-\delta_{k,l}D_{R_1}^{e_k}u.
\end{aligned}
\end{equation}
Hence,  we use Lemma \ref{lemmaRieszdeco} and Proposition \ref{propapcond} to get
$$\left\|w_N^{\theta} D_{R_1}^{2e_k+e_l}\Delta u\right\|_{L^2}+\left\|w_N^{\theta}D_{R_1}^{e_l+e_k}\partial_{x_k}u\right\|_{L^2} +\left\|D_{R_1}^{e_k}u\right\|_{L^2}\lesssim \left\|w_N^{\theta} |\nabla|^{-1}u\right\|_{L^2}.$$
Finally, from the previous conclusions we have completed the proof of Theorem \ref{localweigh} (ii) for the 2-dimensional case.

\subsection{LWP in \texorpdfstring{$\dot{Z}_{s,r}(\mathbb{R}^3)$, $r\in [7/2,9/2)$.}{}} 

Here we assume that $r\in[7/2,9/2)$  with $r\leq s$ and $\widehat{u_0}(0)=0$. As usual, letting $r=1+m+\theta$,  our estimates are derived from \eqref{eqlocalw0} with $m=2$, $1/2\leq \theta \leq 1$ when $r\in [7/2,4]$, and setting $m=3$, $0\leq \theta <1/2$ if $r\in (4,9/2)$. By recurring arguments employing \eqref{eqlocalw0.1} and 
proceedings cases, starting with the derivatives of higher order and then descending to those of order $1$, it is not difficult to observe
\begin{equation}\label{weigheq16.1}
\begin{aligned}
\sup_{t\in[0,T]} \sum_{1\leq |\beta|\leq m}\left\|\langle x \rangle^{r-|\beta|} \partial^{\beta}u(t)\right\|_{L^2} \leq M,
\end{aligned}
\end{equation}
where $M$ depends on $\left\|u_0\right\|_{H^s}$, $\left\|\langle x \rangle^{r}u_0\right\|_{L^2}$ and $T$. On the other hand, we claim
\begin{equation}\label{weigheq16.0}
\sup_{t\in [0,T]} \left\| \langle x \rangle^{\tilde{\theta}} |\nabla|^{-1}u\right\|_{L^2} \leq M,
\end{equation}
for all $0\leq \tilde{\theta} <3/2$. As above, we let $\phi \in C_c^{\infty}(\mathbb{R}^3)$ such that $\phi(\xi)= 1$ when $|\xi|\leq 1$. We decompose according to
\begin{equation*}
 D_{\xi}^{\tilde{\theta}}\big(|\xi|^{-1}\widehat{u}(\xi)\big) =D_{\xi}^{\tilde{\theta}}\big(|\xi|^{-1}\widehat{u}(\xi) \phi\big)+D_{\xi}^{\tilde{\theta}}\big(|\xi|^{-1}\widehat{u}(\xi)(1-\phi)\big).
\end{equation*}
Given that $\tilde{\theta} \leq 2$, from Sobolev's embedding 
\begin{equation}\label{weigheq16.01}
\begin{aligned}
\left\|D_{\xi}^{\tilde{\theta}}\big(|\xi|^{-1}\widehat{u}(\xi)(1-\phi)\big)\right\|_{L^2} \lesssim \left\| |\xi|^{-1}\widehat{u}(\xi)(1-\phi)\right\|_{H^2_{\xi}} \lesssim \left\|\widehat{u}\right\|_{H^2_{\xi}}\lesssim \left\|\langle x\rangle^{2} u\right\|_{L^2}.
\end{aligned}
\end{equation}
Consequently, it remains to estimate the $L^2$-norm of $D_{\xi}^{\tilde{\theta}}\big(|\xi|^{-1}\widehat{u}(\xi) \phi(\xi)\big)$. The assumption $\widehat{u}(0)=0$ along with Sobolev's embedding yield
\begin{equation}\label{weigheq17}
||\xi|^{-1}\widehat{u}(\xi)|\lesssim \left\|\nabla \widehat{u}\right\|_{L^{\infty}}\lesssim \left\|\langle x\rangle^{5/2+\epsilon} u\right\|_{L^2}.
\end{equation} 
Let us suppose first that $0\leq \tilde{\theta} \leq 1$, the above display then shows
\begin{equation*}
\begin{aligned}
\left\|D^{\tilde{\theta}}_{\xi}\big(|\xi|^{-1}\widehat{u}(\xi) \phi\big)\right\|_{L^2} \lesssim \left\||\xi|^{-1}\widehat{u}(\xi) \phi\right\|_{H^{1}_{\xi}} & \lesssim \left\||\xi|^{-1}\widehat{u}(\xi) (\phi+\nabla \phi)\right\|_{L^2}+\left\||\xi|^{-2} \widehat{u}\phi\right\|_{L^2}+ \left\||\xi|^{-1}\nabla \widehat{u}\phi\right\|_{L^2} \\
&\lesssim \left\| \nabla \widehat{u} \right\|_{L^{\infty}}\big(\left\|\phi\right\|_{L^2}+\left\|\nabla\phi\right\|_{L^2}+\left\||\cdot|^{-1}\phi\right\|_{L^2}\big) \\
&\lesssim \left\|\langle x \rangle^{5/2+\epsilon} u \right\|_{L^2},
\end{aligned}
\end{equation*}
where we have used that $|\xi|^{-1} \in L^{2}_{loc}(\mathbb{R}^3)$. 
 This concludes \eqref{weigheq16.0} 
as soon as $0\leq \tilde{\theta} \leq 1$. To deduce \eqref{weigheq16.0} when $1<\tilde{\theta}<3/2$, we let $0<\theta^{\ast}<1/2$ and equivalently we shall bound the $L^2$-norm of the expression
 \begin{equation}\label{weigheq16.4}
 \begin{aligned}
 D^{\theta^{\ast}}_{\xi}\partial_{\xi_l}\big(|\xi|^{-1}\widehat{u}(\xi)\phi\big)=& -D^{\theta^{\ast}}_{\xi}\big( |\xi|^{-3}\xi_l\widehat{u}(\xi)\phi\big)-iD^{\theta^{\ast}}_{\xi}\big(|\xi|^{-1}\widehat{x_lu}(\xi)\phi\big) \\
 &+ D^{\theta^{\ast}}_{\xi}\big(|\xi|^{-1}\widehat{u}(\xi)\partial_{\xi_l}\phi\big), 
 \end{aligned}
 \end{equation}
for all $l=1,2,3$. Since $\partial_{\xi_l}\phi$ is supported outside of the origin, the last term on the r.h.s of \eqref{weigheq16.4} is bounded as in \eqref{weigheq16.01}. To control the remaining parts we require a preliminary result.
\begin{lemma}\label{lemmwei1} 
Let $\phi, \psi \in C^{\infty}_c(\mathbb{R}^3)$ and $0< \theta^{\ast}<\frac{1}{2}$ fixed, then
\begin{equation}\label{weigheq16.1.1}
 \begin{aligned}
 \left\|\phi \mathcal{D}^{\theta^{\ast}}(|\cdot|^{-1}\psi) \right\|_{L^2} \lesssim_{\theta^{\ast},\phi,\psi} 1
 \end{aligned}
 \end{equation}
 and 
 \begin{equation}\label{weigheq16.2}
 \begin{aligned}
 \left\|\phi D^{\theta^{\ast}}(|\cdot|^{-1}\psi) \right\|_{L^2} \lesssim_{\theta^{\ast},\phi,\psi} 1.
 \end{aligned}
 \end{equation}
\end{lemma}

\begin{proof} We write 
\begin{equation*}
\begin{aligned}
\left\|\phi \mathcal{D}^{\theta^{\ast}}(|\cdot|^{-1}\psi)(\xi) \right\|_{L^2}^2&=\int_{\mathbb{R}^{3}\times \mathbb{R}^{3}} \frac{|\phi(\xi)|^2 \left||\xi|^{-1}\psi(\xi)-|\eta|^{-1} \psi(\eta)\right|^2}{|\xi-\eta|^{3+2\theta^{\ast}}}\, d\eta \, d\xi \\
&\lesssim \int_{\mathbb{R}^{3}\times \mathbb{R}^{3}} \frac{|\phi(\xi)|^2}{|\xi|^2}\frac{ \left|\psi(\xi)-\psi(\eta)\right|^2}{|\xi-\eta|^{3+2\theta^{\ast}}}\, d\eta \, d\xi + \int_{\mathbb{R}^{3}\times \mathbb{R}^{3}} |\phi(\xi)|^2 \frac{ \left||\xi|^{-1} -|\eta|^{-1}\right|^2\left|\psi(\eta)\right|^2}{|\xi-\eta|^{3+2\theta^{\ast}}}\, d\eta \, d\xi \\
&=\widetilde{\mathcal{I}}+\widetilde{\mathcal{II}}.
\end{aligned}
\end{equation*}
From \eqref{prelimneq0.11} and the fact that $|\xi|^{-1}\phi(\xi)\in L^2(\mathbb{R}^3)$,
\begin{equation*}
\begin{aligned}
\widetilde{\mathcal{I}} \lesssim \left\||\cdot|^{-1}\phi\mathcal{D}^{\theta^{\ast}}\psi\right\|_{L^2}^2 \lesssim \big( \left\|\psi\right\|_{L^{\infty}}+\left\|\nabla\psi\right\|_{L^{\infty}}\big)^2 \left\||\cdot|^{-1}\phi\right\|_{L^2}^2.
\end{aligned}
\end{equation*}
On the other hand, gathering together Fubinni's theorem, H\" older's inequality and Hardy-Littlewood-Sobolev inequality we find
\begin{equation*}
\begin{aligned}
\widetilde{\mathcal{II}} \lesssim \int_{\mathbb{R}^{3}\times \mathbb{R}^{3}} \frac{|\psi(\eta)|^2}{|\eta|^2} \frac{1}{|\xi-\eta|^{3-(2-2\theta^{\ast})}} \frac{|\phi(\xi)|^2}{|\xi|^2}\, d\eta \, d\xi &\lesssim \left\||\eta|^{-2}|\psi(\eta)|^2 \frac{1}{|\cdot|^{3-(2-2\theta^{\ast})}}\ast ||\cdot|^{-1}\phi(\cdot)|^2(\eta) \right\|_{L^1} \\
&\lesssim  \left\||\cdot|^{-1}\psi\right\|_{L^{2p}}^2\left\||\cdot|^{-1}\phi\right\|_{L^{2q}}^2,
\end{aligned}
\end{equation*}
where in order to control the above expression one must assure that $1<p,q<3/2$ with
\begin{equation*}
\frac{1}{q}=\frac{5}{3}-\frac{1}{p}-\frac{2\theta^{\ast}}{3}, \hspace{0.5cm} 0<\theta^{\ast}<1/2.
\end{equation*}
Note that $2/3<1/q<1$, if and only if, $(2-2\theta^{\ast})/3<1/p < (3-2\theta^{\ast})/3$, and since $2/3<1/p<1$, we get $$\frac{2}{3}<\frac{1}{p} < \frac{3-2\theta^{\ast}}{3}.$$
Consequently, for fixed $\theta^{\ast}\in (0,\frac{1}{2})$, one can always find $p$ assuming the above condition. This establishes \eqref{weigheq16.1.1}. To prove the last assertion of the lemma, we use the commutator estimate \eqref{prelimneq5} to find
\begin{equation*}
\begin{aligned}
 \left\|\phi D^{\theta^{\ast}}(|\cdot|^{-1}\psi)\right\|_{L^2} & \lesssim \left\|[D^{\theta^{\ast}},\phi]|\cdot|^{-1}\psi\right\|_{L^2}+\left\|D^{\theta^{\ast}}(\phi|\cdot|^{-1}\psi)\right\|_{L^2} \\
 &  \lesssim \big(\left\||\cdot|^{\theta^{\ast}}\widehat{\phi}\right\|_{L^{1}}+\left\|\mathcal{D}^{\theta^{\ast}}\phi\right\|_{L^{\infty}}+\left\|\phi\right\|_{L^{\infty}}\big)\left\||\cdot|^{-1}\psi\right\|_{L^2}+\left\|\phi \mathcal{D}^{\theta^{\ast}}(|\cdot|^{-1}\psi)\right\|_{L^2}
\end{aligned}
\end{equation*}
which is bounded by \eqref{weigheq16.1}.
\end{proof}
Now we can estimate the first term on the r.h.s of \eqref{weigheq16.4}. In view of the zero mean assumption and Sobolev's embedding we get
\begin{equation}\label{weigheq16.5}
||\xi|^{-2}\xi_l \widehat{u}(\xi)| \lesssim \left\|\nabla \widehat{u}\right\|_{L^{\infty}} \lesssim \left\|\langle x \rangle^{5/2+\epsilon}u\right\|_{L^2}, 
\end{equation}
where we have set $\epsilon>0$ small to control the above expression by the result in Theorem \ref{localweigh} (i). Thus, let $\tilde{\phi}\in C_c^{\infty}(\mathbb{R}^2)$ with $\phi \tilde{\phi}=\phi$, combining \eqref{weigheq16.5}, Lemmas \ref{lemmapreli1} and \ref{weigheq16.2} we get 
\begin{equation*}
\begin{aligned}
&\left\|D_{\xi}^{\theta^{\ast}}\big(|\xi|^{-3}\xi_l\widehat{u}(\xi)\tilde{\phi}(\xi) \phi(\xi)\big)\right\|_{L^2} \\
&\hspace{1cm}\lesssim \left\||\xi|^{-3}\xi_l\widehat{u}(\xi) \phi(\xi)\right\|_{L^2}+ \left\||\xi|^{-1}\tilde{\phi} \mathcal{D}_{\xi}^{\theta^{\ast}}(|\xi|^{-2}\xi_l\widehat{u}\phi)\right\|_{L^2}+\left\||\xi|^{-2}\xi_l\widehat{u}\phi \mathcal{D}^{\theta^{\ast}}(|\cdot|^{-1}\tilde{\phi})\right\|_{L^2} \\
&\hspace{1cm}\lesssim \big(\left\||\cdot|^{-1}\phi\right\|_{L^2}+ \left\||\cdot|^{-1-\theta^{\ast}}\tilde{\phi} \right\|_{L^2}+\left\|\phi \mathcal{D}^{\theta^{\ast}}(|\cdot|^{-1}\tilde{\phi})\right\|_{L^2} \big)\left\|\nabla \widehat{u}\right\|_{L^{\infty}}\\
&\hspace{1cm} \lesssim \left\|\langle x \rangle^{5/2+\epsilon}u\right\|_{L^2}.
\end{aligned}
\end{equation*}
 To deal with the second term on the r.h.s of \eqref{weigheq16.4}, we use Lemma \ref{lemmwei1} to find
 \begin{equation*}
 \begin{aligned}
 \left\|D^{\theta^{\ast}}_{\xi}\big(|\cdot|^{-1}\widehat{x_lu}\phi \big)\right\|_{L^2} &\lesssim \left\|[D^{\theta^{\ast}}_{\xi},\widehat{x_lu}\tilde{\phi}]|\cdot|^{-1}\phi\right\|_{L^2}+\left\|\widehat{x_lu}\tilde{\phi}D^{\theta^{\ast}}_{\xi}(|\cdot|^{-1}\phi)\right\|_{L^2} \\
 &\lesssim \left\||\cdot|^{\theta^{\ast}} (x_l u \ast \tilde{\phi}^{\vee})\right\|_{L^1}\left\||\cdot|^{-1}\phi\right\|_{L^2}+\left\|\langle x\rangle^{5/2+\epsilon}u\right\|_{L^{2}}\left\|\tilde{\phi}D^{\theta^{\ast}}_{\xi}(|\cdot|^{-1}\phi)\right\|_{L^2} \\
 &\lesssim \left\||x|^{\theta^{\ast}}x_l u\right\|_{L^1} \left\| \tilde{\phi}^{\vee}\right\|_{L^1}\left\||\cdot|^{-1}\phi\right\|_{L^2}+\left\|x_l u \right\|_{L^1}\left\| |\cdot|^{\theta^{\ast}}\tilde{\phi}^{\vee}\right\|_{L^1}\left\||\cdot|^{-1}\phi\right\|_{L^2} \\
 &\hspace{0.8cm}+\left\|\langle x\rangle^{5/2+\epsilon}u\right\|_{L^{2}}\left\|\tilde{\phi}D^{\theta^{\ast}}_{\xi}(|\cdot|^{-1}\phi)\right\|_{L^2} \\
  &\lesssim \left\|\langle x\rangle^{5/2+\theta^{\ast}+\epsilon} u\right\|_{L^2},
 \end{aligned}
 \end{equation*} 
 where we have used that $\left\|\widehat{x_lu}\right\|_{L^{\infty}}\lesssim \left\|\langle x \rangle^{5/2+\epsilon}u\right\|_{L^{2}}$ for all $\epsilon>0$. Notice that when $0<\epsilon<1-\theta^{\ast}$, Theorem \ref{localweigh} (i) assures that the r.h.s of the above inequality is controlled. This shows that \eqref{weigheq16.4} is bounded for all $l=1,2,3$, which establishes \eqref{weigheq16.0}.


\subsubsection*{Proof of LWP in $\dot{Z}_{s,r}(\mathbb{R}^3)$, $r\in [7/2,9/2)$.}

In light of \eqref{weigheq16.1}, \eqref{weigheq16.0} and Proposition \ref{propapcond} with $0\leq \theta \leq 1$, one can employ the same line of arguments leading to LWP in $\dot{Z}_{s,r}(\mathbb{R}^2)$, $r\in [3,4)$ to deduce the same conclusion in $ Z_{s,r}(\mathbb{R}^3)$ $r\in [7/2,4]$, $r\leq s$ (the extension to $r=4$ is given by the fact that $w_N^2$ satisfies the $A_2(\mathbb{R}^3)$ condition). 

Accordingly, it remains to establish LWP when the decay parameter $r\in (4,9/2)$. This conclusion is obtained from \eqref{eqlocalw0} with $m=3$ and $0<\theta < 1/2$. Under these restrictions, the estimates for $Q_1$, $Q_2$ and $Q_3$ follow from  \eqref{weigheq16.1} and recurring arguments. Finally, in view of identity \eqref{identity2} with $\gamma=3e_k$ and using that $w^{2+2\theta}_N$ satisfies the $A_2(\mathbb{R}^3)$ condition when $0<\theta <1/2$, it is seen that 
\begin{equation}
\begin{aligned}
\left\|Q_4 \right\|_{L^2} 
&\lesssim  \left\|\langle x \rangle^{1+\theta}|\nabla|^{-1} u\right\|_{L^2}+\left\|\langle x \rangle^{2+\theta} u\right\|_{L^2}+\left\|\langle x \rangle^{r-1} \nabla u\right\|_{L^2},
\end{aligned}
\end{equation}
which is bounded by previous cases, \eqref{weigheq16.1} and \eqref{weigheq16.0}. This completes the proof of the Theorem \ref{localweigh} (ii). 

\section{Proof of Theorem \ref{sharpdecay}.}


We begin by introducing some notation and general considerations independent of the dimension to be applied in the proof of Theorem \ref{sharpdecay}. We split $F_3^k$ defined by \eqref{eqsharp1} as 
\begin{equation} \label{eqsharp2}
F_{3,1}^{k}(t,\xi,f)=\partial_{\xi_k}^3\big(it\xi_1|\xi|\big)e^{it\xi_1|\xi|}f(\xi), \text{ and } F_{3,2}^{k}(t,\xi,f)=F_{3}^{k}(t,\xi,f)-F_{3,1}^{k}(t,\xi,f).
\end{equation}
In addition, we define $\widetilde{F}_{3,1}^k$ and $\widetilde{F}_{3,2}^k$ as in \eqref{primeF}, that is,
\begin{equation}
\begin{aligned}
&\widetilde{F}_{3,l}^k(t,\xi,f)=e^{-it\xi_1|\xi|}F_{3,l}^k(t,\xi,f), \hspace{0.5cm} l=1,2.
\end{aligned}
\end{equation}
Without loss of generality we shall assume that $t_1=0<t_2$, i.e., $u_0\in Z_{d/2+2,d/2+2}(\mathbb{R}^d)$. The solution of the IVP \eqref{HBO-IVP} can be represented by Duhamel's formula
\begin{equation}\label{inteequ}
u(t)=e^{t\mathcal{R}_1 \Delta }u_0-\int_0^{t} e^{(t-\tau)\mathcal{R}_1 \Delta } u(\tau)\partial_{x_1} u(\tau) \, d\tau
\end{equation}
or equivalently via the Fourier transform 
\begin{equation*}
\widehat{u}(t)=e^{it\xi_1|\xi| }\widehat{u_0}-\frac{i}{2}\int_0^{t} e^{i(t-\tau)\xi_1|\xi| } \xi_1 \widehat{u^2}(\tau) \, d\tau.
\end{equation*}
By means of the notation introduced in \eqref{eqsharp1} and \eqref{eqsharp2}, we have for $k=1,2$ that
\begin{equation}\label{inteqweigh}
\partial_{\xi_k}^3 \widehat{u}(t)=\sum_{m=1}^2F_{3,m}^{k}(t,\xi,\widehat{u_0})-\frac{i}{2}\int_0^{t} F_{3,m}^{k}(t-\tau,\xi,\xi_1\widehat{u^2}) \, d\tau.
\end{equation}
Notice that $\partial_{\xi_k}^3\big(\xi_1|\xi|\big)$ is locally integrable in $\mathbb{R}^2$ but not square integrable at the origin. The idea is to use this fact to determinate that all terms in \eqref{inteqweigh} except $F_{3,1}^{k}(t,\xi,\widehat{u_0})$ have the appropriate decay at a later time in dimension $d=2$. When $d=3$, we shall use that for $\phi \in C_{c}^{\infty}(\mathbb{R}^3)$ , $\mathcal{D}_{\xi}^{1/2}(\partial_{\xi_k}^3\big(\xi_1|\xi|\big)\phi)(\xi) \notin L^{2}(\mathbb{R}^3)$  to reach the same conclusion. At the end, these facts lead to the proof of Theorem \ref{sharpdecay}.
\\ \\
Next, we proceed to infer some estimates for $F^k_{3,l}(t,\xi,f)$ and $\widetilde{F}^k_{3,l}(t,\xi,f)$, assuming that $f$ is a sufficiently regular function with enough decay and setting $0\leq t \leq T$. Let $a, b\in \mathbb{R}$, in view of the identities \eqref{eqsharp1.1} and \eqref{eqsharp2.1}, it is not difficult to deduce
\begin{equation}\label{eqsharp3.3}
\begin{aligned}
\left\|\langle \xi \rangle^{a}\widetilde{F}_{3,2}^k(t,\xi,f)\right\|_{H^b_{\xi}}
 \lesssim_{T}&\sum_{m=0}^3 \sum_{j=0}^{3-m}\left\|\langle \xi \rangle^{a}(\partial_{\xi_k}(it\xi_1|\xi|))^j\partial_{\xi_k}^mf\right\|_{H^{b}_{\xi}} \\
&+\sum_{m=0}^1 \sum_{j=0}^{1-m}\left\|\langle \xi \rangle^{a}\partial_{\xi_k}^2(it\xi_1|\xi|)(\partial_{\xi_k}(it\xi_1|\xi|))^j\partial_{\xi_k}^mf \right\|_{H^{b}_{\xi}}.
\end{aligned}
\end{equation}
In particular, since our arguments in dimension $d=3$ require of localization in frequency with a function $\phi \in C^{\infty}_c(\mathbb{R}^3)$, the same reasoning yields
\begin{equation}\label{eqsharp3.4} 
\begin{aligned}
\left\|\langle \xi \rangle^{a}\widetilde{F}_{3,2}^k(t,\xi,f)\phi\right\|_{H^b_{\xi}}  \lesssim &\sum_{m=0}^3 \sum_{j=0}^{3-m}\left\|\langle \xi \rangle^{a}(\partial_{\xi_k}(it\xi_1|\xi|))^j\partial_{\xi_k}^mf \phi\right\|_{H^{b}_{\xi}} \\
&+\sum_{m=0}^1 \sum_{j=0}^{1-m}\left\|\langle \xi \rangle^{a}\partial_{\xi_k}^2(it\xi_1|\xi|)(\partial_{\xi_k}(it\xi_1|\xi|))^j\partial_{\xi_k}^mf \phi\right\|_{H^{b}_{\xi}}.
\end{aligned}
\end{equation}
On the other hand, since \eqref{eqsharp2.1} implies that $|\partial_{\xi_k}^{l}(\xi_1|\xi|)|\lesssim \langle\xi \rangle^{2-l}$, $l=1,2$, one can take $b=0$ in \eqref{eqsharp3.3} to find
\begin{equation}\label{eqsharp3} 
\begin{aligned}
\left\|\langle \xi \rangle^{a}F_{3,2}^k(t,\xi,f)\right\|_{L^2}&=\left\|\langle \xi \rangle^{a}\tilde{F}_{3,2}^k(t,\xi,f)\right\|_{L^2} \lesssim \sum_{m=0}^3 \sum_{j=0}^{3-m}\left\|\langle \xi \rangle^{a+j}\partial_{\xi_k}^mf\right\|_{L^2}.
\end{aligned}
\end{equation}
We can now return to the proof of Theorem \ref{sharpdecay}. We divide our arguments according to the dimension.


\subsection{ Dimension \texorpdfstring{$d=2$.}{}}

In this case, we assume that $u\in C([0,T];Z_{2^{+},2}(\mathbb{R}^2))$ solves  \eqref{HBO-IVP} with $u_0,u(t_2)\in Z_{3,3}(\mathbb{R}^2)$ for some $t_2>0$. Additionally, we take $k=1,2$ fixed. Recalling \eqref{inteqweigh}, we have 
\begin{claim}\label{lemmasharp1} 
The following estimate hold:
\begin{equation}
F_{3,2}^{k}(t,\xi,\widehat{u_0})-\frac{i}{2}\sum_{m=1}^2\int_0^{t} F_{3,m}^{k}(t-\tau,\xi,\xi_1\widehat{u^2}) \, d\tau \in L^2(\langle \xi \rangle^{-4} d\xi) 
\end{equation} 
for all $t\in[0,T]$.
\end{claim}
Let us suppose for the moment the conclusion of Claim \ref{lemmasharp1}, thus one has
\begin{equation*}
\begin{aligned}
&\partial_{\xi_k}^3 \widehat{u}(t) \in L^2(\langle \xi \rangle^{-4} d\xi) \text{ if and only if } F_{3,1}^{k}(t,\xi,\widehat{u_0}) \in  L^2(\langle \xi \rangle^{-4} d\xi) .
\end{aligned}
\end{equation*}
Let $\phi\in C^{\infty}_c(\mathbb{R}^2)$ with $\phi \equiv 1$ when $|\xi|\leq 1$. We divide $F_{3,1}^{k}(t,\xi,\widehat{u_0})$ as 
\begin{equation*}
\begin{aligned}
F_{3,1}^{k}(t,\xi,\widehat{u_0})=&\partial_{\xi_k}^3(it \xi_1 |\xi|)(e^{it\xi_1|\xi|}-1)\widehat{u_0}(\xi)\phi +\partial_{\xi_k}^3(it \xi_1 |\xi|)(\widehat{u_0}(\xi)-\widehat{u_0}(0))\phi \\
&+\partial_{\xi_k}^3(it \xi_1 |\xi|)\widehat{u_0}(0)\phi+ \partial_{\xi_k}^3(it \xi_1 |\xi|)e^{it\xi_1|\xi|}\widehat{u_0}(\xi)(1-\phi)\\
=:& F_{3,1,1}^{k}+F_{3,1,2}^{k}+F_{3,1,3}^{k}+F_{3,1,4}^{k}.
\end{aligned}
\end{equation*}
Since $|\partial_{\xi_k}^3(it \xi_1 |\xi|)| \lesssim |\xi|^{-1}$, $\xi\neq 0$ and $t\leq T$ the mean value inequality shows that the $L^2$-norms of $F_{3,1,1}^{k}$ and $F_{3,1,4}^{k}$ are bounded by a constant (depending on $T$) times  $\left\|\widehat{u_0}\right\|_{L^{2}}$. Moreover, Sobolev's embedding gives
\begin{equation}
\begin{aligned}
\left\|F_{3,1,2}^{k}\right\|_{L^2} \lesssim \left\|\nabla \widehat{u_0}\right\|_{L^{\infty}}\left\|\phi\right\|_{L^2} \lesssim \left\|J_{\xi}^{2^{+}} \widehat{u_0}\right\|_{L^2} \lesssim \left\|\langle x \rangle^3 u_0\right\|_{L^2}.
\end{aligned} 
\end{equation}
Hence, we get
\begin{equation*}
\begin{aligned}
\partial_{\xi_k}^3 \widehat{u}(t) \in  L^2(\langle \xi \rangle^{-4} d\xi) \text{ if and only if } \partial_{\xi_k}^3(it \xi_1 |\xi|)\widehat{u_0}(0)\phi(\xi) \in  L^2(\langle \xi \rangle^{-4} d\xi) .
\end{aligned}
\end{equation*}
Considering that $u(t_2)\in Z_{3,3}(\mathbb{R}^2)$, the above implication holds at $t_2>0$. At the same time,  $|\partial_{\xi_k}^3(\xi_1 |\xi|)|^2$ is not integrable at the origin, so it must be the case that $\widehat{u_0}(0)=0$.

\begin{proof}[Proof of Claim \ref{lemmasharp1}]
In view of \eqref{eqsharp3} with $a=-2$ we find
\begin{equation}\label{eqsharp4}
\begin{aligned}
\left\|\langle \xi \rangle^{-2}F_{3,2}^{k}(t,\xi,\widehat{u_0})\right\|_{L^2} 
\lesssim & \left\|\langle \xi \rangle \widehat{u_0}\right\|_{L^2}+\left\|\partial_{\xi_k}\widehat{u_0}\right\|_{L^2}+\sum_{m=0}^3 \left\|\langle \xi \rangle^{-1}\partial_{\xi_k}^m\widehat{u_0}\right\|_{L^2}.
\end{aligned}
\end{equation}
Noticing that the r.h.s of \eqref{eqsharp4} is bounded by $ \left\|Ju_0\right\|_{L^2}+ \left\|\langle x \rangle^{3}u_0\right\|_{L^2}$, we complete the estimate for the homogeneous part of the integral equation. To control the integral term, replacing $\widehat{u_0}$ by $\widehat{u\partial_{x_1}u}$ in \eqref{eqsharp4} and using \eqref{prelimneq3}, we observe that it is enough to show
\begin{equation}\label{eqsharp4.1.1}
    u\partial_{x_1}u\in L^{\infty}([0,T];Z_{1,3}(\mathbb{R}^2)).
\end{equation}
Indeed, $u\partial_{x_1}u\in H^1(\mathbb{R}^2)$ follows from the fact that $H^{2}(\mathbb{R}^2)$ is a Banach algebra. In addition, the hypothesis $u\in Z_{2^{+},2}(\mathbb{R}^2)$ assures that there exists $\epsilon>0$ such that $u\in H^{2+\epsilon}(\mathbb{R}^2)$, as a result \eqref{prelimneq3} yields
\begin{equation}\label{eqsharp4.3}
\begin{aligned}
\left\|\langle x \rangle^{3}u^2\right\|_{L^2} \lesssim \left\|\langle x \rangle u\right\|_{L^{\infty}}\left\|\langle x \rangle^2 u\right\|_{L^2} &\lesssim  \left\|J^{1+\epsilon/2}(\langle x \rangle u)\right\|_{L^2}\left\|\langle x \rangle^2 u\right\|_{L^2} \\
 &\lesssim \left\|\langle x \rangle^2 u \right\|_{L^2}^{3/2}\left\|J^{2+\epsilon} u\right\|_{L^2}^{1/2}.
\end{aligned}
\end{equation}
This establishes \eqref{eqsharp4.1.1} and consequently the proof of Claim \ref{lemmasharp1}.
\end{proof}

\subsection{ Dimension \texorpdfstring{$d=3$.}{}}

We consider $u\in C([0,T];Z_{3,3}(\mathbb{R}^3))$ solution of \eqref{HBO-IVP} with $u_0,u(t_2)\in Z_{7/2,7/2}(\mathbb{R}^2)$ for some $t_2>0$. Our arguments require localizing near the origin in Fourier frequencies by a function $\phi \in C^{\infty}_{c}(\mathbb{R}^3)$ with $\phi(\xi)=1$ if $|\xi|\leq 1$. Thus, recalling \eqref{inteqweigh} we have:
\begin{claim}\label{lemmasharp2}
Let $k=1,2,3$. Then
\begin{equation}\label{eqsharp11.1}
F_{3,2}^{k}(t,\xi,\widehat{u_0})\phi(\xi)-\frac{i}{2}\sum_{m=1}^2\int_0^{t} F_{3,m}^{k}(t-\tau,\xi,\xi_1\widehat{u^2})\phi(\xi) \, d\tau \in H^{1/2}_{\xi}(\mathbb{R}^3) 
\end{equation} 
for all $t\in[0,T]$.
\end{claim}
Let us suppose for the moment that Claim \ref{lemmasharp2} holds, then
\begin{equation*}
\begin{aligned}
&\partial_{\xi_k}^3 \widehat{u}(t)\phi  \in H^{1/2}_{\xi}(\mathbb{R}^3)\text{ if and only if } F_{3,1}^{k}(t,\xi,\widehat{u_0})\phi \in H^{1/2}_{\xi}(\mathbb{R}^3).
\end{aligned}
\end{equation*}
We split $F_{3,1}^{k}$ as 
\begin{equation*}
\begin{aligned}
    F_{3,1}^{k}(t,\xi,\widehat{u_0})\phi=&\partial_{\xi_k}^3(it \xi_1 |\xi|)(e^{it\xi_1|\xi|}-1)\widehat{u_0}(\xi)\phi +\partial_{\xi_k}^3(it \xi_1 |\xi|)(\widehat{u_0}(\xi)-\widehat{u_0}(0))\phi+\partial_{\xi_k}^3(it \xi_1 |\xi|)\widehat{u_0}(0)\phi \\
=:& F_{3,1,1}^{k}+F_{3,1,2}^{k}+F_{3,1,3}^{k}.
\end{aligned}
\end{equation*}
The mean value inequality reveals
\begin{equation}\label{eqsharp11.2}
\begin{aligned}
\left\|F_{3,1,1}^{k}\right\|_{H^{1/2}_{\xi}}\leq \left\|F_{3,1,1}^{k}\right\|_{H^{1}_{\xi}}
 &\lesssim \left\|\langle \xi \rangle\widehat{u_0}\phi \right\|_{L^2}+\left\||\xi|\nabla_{\xi}\big(\widehat{u_0}\phi \big)\right\|_{L^2} \lesssim \left\|\widehat{u_0}\right\|_{H^1_{\xi}} \lesssim \left\|\langle x \rangle u_0 \right\|_{L^2},
\end{aligned}
\end{equation}
and from Sobolev's embedding and the fact that $|\cdot|^{-1}\phi \in L^2(\mathbb{R}^3)$ one gets
\begin{equation}\label{eqsharp11.3}
\begin{aligned}
\left\|F_{3,1,2}^{k}\right\|_{H^{1/2}_{\xi}}\lesssim \left\|F_{3,1,2}^{k}\right\|_{H^{1}_{\xi}}
 &\lesssim \big(\left\|\phi\right\|_{H^1}+\left\||\cdot|^{-1}\phi \right\|_{L^2}\big)\left\|\nabla \widehat{u_0}\right\|_{L^{\infty}} \lesssim \left\|\langle x \rangle^{3}u_0\right\|_{L^2}.
\end{aligned}
\end{equation}
Hence, 
\begin{equation}\label{eqsharp14.2}
\begin{aligned}
&\partial_{\xi_k}^3 \widehat{u}(t)\phi \in  H^{1/2}_{\xi}(\mathbb{R}^3) \text{ if and only if } \partial_{\xi_k}^3(it \xi_1 |\xi|)\widehat{u_0}(0)\phi  \in  H^{1/2}_{\xi}(\mathbb{R}^3).
\end{aligned}
\end{equation}
Letting $k=1$ in \eqref{eqsharp14.2}, we claim
\begin{equation}\label{eqsharp14.3}
\mathcal{D}_{\xi}^{1/2}\big(\partial_{\xi_1}^3(\xi_1 |\xi|)\phi\big) \notin L^{2}(\mathbb{R}^3).
\end{equation}
Consequently, since \eqref{eqsharp14.2} holds for $t=t_2>0$, \eqref{eqsharp14.3} imposes that $\widehat{u_0}(0)=0$. We now turn to the proof of \eqref{eqsharp14.3}. For a given $x=(x_1,x_2,x_3) \in \mathbb{R}^3$, we denote by $\tilde{x}=(x_2,x_3)\in \mathbb{R}^2$. Let
\begin{equation*}
F(\xi):=\partial_{\xi_1}^3(\xi_1 |\xi|)=3(\xi_2^2+\xi_3^2)^2/|\xi|^5=3|\tilde{\xi}|^4/|\xi|^5
\end{equation*}
and the region 
\begin{equation*}
\mathcal{P}:=\left\{x\in \mathbb{R}^3 : |x|\leq 2^{1/4}|\tilde{x}|, \hspace{0.2cm} \, |x| \leq 1/16 \right\}.
\end{equation*}
When $\xi \in \mathcal{P}$ and $4|\xi|\leq |\eta|\leq 1/2$, one has $|\xi-\eta|\geq 3|\xi|$ and $|\tilde{\xi}|^4\geq |\xi|^4/2$, from these deductions,
\begin{equation*}
\begin{aligned}
\left|F(\xi)-F(\xi-\eta)\right|&=\frac{3}{|\xi|^5|\xi-\eta|^5}\left||\xi-\eta|^5|\tilde{\xi}|^4-|\xi|^5|\tilde{\xi}-\tilde{\eta}|^4\right| \\
& \geq \frac{3}{|\xi|^5|\xi-\eta|^5}\big(|\xi-\eta|^5|\tilde{\xi}|^4-|\xi|^4|\xi-\eta|^5/3\big) \gtrsim |\xi|^{-1}.
\end{aligned}
\end{equation*}
Hence,
\begin{equation}\label{eqsharp14.0}
\begin{aligned}
\big(\mathcal{D}_{\xi}^{1/2}(\partial_{\xi_1}^3(\xi_1 |\xi|)\phi)\big)^2(\xi) \chi_{\mathcal{P}}(\xi) &\geq \int_{4|\xi|\leq |\eta|\leq 1/2} \frac{\left|F(\xi)-F(\xi-\eta)\right|^2}{|\eta|^{4}}\, d\eta \, \chi_{\mathcal{P}}(\xi)\\
&\gtrsim  \frac{1}{|\xi|^2}\int_{4|\xi|\leq |\eta|\leq 1/2} \frac{1}{|\eta|^{4}}\, d\eta \chi_{\mathcal{P}}(\xi)\gtrsim \frac{1}{|\xi|^3} \chi_{\mathcal{P}}(\xi),
\end{aligned}
\end{equation}
where $\chi_{\mathcal{P}}$ stands for the indicator function on the set $\mathcal{P}$. Therefore, given that $|\xi|^{-3/2} \chi_{\mathcal{P}}\notin L^{2}(\mathbb{R}^3)$,  we get $\mathcal{D}_{\xi}^{1/2}(\partial_{\xi_1}^3(\xi_1 |\xi|)\phi) \notin L^{2}(\mathbb{R}^3)$.

\begin{proof}[Proof of Claim \ref{lemmasharp2}] 
Letting $\tilde{\phi}\in C^{\infty}_c(\mathbb{R}^3)$ with $\tilde{\phi}\phi=\phi$, Proposition \ref{prelimneq0.2} yields
\begin{equation}\label{eqsharp14.1}
\begin{aligned}
\left\|\mathcal{D}_{\xi}^{1/2}(F_j^{k}(t,\xi,f)\phi)\right\|_{L^2} &\lesssim \left\|\mathcal{D}_{\xi}^{1/2}(e^{it\xi_1|\xi|})\widetilde{F}_j^{k}(t,\xi,f)\phi\right\|_{L^2}+\left\|\mathcal{D}_{\xi}^{1/2}(\widetilde{F}_j^{k}(t,\xi,f)\phi)\right\|_{L^2} \\
&\lesssim \left\|\mathcal{D}_{\xi}^{1/2}(e^{it\xi_1|\xi|}) \tilde{\phi}\right\|_{L^{\infty}}\left\|\widetilde{F}_j^{k}(t,\xi,f)\phi \right\|_{L^2}+\left\|\mathcal{D}_{\xi}^{1/2}(\widetilde{F}_j^{k}(t,\xi,f)\phi)\right\|_{L^2}  \\
& \lesssim \left\|\widetilde{F}_j^{k}(t,\xi,f)\phi\right\|_{H^{1/2}_{\xi}}.
\end{aligned}
\end{equation}
Analogously, we bound the $H^{1/2}_{\xi}$-norm of $F_{3,2}^{k}(t,\xi,f)\phi$ by that of  $\widetilde{F}_{3,2}^{k}(t,\xi,f)\phi$. Consequently the above computation reduces our arguments to bound \eqref{eqsharp11.1} for the operators $\widetilde{F}_{3,m}$.
Letting $f=\widehat{u_0}$ and $b=1/2$ in \eqref{eqsharp3.4}, repeated applications of Proposition \ref{lemmasharp3} show
\begin{equation} 
\begin{aligned}
\left\|\widetilde{F}_{3,2}^k(t,\xi,\widehat{u_0})\phi\right\|_{H^{1/2}_{\xi}} 
\lesssim & \sum_{m=0}^{3}\left\|\partial_{\xi_k}^m\widehat{u_0}\phi\right\|_{H^{1/2}_{\xi}}+\left\|\widehat{u_0}\right\|_{H^{(1/2)^{+}}}+\left\|\partial_{\xi_k}\widehat{u_0}\right\|_{H^{(1/2)^{+}}}  \lesssim  \left\|\langle x \rangle^{3+1/2} u_0 \right\|_{L^2}.
\end{aligned}
\end{equation}
On the other hand, employing \eqref{eqsharp3.4} with $f=\xi_1\widehat{u^2}$ and $b=1/2$, it is deduced
\begin{equation}\label{eqsharp14.1.1}
\begin{aligned}
\left\|\widetilde{F}_{3,2}^k(t-\tau,\xi,\xi_1\widehat{u^2})\phi\right\|_{H^{1/2}_{\xi}} &\lesssim \sum_{m=0}^{3}\left\|\partial_{\xi_k}^m(\xi_1\widehat{u^2})\phi\right\|_{H^{1}} \lesssim \left\|\langle x \rangle^{4}u^2 \right\|_{L^2}.
\end{aligned}
\end{equation}
This expression is controlled since
$$u^2\in L^{\infty}([0,T];L^2(|x|^{8}\, dx)),$$
which holds arguing as in \eqref{eqsharp4.3} employing complex interpolation \eqref{prelimneq3}. Finally, one can follow the ideas around \eqref{eqsharp11.3} to bound $\left\|\widetilde{F}_{3,1}^k(t-\tau,\xi,\xi_1\widehat{u^2})\phi\right\|_{H^{1}_{\xi}}$ by the r.h.s of \eqref{eqsharp14.1.1}. The proof is now completed.
\end{proof}


\section{Proof of Theorem \ref{threetimesharp}}


We first discuss the main ideas leading to the proof of Theorem \ref{threetimesharp}. By hypothesis, there exist three different times $t_1,t_2$ and $t_3$ such that
\begin{equation}\label{eqthreshar2}
u(\cdot, t_j)\in Z_{d/2+3,d/2+3}(\mathbb{R}^d), \hspace{0.5cm} j=1,2,3,
\end{equation}
The equation in \eqref{HBO-IVP} yields the following identities,
\begin{equation}\label{eqthreshar3}
\frac{d}{dt}\int x_l u(x,t)\, dx=\frac{\delta_{1,l}}{2}\left\|u(t)\right\|_{L^2}^2=\frac{\delta_{1,l}}{2}\left\|u_0\right\|_{L^2}^2, \hspace{0.3cm} l=1,\dots, d 
\end{equation}
and hence
\begin{equation}\label{eqthreshar3.2}
\int x_l u(x,t)\, dx=\int x_l u_0(x)\, dx+\frac{\delta_{1,l}}{2}\left\|u_0\right\|_{L^2}^2, \hspace{0.3cm} l=1,\dots, d.
\end{equation}
If we prove that there exist $\tilde{t}_1\in (t_1,t_2)$ and $\tilde{t}_2\in (t_2,t_3)$ such that 
\begin{equation*}
\begin{aligned}
\int x_1 u(x,\tilde{t}_{j})\, dx=0, \text{ for all } j=1,2,
\end{aligned}
\end{equation*}
in view of \eqref{eqthreshar3} with $l=1$, it follows that $u\equiv 0$. In this manner, assuming \eqref{eqthreshar2}, we just need to show  that there exists $\tilde{t}_1\in (t_1,t_2)$ such that
\begin{equation*}
\begin{aligned}
\int x_1 u(x,\tilde{t}_1)\, dx=0.
\end{aligned}
\end{equation*}
Without loss of generality, we let $t_1=0<t_2<t_3$, that is, $$u_0, u(t_j)\in Z_{d/2+3,d/2+3}(\mathbb{R}^d), \hspace{0.5cm} j=2,3.$$  

Next, we introduce some further notation and estimates to be used in the proof of Theorem \ref{threetimesharp}. For a given $k=1,\dots, d$, recalling \eqref{eqsharp1.1}, we split $F_4^k$ as
\begin{equation}\label{eqthreshar1}
F_4^k(t,\xi,f)=F_{4,1}^k(t,\xi,f)+F_{4,2}^k(t,\xi,f),
\end{equation}
where
\begin{equation*}
F_{4,1}^k(t,\xi,f)= \partial_{\xi_k}^4(it\xi_1|\xi|)e^{it\xi_1|\xi|}f(\xi)+4\partial_{\xi_k}^3(it\xi_1|\xi|)e^{it\xi_1|\xi|}\partial_{\xi_k}f(\xi).
\end{equation*}
In addition, we set
\begin{equation}\label{eqsharp3.5}
\begin{aligned}
\widetilde{F}_{4,l}^k(t,\xi,f)=e^{-it\xi_1|\xi|}F_{4,l}^k(t,\xi,f), \hspace{0.5cm} l=1,2. 
\end{aligned}
\end{equation}
We require to estimate the following differential equation obtained from \eqref{inteequ},
\begin{equation}
\partial_{\xi_k}^4 \widehat{u}(t)=\sum_{m=1}^2F_{4,m}^{k}(t,\xi,\widehat{u_0})-\frac{i}{2}\int_0^{t} F_{4,m}^{k}(t-\tau,\xi,\xi_1\widehat{u^2}) \, d\tau,
\end{equation}
for each $k=1,\dots,d$. Now, we proceed to bound the terms $F^k_j$. To localize in frequency, taking $g \in L^{\infty}(\mathbb{R}^d)$, \eqref{eqsharp1.1} gives
\begin{equation}\label{eqsharp4.1}
\begin{aligned}
\left\|\langle \xi \rangle^{a}\widetilde{F}_{4,2}^k(t,\xi,f)g\right\|_{H^b_{\xi}}  \lesssim &\left\|\langle \xi \rangle^{a}\partial_{\xi_k}^3(it\xi_1|\xi|)\partial_{\xi_k}(it\xi_1|\xi|)f g\right\|_{H^{b}_{\xi}}+\sum_{m=0}^4\sum_{j=0}^{4-m}\left\|\langle \xi \rangle^{a}(\partial_{\xi_k}(it\xi_1|\xi|))^j\partial_{\xi_k}^mf g\right\|_{H^{b}_{\xi}} \\
&+\sum_{m=0}^2 \big(\sum_{j=0}^{2-m}\left\|\langle \xi \rangle^{a}\partial_{\xi_k}^2(it\xi_1|\xi|)(\partial_{\xi_k}(it\xi_1|\xi|))^j\partial_{\xi_k}^mf g \right\|_{H^{b}_{\xi}} \\ &\hspace{0.5cm}+\left\|\langle \xi \rangle^{a}(\partial_{\xi_k}^2(it\xi_1|\xi|))^j\partial_{\xi_k}^mf g\right\|_{H^{b}_{\xi}}\big).
\end{aligned}
\end{equation}
In particular, setting $b=0$, $g=1$ and using that $|\partial_{\xi_k}^{l}(it\xi_1|\xi|)|\lesssim |\xi |^{2-l}$, $l=1,2$ and $|\partial_{\xi_k}^{3}(it\xi_1|\xi|)|\lesssim |\xi|^{-1}$, we have
\begin{equation}\label{eqthreshar1.1}
\begin{aligned}
&\left\|\langle \xi \rangle^a F_{4,2}^k(t,\xi,f)\right\|_{L^2}=\left\|\langle \xi \rangle^a \widetilde{F}_{4,2}^k(t,\xi,f)\right\|_{L^2} \lesssim \sum_{m=0}^4 \sum_{j=0}^{4-m}\left\|\langle \xi \rangle^{a+j}\partial_{\xi_k}^{m} f\right\|_{L^2}.
\end{aligned}
\end{equation}
Additionally, when $f=\widehat{u_0}$, we define the operators
\begin{equation*}
    \begin{aligned}
     F_{4,1,1}^k(t,\xi,\widehat{u_0}(\xi))&=\partial_{\xi_k}^4(it\xi_1|\xi|)(e^{it\xi_1|\xi|}-1)\widehat{u_0}(\xi), \\
      F_{4,1,2}^k(t,\xi,\widehat{u_0}(\xi))&=\sum_{|\beta|=2}\partial_{\xi_k}^4(it\xi_1|\xi|)R_{\beta}(\widehat{u_0},\xi)\xi^{\beta}\phi(\xi),\\
      F_{4,1,3}^k(t,\xi,\widehat{u_0}(\xi))&=\partial_{\xi_k}^4(it\xi_1|\xi|)\widehat{u_0}(\xi)(1-\phi(\xi)),\\
F_{4,1,4}^k(t,\xi,\widehat{u_0}(\xi))&=4\partial_{\xi_k}^3(it\xi_1|\xi|)(e^{it\xi_1|\xi|}-1)\partial_{\xi_k}\widehat{u_0}(\xi), \\
F_{4,1,5}^k(t,\xi,\widehat{u_0}(\xi))&=4\partial_{\xi_k}^3(it\xi_1|\xi|)(\partial_{\xi_k}\widehat{u_0}(\xi)-\partial_{\xi_k}\widehat{u_0}(0)) \phi(\xi), \\
F_{4,1,6}^k(t,\xi,\widehat{u_0}(\xi))&= 4\partial_{\xi_k}^3(it\xi_1|\xi|)\partial_{\xi_k}\widehat{u_0}(\xi)(1-\phi(\xi)), \\
    \end{aligned}
\end{equation*}
where $\phi\in C_c^{\infty}(\mathbb{R}^d)$ is radial such that $\phi=1$ when $|\xi|\leq 1$ and
\begin{equation*}
R_{\beta}(\widehat{u_0},\xi)=\frac{|\beta|}{\beta!}\int_0^1(1-\nu)^{|\beta|-1}\partial^{\beta}\widehat{u_0}(\nu\xi)\, d\nu.
\end{equation*}
Consequently, when $\widehat{u}(0)=\widehat{u_0}(0)=0$, it holds
\begin{equation}\label{eqthreshar2.0}
\begin{aligned}
F_{4,1}^k(t,\xi,\widehat{u_0}(\xi))=\sum_{j=1}^6 F_{4,1,j}^k(t,\xi,\widehat{u_0}(\xi))+\partial_{\xi_k}^4(it\xi_1|\xi|)\nabla \widehat{u_0}(0)\cdot \xi \phi+4\partial_{\xi_k}^3(it\xi_1|\xi|)\partial_{\xi_k}\widehat{u_0}(0)\phi(\xi).
\end{aligned}
\end{equation}
Notice that \eqref{eqthreshar2.0} is still valid replacing $\widehat{u_0}$ by $\xi_1\widehat{u^2}$. We are now in position to prove Theorem \ref{threetimesharp}. We divide our arguments according to the dimension.

\subsection{ Dimension \texorpdfstring{$d=2$.}{}}

Suppose that $u\in C([0,T];\dot{Z}_{3,3}(\mathbb{R}^2))$ with $\widehat{u_0}, u(t_2)\in Z_{4,4}(\mathbb{R}^2)$. Under these considerations we have:
\begin{claim}\label{claim3}
We find the following estimate to hold:
\begin{equation}\label{eqthreshar1.2}
\sum_{j=1}^{6} F_{4,1,j}^{k}(t,\xi,\widehat{u_0})-\frac{i}{2}\int_0^t F_{4,1,j}^{k}(t-\tau,\xi,\xi_1\widehat{u^2})\, d \tau \in L^{2}(\mathbb{R}^2)
\end{equation}
and
\begin{equation}\label{eqthreshar1.3}
F_{4,2}^{k}(t,\xi,\widehat{u_0})-\frac{i}{2}\int_0^{t} F_{4,2}^{k}(t-\tau,\xi,\xi_1\widehat{u^2}) \, d\tau \in L^2(\langle \xi \rangle^{-8} d\xi) 
\end{equation} 
for all $t\in[0,T]$.
\end{claim}

\begin{proof}
We first prove \eqref{eqthreshar1.2}. The mean value inequality shows that $F_{4,1,j}^k(t,\xi,\widehat{u_0}(\xi))$ is bounded by the $L^2$-norm of $u_0$ for all $j\neq 2,5$.  We use Sobolev's embedding to find 
\begin{equation} \label{eqthreshar1.5}
\begin{aligned}
\left\|F_{4,1,2}^k(t,\xi,\widehat{u_0})\right\|_{L^2}\lesssim \sum_{|\beta|=2}\left\|\int_0^1(1-\nu)\partial^{\beta}\widehat{u_0}(\nu \xi)\phi\, d\nu\right\|_{L^2} 
&\lesssim \sum_{|\beta|=2}\int_0^1(1-\nu)\left\|\partial^{\beta}\widehat{u_0}(\nu \xi)\right\|_{L^4}\left\|\phi\right\|_{L^4}\, d\nu \\
&\lesssim \sum_{|\beta|=2}(\int_0^1(1-\nu)\nu^{-1/2}\, d\nu)\left\|\partial^{\beta}\widehat{u_0}\right\|_{L^4}\left\|\phi\right\|_{L^4} \\
&\lesssim \sum_{|\beta|=2} \left\|D^{1/2}\partial^{\beta}\widehat{u_0}\right\|_{L^2} \lesssim \left\|\langle x \rangle^{2+1/2}u_0\right\|_{L^2}.
\end{aligned}
\end{equation}
This argument provides the same bound for $F_{4,1,5}^k$, since one can write
\begin{equation*}
F_{4,1,5}^k(t,\xi,\widehat{u_0}(\xi))=4\partial_{\xi_k}^3(it\xi_1|\xi|)\int_0^{1} \nabla\partial_{\xi_k}\widehat{u_0}(\nu \xi)\cdot \xi \, \phi \, d\nu.
\end{equation*}
On the other hand, given that $u\in C([0,T];\dot{Z}_{3,3}(\mathbb{R}^2))$, it is possible to argue as in the deduction of \eqref{eqsharp4.3} to find
\begin{equation}\label{eqthreshar1.4}
 u\partial_{x_1}u\in L^{\infty}([0,T];L^{2}(|x|^{5}\, dx))\hspace{0.3cm} \text{ and } \hspace{0.3cm} u^2\in L^{\infty}([0,T];L^{2}(|x|^{8}\, dx)) . 
\end{equation}
 Thus, replacing $\widehat{u_0}$ by $\xi_1 \widehat{u^2}$ in the preceding discussions and employing \eqref{eqthreshar1.4}, we conclude \eqref{eqthreshar1.2}.
\\ \\
Next we deduce \eqref{eqthreshar1.3}. To estimate the homogeneous part,  we employ \eqref{eqthreshar1.1} with $a=-4$ and $f=\widehat{u_0}$ to deduce
\begin{equation}\label{eqthreshar1.4.1}
\begin{aligned}
\left\|\langle \xi \rangle^{-4}F_{4,2}^k(t,\xi,\widehat{u_0}(\xi))\right\|_{L^2} \lesssim \left\|\widehat{u_0}\right\|_{L^2}+\sum_{m=0}^4 \left\|\langle \xi \rangle^{-1}\partial_{\xi_k}^m\widehat{u_0}\right\|_{L^2},
\end{aligned}
\end{equation}
and so the above inequality is controlled after Plancherel's theorem  by $\left\|\langle x \rangle^{4}u_0\right\|_{L^2}$.
Finally, replacing $\widehat{u_0}$ by $\xi_1 \widehat{u^2}$ in \eqref{eqthreshar1.4.1}, one can control the resulting expression by \eqref{eqthreshar1.4} and the fact $u\partial_{x_1}u\in H^3(\mathbb{R}^3)$. This completes the deduction of \eqref{eqthreshar1.3}.


\end{proof}
Summing up we get
\begin{equation}\label{eqthreshar2.1}
\begin{aligned}
\partial_{\xi_k}^4 \widehat{u}(t)& \in L^2(\langle \xi \rangle^{-8} d\xi), \hspace{0.2cm} \text{ if and only if }   \\
&t\partial_{\xi_k}^4(\xi_1|\xi|)\nabla \widehat{u_0}(0)\cdot \xi  \phi(\xi) -\frac{i}{2}\int_0^{t}(t-\tau)\partial_{\xi_k}^4(\xi_1|\xi|)\nabla\big( \xi_1  \widehat{u^2}\big)(0,\tau)\cdot \xi\phi(\xi)\, d\tau\\
&+4t\partial_{\xi_k}^3(\xi_1|\xi|)\partial_{\xi_k}\widehat{u_0}(0)\phi(\xi)-4\frac{i}{2}\int_0^{t}(t-\tau)\partial_{\xi_k}^3(\xi_1|\xi|)\partial_{\xi_k} \big(\xi_1 \widehat{u^2}\big)(0,\tau)\phi(\xi)\, d\tau \\
& \hspace{5cm}\in  L^2(\langle \xi \rangle^{-8} d\xi),
\end{aligned}
\end{equation}
for fixed $t \geq 0$. Let us denote by
\begin{equation}\label{eqthreshar3.0}
\begin{aligned}
\mathcal{C}_l(t):=t\partial_{\xi_l}\widehat{u_0}(0) -\frac{i}{2}\int_0^{t}(t-\tau)\partial_{\xi_l}\big( \xi_1  \widehat{u^2}\big)(0,\tau)\, d\tau, \hspace{0.5cm} l=1,2.
\end{aligned}
\end{equation}
The hypothesis at $t=t_2$, the fact that $\langle \xi \rangle  \sim 1$ on the support of $\phi$ and \eqref{eqthreshar2.1} imply
\begin{equation}\label{eqthreshar3.1}
\begin{aligned}
&\sum_{l=1}^2\mathcal{C}_l(t_2)\partial_{\xi_k}^4(\xi_1|\xi|)\xi_l\phi(\xi)+4\mathcal{C}_k(t_2)\partial_{\xi_k}^3(\xi_1|\xi|)\phi(\xi) \in  L^2(\mathbb{R}^2).
\end{aligned}
\end{equation}
From this, we claim that
\begin{equation}\label{eqthreshar3.3}
\mathcal{C}_1(t_2)=\mathcal{C}_2(t_2)=0.
\end{equation}
Let us first write $\mathcal{C}_1(t)$ in a more convenient way for our arguments. We have
\begin{equation*}
\begin{aligned}
\partial_{\xi_l}\widehat{u_0}(0)=-\widehat{ix_lu_0}(0) =-i\int x_lu_0(x) dx
\end{aligned}
\end{equation*}
and by \eqref{eqthreshar3},
\begin{equation}\label{eqthreedt4.2}
\begin{aligned}
\partial_{\xi_l}\big(i\xi_1/2\widehat{u^2}\big)(0,\tau)=\widehat{-ix_lu\partial_{x_1}u}(0,\tau)&=-i\int x_lu\partial_{x_1}u(x,\tau)\, dx \\ &=i\frac{\delta_{1,l}}{2} \left\|u(\tau)\right\|_{L^2}^2 
=i\delta_{1,l}\frac{d}{dt}\int x_l u(x,t)\, dx.
\end{aligned}
\end{equation}
Integration by parts then gives 
\begin{equation}\label{eqthreshar4.1}
\begin{aligned}
\mathcal{C}_l(t)=&t\partial_{\xi_l}\widehat{u_0}(0) -\frac{i}{2}\int_0^{t}(t-\tau)\partial_{\xi_l} \big(\xi_1 \widehat{u^2}\big)(0,\tau) \, d\tau \\
=&-it\int x_l u_0(x)\, dx -i\delta_{1,l}\int_0^{t}(t-\tau)\frac{d}{d\tau}\big(\int x_l u(x,\tau)\, dx\big)d\tau \\
=&-it(1-\delta_{1,l})\int x_l u_0(x)\, dx-i\delta_{1,l}\int_0^{t}\int x_l u(x,\tau)\, dx \, d\tau.
\end{aligned}
\end{equation}
\\ 
Let us suppose for the moment that \eqref{eqthreshar3.3} holds, 
as a result the equation \eqref{eqthreshar4.1} shows
\begin{equation*}
\begin{aligned}
0=\mathcal{C}_1(t_2)=-i\int_0^{t_2}\int x_1 u(x,\tau)\, dx \, d\tau.
\end{aligned}
\end{equation*}
In this manner, the continuity of the application $\tau\mapsto \int  x_1 u(x,\tau)\, dx$ assures that there exists a time $\tilde{t}_1\in (0,t_2)$ at which this map vanishes. According to our reasoning at the beginning of this section, this concludes the proof of Theorem \ref{threetimesharp} when $d=2$.
\\ \\
We can now return to deduce \eqref{eqthreshar3.3}. We set
$$G(\xi):=\sum_{l=1}^2 i\mathcal{C}_l(t_2)\partial_{\xi_k}^4(\xi_1|\xi|)\xi_l+4i\mathcal{C}_k(t_2)\partial_{\xi_k}^3(\xi_1|\xi|).$$
Given that $G(\nu\xi)=\nu^{-1}G(\xi),$ $\xi\neq 0$, $\nu>0$, changing to polar coordinates and recalling that $\phi$ is radial, we find
\begin{equation}\label{eqthreshar4}
\left\|G(\xi)\phi(\xi)\right\|_{L^2}^2 \sim \big(\int_{\mathbb{S}^1}|G(x)|^2 \, dS(x)\big)\int_{0}^{\infty} |\nu|^{-1}|\phi(\nu)|^2\, d\nu.
\end{equation}
Since $|v|^{-1}\phi(\nu)$ is not integrable,  \eqref{eqthreshar3.1} implies that $G\equiv 0$. However, the functions $\partial_{\xi_k}^4(\xi_1|\xi|)\xi_1$, $\partial_{\xi_k}^4(\xi_1|\xi|)\xi_2$ and $\partial_{\xi_k}^3(\xi_1|\xi|)$ are linear independent (on $\mathbb{R}$), so it must be the case that $\mathcal{C}_1(t_2)=\mathcal{C}_2(t_2)=0$, which is \eqref{eqthreshar3.3}.  

\subsection{ Dimension \texorpdfstring{$d=3$.}{}}

Here we assume that $u\in C([0,T];\dot{Z}_{4,4}(\mathbb{R}^3))$ with $u_0, u(t_2) \in Z_{9/2,9/2}(\mathbb{R}^{3})$. Recalling the notation \eqref{eqthreshar2.0}, we state:
\begin{claim}
 One has:
\begin{equation}\label{eqthreedt1.0}
\sum_{j=1}^6F_{4,1,j}^{k}(t,\xi,\widehat{u_0})-\frac{i}{2}\int_0^{t} F_{4,1,j}^{k}(t-\tau,\xi,\xi_1\widehat{u^2})\phi(\xi) \, d\tau \in H^{1}_{\xi}(\mathbb{R}^{3}).
\end{equation}
and
\begin{equation}\label{eqthreedt0.6}
\langle \xi \rangle^{-2}F_{4,2}^{k}(t,\xi,\widehat{u_0})-\frac{i}{2}\int_0^{t} \langle \xi \rangle^{-2}F_{4,2}^{k}(t-\tau,\xi,\xi_1\widehat{u^2}) \, d\tau \in H^{1/2}_{\xi}(\mathbb{R}^3).
\end{equation} 
for all $t\in[0,T]$.
\end{claim}

\begin{proof}
We first establish \eqref{eqthreedt1.0}. The mean value inequality, the fact that $|\xi|^{-1}\in L^{2}_{loc}(\mathbb{R}^3)$ and a similar reasoning to \eqref{eqsharp11.2} and \eqref{eqsharp11.3} establish 
\begin{equation}
\begin{aligned}
\left\| F_{4,1,j}^k(t,\xi,\widehat{u_0}(\xi))\right\|_{H_{\xi}^{1}} 
&\lesssim  \left\|\langle x \rangle^{2} u_0\right\|_{L^2}+\left\| u_0\right\|_{H^2}
\end{aligned}
\end{equation}
for all $j=1,3,4,6$.
 An analogous argument to \eqref{eqthreshar1.5}, changing variables and using Sobolev's embedding provides
\begin{equation}
\begin{aligned}
\left\|F_{4,1,2}^k(t,\xi,\widehat{u_0}(\xi))\right\|_{H^{1}}  
&\lesssim \sum_{|\beta|=2} \big(\left\||\cdot|^{-1}\phi\right\|_{L^2}+\left\|\phi\right\|_{H^{1}}\big)\left\|R_{\beta}(\widehat{u_0},\xi)\right\|_{L^{\infty}}+ \left\|\nabla R_{\beta}(\widehat{u_0},\xi)\phi\right\|_{L^2}\\
&\lesssim \sum_{|\beta|=2} \left\|\partial^{\beta}\widehat{u_0}\right\|_{L^{\infty}}+ \sum_{|\beta|=2}(\int_0^1(1-\nu)\, d\nu)\left\|\nabla\partial^{\beta}\widehat{u_0}\right\|_{L^3}\left\|\phi\right\|_{L^6} \\
&\lesssim \left\|\langle x \rangle^{4} u_0\right\|_{L^{2}}+\sum_{|\beta|=2} \left\|D^{1/2}\nabla\partial^{\beta}\widehat{u_0}\right\|_{L^2} \lesssim \left\|\langle x \rangle^{4}u_0\right\|_{L^2}.
\end{aligned}
\end{equation}
The estimate $F_{4,1,5}^k(t,\xi,\widehat{u_0}(\xi))$ is obtained in a similar fashion to $F_{4,1,2}^k(t,\xi,\widehat{u_0}(\xi))$. This concludes the considerations for the homogeneous part in \eqref{eqthreedt1.0}. On the other hand, given that $u\in C([0,T];\dot{Z}_{4,4}(\mathbb{R}^3))$, by a similar reasoning to \eqref{eqsharp4.3} one has
\begin{equation}\label{eqthreedt0.7}
 u\partial_{x_1}u\in  L^{\infty}([0,T];\dot{Z}_{3,9/2}(\mathbb{R}^3)).
\end{equation}
This enables us to change the roles of $\widehat{u_0}$ by $\xi_1 \widehat{u^2}$ in the above estimates to conclude \eqref{eqthreedt1.0}.
\\ \\
Let us now establish \eqref{eqthreedt0.6}. The inequality \eqref{prelimneq0.1} and Proposition \ref{prelimneq0.2} imply
\begin{equation}\label{eqthreedt0.4} 
\begin{aligned}
\left\|\langle \xi \rangle^{-2}F_{4,2}^k(t,\xi,\widehat{u_0})\right\|_{H^{1/2}_{\xi}} 
\lesssim & \left\|\langle \xi \rangle^{-3/2}\widetilde{F}_{4,2}^k(t,\xi,\widehat{u_0})\right\|_{L^2}+\left\|\langle \xi \rangle^{-2}\widetilde{F}_{4,2}^k(t,\xi,\widehat{u_0})\phi\right\|_{H^{1/2}_{\xi}} \\
&+\left\|\langle \xi \rangle^{-2}\widetilde{F}_{4,2}^k(t,\xi,\widehat{u_0})(1-\phi)\right\|_{H^{1/2}_{\xi}}.
\end{aligned}
\end{equation}
We proceed then to estimate each term on the r.h.s of \eqref{eqthreedt0.4}. From \eqref{eqthreshar1.1} with $a=-3/2$ we find 
\begin{equation}\label{eqthreedt0.4.1} 
\begin{aligned}
\left\|\langle \xi \rangle^{-3/2}\widetilde{F}_{4,2}^k(t,\xi,\widehat{u_0})\right\|_{L^2} 
& \lesssim  \sum_{m=0}^4\sum_{j=0}^{4-m} \left\|\langle \xi \rangle^{-3/2+j}\partial_{\xi_k}^m\widehat{u_0}\right\|_{L^{2}}\\
& \lesssim  \sum_{m=0}^2\sum_{j=0}^{1} (\cdots)+ \sum_{m=3}^4\sum_{j=0}^{4-m} (\cdots)+ \sum_{m=1}^2\sum_{j=2}^{4-m} (\cdots)+ \sum_{j=2}^{4} \left\|\langle \xi \rangle^{-3/2+j} \widehat{u_0}\right\|_{L^{2}}\\
& \lesssim \left\|J^{4}_{\xi}\widehat{u_0}\right\|_{L^{2}}+ \sum_{m=1}^2\sum_{j=2}^{4-m} \left\|\langle \xi \rangle^{-3/2+j}\partial_{\xi_k}^m\widehat{u_0}\right\|_{L^{2}}+\left\|\langle \xi \rangle^{5/2}\widehat{u_0}\right\|_{L^{2}}.
\end{aligned}
\end{equation}
In view of the inequality $\left\|\langle \xi \rangle^{-3/2+j}\partial_{\xi_k}^m\widehat{u_0}\right\|_{L^{2}}  \lesssim \left\|\partial_{\xi_k}^{m}\big( \langle \xi \rangle^{j-3/2}\widehat{u_0}\big)\right\|_{L^{2}}+\left\|[\langle \xi \rangle^{j-3/2},\partial_{\xi_k}^m]\widehat{u_0}\right\|_{L^{2}}$ and complex interpolation,
\begin{equation}
\begin{aligned}
\sum_{m=1}^2\sum_{j=2}^{4-m} \left\|\langle \xi \rangle^{-3/2+j}\partial_{\xi_k}^m\widehat{u_0}\right\|_{L^{2}} 
& \lesssim \sum_{m=1}^2\sum_{j=2}^{4-m}  \left\|J_{\xi}^{m}\big( \langle \xi \rangle^{j-3/2}\widehat{u_0}\big)\right\|_{L^{2}} +\left\|\langle \xi \rangle^{5/2}\widehat{u_0}\right\|_{L^{2}} \\
& \lesssim \sum_{m=1}^2\sum_{j=2}^{4-m}  \left\|\langle \xi \rangle^{5/2}\widehat{u_0}\right\|_{L^{2}}^{(2j-3)/5}\left\|J_{\xi}^{5m/(8-2j)}\widehat{u_0}\right\|_{L^{2}}^{(8-2j)/5} +\left\|\langle \xi \rangle^{5/2}\widehat{u_0}\right\|_{L^{2}} \\
& \lesssim \left\|J_{\xi}^{5/2}\widehat{u_0}\right\|_{L^{2}}+\left\|\langle \xi \rangle^{5/2}\widehat{u_0}\right\|_{L^{2}}.
\end{aligned}
\end{equation}
Plugging the above conclusion in \eqref{eqthreedt0.4.1} gives
\begin{equation}\label{eqthreedt0.5}
\begin{aligned}
\left\|\langle \xi \rangle^{-3/2}\widetilde{F}_{4,2}^k(t,\xi,\widehat{u_0})\right\|_{L^2} &\lesssim \left\|J^{4}_{\xi}\widehat{u_0}\right\|_{L^{2}}+\left\|\langle \xi \rangle^{5/2}\widehat{u_0}\right\|_{L^{2}}\lesssim \left\|\langle x \rangle^4 u_0\right\|_{L^{2}}+\left\|J^{5/2}u_0\right\|_{L^{2}}.
\end{aligned}
\end{equation}
To treat the  second term on the r.h.s of \eqref{eqthreedt0.4}, in view of Proposition \ref{prelimprop1} with $h=\langle \xi \rangle^{-2}$, we shall estimate the $H^{1/2}_{\xi}(\mathbb{R}^3)$-norm of $\widetilde{F}_{4,2}^k(t,\xi,\widehat{u_0})\phi$. Therefore, setting $a=0$, $g=\phi$ and $b=1/2$ in  \eqref{eqsharp4.1}, after repeated applications of Proposition \ref{lemmasharp3} we find
\begin{equation}\label{eqthreedt0.5.1}
\begin{aligned}
\left\|\tilde{F_{4,2}^{k}}(t,\xi,\widehat{u_0})\phi\right\|_{H^{1/2}_{\xi}} &\lesssim \sum_{l=0}^4 \left\|\partial_{\xi_k}^{l}\widehat{u_0} \right\|_{H^{1/2}_{\xi}}+\sum_{m=0}^2\left\|\partial_{\xi_k}^{l}\widehat{u_0} \right\|_{H^{(1/2)^{+}}_{\xi}} 
\lesssim \left\|\langle x \rangle^{9/2}u_0\right\|_{L^2}.
\end{aligned}
\end{equation}
Next we deal with the remaining term on the r.h.s of \eqref{eqthreedt0.4}. Let us first deduce some additional inequalities. Let $P(\xi)$ be a homogeneous polynomial of degree $k$ with $1\leq k\leq 4$, $l$ an integer number such that $0\leq l\leq k$ and $f$ a sufficiently regular function. Then if $k-l \leq 2$, from \eqref{prelimneq0.11} we get
\begin{equation}\label{eqthreedt0.1}
\begin{aligned}
\left\|\mathcal{D}_{\xi}^{1/2}\big(\langle \xi \rangle^{-2}\frac{P(\xi)}{|\xi|^{l}}f(1-\phi)\big)\right\|_{L^2}\lesssim & \left\|\mathcal{D}_{\xi}^{1/2}\big(\langle \xi \rangle^{-2}\frac{P(\xi)}{|\xi|^{l}} (1-\phi)\big)\right\|_{L^{\infty}}\left\|f\right\|_{L^2} +\left\|\langle \xi \rangle^{-2}\frac{P(\xi)}{|\xi|^{l}} (1-\phi)\right\|_{L^{\infty}}\left\|\mathcal{D}_{\xi}^{1/2}f\right\|_{L^2} \\
\lesssim & \left\|f\right\|_{H^{1/2}_{\xi}},
\end{aligned}
\end{equation}
and when $k-l >2$, 
\begin{equation} \label{eqthreedt0.2}
\begin{aligned}
\left\|\mathcal{D}_{\xi}^{1/2}\big(\langle \xi \rangle^{-2}\frac{P(\xi)}{|\xi|^{l}}f (1-\phi)\big)\right\|_{L^2}\lesssim & \left\|\mathcal{D}_{\xi}^{1/2}\big(\langle \xi \rangle^{l-k}\frac{P(\xi)}{|\xi|^{m}} (1-\phi)\big)\right\|_{L^{\infty}}\left\|\langle \xi \rangle^{k-l-2}f\right\|_{L^2} \\
& +\left\|\langle \xi \rangle^{l-k}\frac{P(\xi)}{|\xi|^{l}} (1-\phi)\right\|_{L^{\infty}}\left\|\mathcal{D}_{\xi}^{1/2}\big(\langle \xi \rangle^{k-l-2}f\big)\right\|_{L^2} \\
\lesssim & \left\|\langle \xi \rangle^{k-l-2}f\right\|_{H^{1/2}_{\xi}}.
\end{aligned}
\end{equation}
In consequence, letting $g=1-\phi$, $a=-2$ and $b=1/2$ in \eqref{eqsharp4.1}, after applying \eqref{eqthreedt0.1}, \eqref{eqthreedt0.2} and \eqref{prelimneq3} to the resulting inequality one has
\begin{equation}\label{eqthreedt0.5.2}
\begin{aligned}
\left\|\langle \xi \rangle^{-2}\tilde{F_{4,2}^k}(t,\xi,\widehat{u_0})(1-\phi)\right\|_{H^{1/2}_{\xi}} 
&\lesssim \left\|\langle \xi \rangle^{5/2}\widehat{u_0}\right\|_{L^2}+\left\|J_{\xi}^{9/2}\widehat{u_0}\right\|_{L^2} \sim  \left\|\langle x \rangle^{9/2}u_0\right\|_{L^2}+\left\|J^{5/2}u_0\right\|_{L^2}.
\end{aligned}
\end{equation}
Finally, collecting \eqref{eqthreedt0.5}, \eqref{eqthreedt0.5.1} and \eqref{eqthreedt0.5.2}, we complete the analysis of the homogeneous part in \eqref{eqthreedt0.6}. The estimate for the integral term is achieved by the same estimates applied to $\xi_1 \widehat{u^2}$ in view of \eqref{eqthreedt0.7}. 

\end{proof}
Summing up,  we can conclude that
\begin{equation}\label{eqthreedt4}
\begin{aligned}
\partial_{\xi_k}^4 \widehat{u}(t) &\in H^{1/2}_{\xi}(\mathbb{R}^3) \hspace{0.2cm} \text{ implies } \\ 
&\langle \xi \rangle^{-2}\partial_{\xi_k}^4 \widehat{u}(t) \in H^{1/2}_{\xi}(\mathbb{R}^3), \hspace{0.2cm} \text{ which holds if and only if }   \\
&\sum_{l=1}^3 \mathcal{C}_l(t_2)\langle \xi \rangle^{-2}\partial_{\xi_k}^4(\xi_1|\xi|)\xi_l\phi(\xi)+4\mathcal{C}_k(t_2)\langle \xi \rangle^{-2}\partial_{\xi_k}^3(\xi_1|\xi|)\phi(\xi) \in  H_{\xi}^{1/2}(\mathbb{R}^3),
\end{aligned}
\end{equation}
for fixed $t \geq 0$, where we have defined  $\mathcal{C}_l(t)$ exactly as in \eqref{eqthreshar3.0} extending to  $l=1,2,3$.
\\ 
We now focus on \eqref{eqthreedt4} when $k=1$. Given $\xi=(\xi_1,\xi_2,\xi_3)\in \mathbb{R}^3$, we denoted by $\tilde{\xi}=(\xi_2,\xi_3)\in \mathbb{R}^2$ and 
\begin{equation*}
\begin{aligned}
G(\xi):&=\sum_{l=1}^3 i\mathcal{C}_l(t_2)\partial_{\xi_1}^4(\xi_1|\xi|)\xi_l\langle \xi \rangle^{-2}+4i\mathcal{C}_1(t_2)\partial_{\xi_1}^3(\xi_1|\xi|)\langle \xi \rangle^{-2} \\
&= |\xi|^{-5}|\tilde{\xi}|^4 \langle \xi \rangle^{-2}\left(-15\sum_{l=1}^3 i\mathcal{C}_l(t)|\xi|^{-2}\xi_1\xi_l+12i\mathcal{C}_1(t)\right).
\end{aligned}
\end{equation*}
Whenever $\mathcal{C}_1(t)\neq 0$ for some $t>0$ fixed, we claim that
\begin{equation}\label{eqthreedt4.3}
\mathcal{D}_{\xi}^{1/2}\big(G(\cdot)\phi\big)\notin L^{2}(\mathbb{R}^3).
\end{equation}
 Since \eqref{eqthreedt4} is valid at $t_2>0$ and $k=1$, once we have established \eqref{eqthreedt4.3}, it must follow that
\begin{equation}\label{eqthreedt4.1}
\mathcal{C}_1(t_2)=0. 
\end{equation}
This in turn allows us to proceed as in the previous subsection to infer Theorem \ref{threetimesharp} in the three-dimensional case. In this manner, it remains to prove claim \eqref{eqthreedt4.3}. Suppose that for some $t>0$, $\mathcal{C}_1(t)\neq 0$, we choose then a fixed constant $K$ satisfying
\begin{equation*}
0< K \leq \min\left\{\frac{1}{15},\frac{|\mathcal{C}_1(t)|}{15|\mathcal{C}_2(t)|},\frac{|\mathcal{C}_1(t)|}{15|\mathcal{C}_3(t)|}\right\}
\end{equation*}
and we define
\begin{equation}
\mathcal{P}_{K}:=\left\{x\in \mathbb{R}^3 : |x|\leq (1-K^2)^{-1/2}|\tilde{x}| \right\}.
\end{equation}
Notice that when $x\in \mathcal{P}_{K}$, one has that $|x_1|\leq K|x|$ and so
\begin{equation*}
\begin{aligned}
\big|15\sum_{l=1}^3 \mathcal{C}_l(t)|x|^{-2}x_1x_l \big| \leq 15\sum_{l=1}^3  |\mathcal{C}_l(t)||x|^{-1}|x_1| \leq 3|\mathcal{C}_1(t)|. 
\end{aligned}
\end{equation*}
In addition, let us consider  
\begin{equation}\label{cond1}
\begin{aligned}
\xi \in \mathcal{P}_{K}\cap \left\{|\xi|\leq 1/16\right\},
\end{aligned}
\end{equation}
and for fixed $\xi$ satisfying the above conditions, take
\begin{equation}\label{cond2}
\begin{aligned}
\eta \in \mathcal{P}_{K}\cap \left\{4|\xi|\leq |\eta| \leq 1/2\right\}.
\end{aligned}
\end{equation}
Therefore, for such $\xi$ and $\eta$, one gets the following lower bound
\begin{equation*}
9\langle \xi \rangle^{-2}|\mathcal{C}_1(t)|\frac{|\tilde{\xi}|^4}{|\xi|^5}\leq |G(\xi)|,
\end{equation*}
and since $|\xi_1-\eta_1| \leq 2K|\xi-\eta|$,
\begin{equation*}
 |G(\xi-\eta)|\leq 18 |C_1(t)|\frac{|\tilde{\xi}-\tilde{\eta}|^4}{|\xi-\eta|^5}.
\end{equation*}
Consequently, collecting the above estimates and using that $3|\xi|,3|\eta|/4 \leq |\xi-\eta|$ and $(8/9)^2 \leq \langle \xi \rangle^{-2}\leq 1$, whenever \eqref{cond1} and \eqref{cond2} hold, we arrive at
\begin{equation}\label{eqthreshar4.2}
\begin{aligned}
|G(\xi)-G(\xi-\eta)|&\geq \frac{9|\mathcal{C}_1(t)|}{|\xi|^5|\xi-\eta|^5}\left((8/9)^2|\xi-\eta|^5|\tilde{\xi}|^4-2|\tilde{\xi}-\tilde{\eta}|^4|\xi|^5\right) \\
&\geq \frac{6|\mathcal{C}_1(t)|}{|\xi|^5}\left(\frac{2^5}{3^3}|\tilde{\xi}|^4-|\xi|^4\right)  \gtrsim_{K,|\mathcal{C}_1|}  \frac{1}{|\xi|}. 
\end{aligned}
\end{equation}
Then, \eqref{eqthreshar4.2} and the fact that $\phi \equiv 1$ when $|\xi|\leq 1$ yield
\begin{equation*}
\begin{aligned}
\big(\mathcal{D}_{\xi}^{1/2}(G(\cdot)\phi)\big)^2(\xi)\chi_{\mathcal{P}_K\cap \left\{|\xi|\leq 1/16\right\}}(\xi) &\geq \int_{\eta\in \mathcal{P}_k \cap \left\{4|\xi|\leq |\eta|\leq 1/2\right\}} \frac{|G(\xi)-G(\xi-\eta)|^2}{|\eta|^4} \,d\eta\, \chi_{\mathcal{P}_K\cap \left\{|\xi|\leq 1/16\right\}}(\xi) \\
& \gtrsim \frac{1}{|\xi|^2}\int_{\eta\in \mathcal{P}_k \cap \left\{4|\xi|\leq |\eta|\leq 1/2\right\}} \frac{1}{|\eta|^4} \,d\eta\, \chi_{\mathcal{P}_K\cap \left\{|\xi|\leq 1/16\right\}}(\xi)  \\
& \gtrsim \frac{1}{|\xi|^3} \chi_{\mathcal{P}_K\cap \left\{|\xi|\leq 1/6\right\}}(\xi).
\end{aligned}
\end{equation*}
Considering that $\frac{1}{|\xi|^{3/2}} \chi_{\mathcal{P}_K\cap \left\{|\xi|\leq 1/16\right\}} \notin L^{2}(\mathbb{R}^3)$, the last inequality establishes \eqref{eqthreedt4.3}. The proof is now completed.
\section{Proof of Theorem \ref{proptwotim}.} 

Without loss of generality we may assume that
\begin{equation}
t_1=0 \hspace{0.5cm} \text{ and } \hspace{0.5cm} \int x_1u_0(x)\, dx=0.
\end{equation}
Let us treat first the two-dimensional case. Collecting \eqref{eqthreshar2.1}, \eqref{eqthreshar4.1} and \eqref{eqthreshar3.2}, we have for $t_2 \neq 0$ that
\begin{equation}
\begin{aligned}
\partial_{\xi_k}^4 \widehat{u}(\cdot,t_2) &\in L^2(\mathbb{R}^2) \hspace{0.2cm} \text{ implies } \\
&\partial_{\xi_k}^4 \widehat{u}(\cdot,t_2) \in L^2(\langle \xi \rangle^{-4} d\xi), \hspace{0.2cm} \text{ this holds if and only if }   \\
&0=\int_0^{t_2} \int x_1 u(x,\tau)\, dx\, d\tau=\frac{1}{2}\int_0^{t_2} \tau\left\|u(\tau)\right\|_{L^2}^2 d\tau=\frac{t_2^2}{4} \left\|u_0\right\|_{L^2}^2,
\end{aligned}
\end{equation}
whenever $k=1,2$. A similar conclusion can be drawn for the three-dimensional case after gathering together \eqref{eqthreedt4}, \eqref{eqthreedt4.1} and \eqref{eqthreshar3.2} to deduce
\begin{equation}
\begin{aligned}
 \partial_{\xi_1}^4 \widehat{u}(\cdot,t_2) &\in H^{1/2}(\mathbb{R}^3) \hspace{0.2cm} \text{ implies } \\
&\langle \xi \rangle^{-2}\partial_{\xi_1}^4 \widehat{u}(\cdot,t_2) \in H^{1/2}(\mathbb{R}^3), \hspace{0.2cm} \text{ which holds if and only if }   \\
&0=\int_0^{t_2} \int x_1 u(x,\tau)\, dx\, d\tau =\frac{1}{2}\int_0^{t_2} \tau \left\|u(\tau)\right\|_{L^2}^2 d\tau=\frac{t_2^2}{4} \left\|u_0\right\|_{L^2}^2.
\end{aligned}
\end{equation}


\section{Proof of Theorem \ref{timesharp} }

 Whenever $u\in C([0,T];\dot{Z}_{s,r_d}(\mathbb{R}^d))$ with $r_2=3$, $r_3=4$ and $s\geq d/2+4$ one has
\begin{equation}\label{eqtimesharp1}
    u\partial_{x_1}u \in L^{\infty}([0,T];Z_{d/2+3,d/2+3}(\mathbb{R}^d)).
\end{equation}
Setting $d=2$, we can employ \eqref{eqtimesharp1} to replace all the $L^{2}(\langle \xi \rangle^{-8} \, d\xi)$ estimates provided in the proof of Theorem \ref{threetimesharp} by their equivalents in the space $L^{2}(\mathbb{R}^2)$. This in turn yields
\begin{equation}
\begin{aligned}
\partial_{\xi_k}^4 &\widehat{u}(\cdot,t) \in L^2(\mathbb{R}^2), \hspace{0.2cm} \text{ if and only if }   \\
&0=\int_0^{t} \int x_1 u(x,\tau)\, dx\, d\tau =\int_0^{t} \int x_1u_0(x)\, dx +\frac{\tau}{2}\left\|u_0\right\|_{L^2}^2 d\tau=0, \hspace{0.2cm} \text{ if and only if } \\
&\hspace{0.6cm} t \left( \int x_1u_0(x)\, dx +\frac{t}{4}\left\|u_0\right\|_{L^2}^2 \right)=0,
\end{aligned}
\end{equation}
for each $k=1,2$. On the other hand, when $d=3$, \eqref{eqtimesharp1} establishes that all the estimates exhibited in the proof of Theorem \ref{threetimesharp} can be achieved directly in the space $H^{1/2}_{\xi}(\mathbb{R}^3)$ without the aim of the weight $\langle \xi \rangle^{-2}$. Consequently, 
\begin{equation}
\begin{aligned}
\partial_{\xi_k}^4 &\widehat{u}(\cdot,t) \in H^{1/2}(\mathbb{R}^3), \hspace{0.2cm} \text{ if and only if }   \\
&\int_0^{t} \int x_1 u(x,\tau)\, dx\, d\tau =0, \hspace{0.2cm} \text{ if and only if }  \\
&0=\int_0^{t} \int x_1 u(x,\tau)\, dx\, d\tau =\int_0^{t} \int x_1u_0(x)\, dx +\frac{\tau}{2}\left\|u_0\right\|_{L^2}^2 d\tau=0, \hspace{0.2cm} \text{ if and only if } \\
&\hspace{0.6cm} t \left( \int x_1u_0(x)\, dx +\frac{t}{4}\left\|u_0\right\|_{L^2}^2 \right)=0,
\end{aligned}
\end{equation}
This completes the proof of the theorem.


\section{Appendix}\label{prelim}

This section is devoted to show Proposition \ref{conmuest}. We begin by introducing some notation and preliminaries.
\\
Let $\psi_0\in C_{c}^{\infty}(\mathbb{R}^d)$ such that
\begin{equation*}
    0\leq \psi_0 \leq 1, \hspace{0.2cm} \psi_0(\xi)=1 \text{ for } |\xi| \leq 1, \hspace{0.2cm} \psi_0(\xi)=0 \text{ for } |\xi|\geq 2,
\end{equation*}
and set $\psi(\xi)=\psi_0(\xi)-\psi_0(2\xi)$ which is supported on $1/2\leq |\xi|\leq 2$. For any $f\in \mathcal{S}(\mathbb{R}^d)$ and $j\in \mathbb{Z}$, we define the Littlewood-Paley projection operators
\begin{equation*}
    \begin{aligned}
    &\widehat{P_{j}f}(\xi)=\psi(2^{-j}\xi)\widehat{f}(\xi), \\
    &\widehat{P_{\leq j}f}(\xi)=\psi_0(2^{-j}\xi)\widehat{f}(\xi), \hspace{0.2cm} \xi \in \mathbb{R}^d
    \end{aligned}
\end{equation*}
and $\tilde{P}_j=\sum_{|k-j|\leq 2}P_{k}$.
 Denoting by $\mathcal{S}'(\mathbb{R}^d)$ the space of tempered distributions, we have:
\begin{lemma}\label{bound} Suppose $u\in S'(\mathbb{R}^d)$ with $\supp(\widehat{u})\subset \left\{\xi:\, |\xi|\leq t\right\}$ for some $t>0$. Then
\begin{equation*}
\sup_{z \in \mathbb{R}^d} \frac{|u(x-z)|}{(1+t|z|)^{d}} \lesssim_{d} \mathcal{M}(u)(x)
\end{equation*}
where $\mathcal{M}(u)$ is the usual Hardy-Littlewood maximal function.
\end{lemma}
\begin{proof}
See for instance \cite[Lemma 2.3]{Dli}.
\end{proof}
Our arguments require the following proposition due to Coifman-Meyer (see \cite{CoifmanMeyer,coifman} and \cite{grafakos}).
\begin{prop}\label{coifmay}
Let $\sigma(\xi,\eta)\in C^{\infty}(\mathbb{R}^d\times \mathbb{R}^d \setminus (0,0))$ satisfying
\begin{equation}\label{CoiMCOndi1}
    |\partial_{\xi}^{\gamma_1}\partial_{\eta}^{\gamma_2}\sigma(\xi,\eta)| \lesssim_{\gamma_1,\gamma_2} (|\xi|+|\eta|)^{-(|\gamma_1|+|\gamma_2|)}
\end{equation} 
for all multi-index $\gamma_1,\gamma_2$ and for all $(\xi,\eta) \neq (0,0)$. Define 
\begin{equation}\label{CoiMOper}
    \sigma(D)(f,g)(x)=\int e^{ix \cdot (\xi+\eta)} \sigma(\xi,\eta) \widehat{f}(\xi)\widehat{g}(\eta) d\xi d\eta.
\end{equation}
Then for any $1<p<\infty$,
\begin{equation*}
    \left\|\sigma(D)(f,g)\right\|_{L^p}\lesssim \left\|f\right\|_{L^{\infty}}\left\|g\right\|_{L^p}.
\end{equation*}
\end{prop}

\subsection{Proof of Proposition \ref{conmuest}.}

Without loss of generality we shall deduce \eqref{conmuest} for $\mathcal{R}_1$. In view of Bony's paraproduct decomposition we write 
\begin{equation*}
    \begin{aligned}
    &\mathcal{R}_1(a\partial^{\alpha}f)-a \mathcal{R}_1\partial^{\alpha}f-\sum_{1\leq |\beta| < |\alpha|}\frac{1}{\beta!}\partial^{\beta}aD_{R_1}^{\beta}\partial^{\alpha}f \\
    &\hspace{0.5cm}=\sum_j \mathcal{R}_1 (P_{< j-2} a P_{j}\partial^{\alpha} f)-P_{< j-2}a P_{j} \partial^{\alpha}\mathcal{R}_1 f -\sum_{1 \leq |\beta| < |\alpha|}\frac{1}{\beta!}\partial^{\beta}P_{<j-2}a P_jD_{R_1}^{\beta}\partial^{\alpha}f \\
    &\hspace{0.9cm}+\sum_j \Big( \mathcal{R}_1 (P_j a P_{< j-2}\partial^{\alpha} f+P_j a \tilde{P}_j\partial^{\alpha} f)-(P_ja P_{< j-2} \partial^{\alpha}\mathcal{R}_1 f+P_ja \tilde{P}_j \partial^{\alpha}\mathcal{R}_1 f)  \\
    &\hspace{3.5cm}-\sum_{1 \leq |\beta| < |\alpha|}\frac{1}{\beta!}(\partial^{\beta}P_ja P_{<j-2}D_{R_1}^{\beta}\partial^{\alpha}f +\partial^{\beta}P_ja\tilde{P}_jD_{R_1}^{\beta}\partial^{\alpha}f)\Big)\\
    &\hspace{0.5cm}=:\pi(lh)+\pi(hl+hh).
    \end{aligned}
\end{equation*}
Here $\pi(lh)$ corresponds to the lower-higher frequencies and $\pi(hl+hh)$ combines the higher-lower and higher-higher iterations. We first estimate $\pi(lh)$. The Littlewood-Paley inequality asserts
\begin{equation*}
\left\|\pi(lh)(f,g)\right\|_{L^p}\sim \left\| \big(P_k \pi(lh)(f,g)\big)_{l^2}\right\|_{L^{p}}.
\end{equation*}
Then by support considerations,
\begin{equation*}
\begin{aligned}
       P_k\pi(lh)&=\sum_{|j-k|\leq 2}-\int i^{|\alpha|+1} \eta^{\alpha}\left(\frac{\xi_1+\eta_1}{|\xi+\eta|}-\frac{\eta_1}{|\eta|}-\sum_{1\leq |\beta|< |\alpha|}\frac{1}{\beta!}\partial^{\beta}\left(\frac{\eta_1}{|\eta|}\right)\xi^{\beta}\right)
       \\&\hspace{3cm} \times \psi_k(\xi+\eta)\psi_{<j-2}(\xi)\psi_{j}(\eta)\widehat{a}(\xi)\widehat{f}(\eta) e^{ix \cdot (\xi+\eta)} d\xi d \eta \\
     &=\sum_{|j-k|\leq 2} \sum_{|\beta|=|\alpha|} \sigma_{\beta,j}(D)(P_{<j-2}\partial^{\beta}a,P_jf),
\end{aligned}
\end{equation*}
where by the Taylor's expansion of the function $|x|^{-1}x_1$, we have defined for each multi-index $\beta$ the bilinear operator $\sigma_{\beta,j}(D)$ as in \eqref{CoiMOper} with associated symbol
\begin{equation*}
\begin{aligned}
     \sigma_{\beta,j}(\xi,\eta)= -\frac{i|\beta|}{\beta!} \, \eta^{\alpha}&\left(\int_0^1 (1-\nu)^{|\beta|-1} \partial^{\beta}_x\left(\frac{x_1}{|x|}\right)(\eta+\nu\xi) \, d\nu\right)  \psi_k(\xi+\eta)\phi^{0}_{<j-2}(\xi)\phi^{1}_{j}(\eta),
\end{aligned}
\end{equation*}
for some suitable bump functions satisfy: $\phi^0_{<j-2}(\cdot)=\phi^0(2^{-(j-3)}\cdot)$, $\phi^{1}_j(\cdot)=\phi(2^{-j}\cdot)$ with $\phi^0 \psi_0=\psi_0$, $\phi^1\psi=\psi$, $\text{dist}(\supp(\phi^1),0)>0$ and such that $\phi^{0}_{<j-2}(\xi)\phi^{1}_{j}(\eta)$ is supported in the region $|\xi|\ll |\eta|$.

Consequently, one can verify that $\sigma_{\beta,j}\in C^{\infty}(\mathbb{R}^d\times \mathbb{R}^d)$ is compact supported outside of the origin in the region $|\xi|\ll |\eta|$ and it satisfies \eqref{CoiMCOndi1} uniformly on $\nu\in [0,1]$, for each $j=k-2,k-1,k,k+1,k+2$. These facts allow us to use the Fourier decomposition on a cube in $\mathbb{R}^d \times \mathbb{R}^d$ of side length $C2^{j}$ for $C$ large to deduce
\begin{equation*}
\sigma_{\beta,j}(\xi,\eta)=\sum_{n_1,n_2 \in \mathbb{Z}^d} c_{n_1,n_2,j} e^{i(n_1\cdot\xi+n_2\cdot \eta)/C2^{j}}
\end{equation*}
where the Fourier coefficients $\left\{c_{n_1,n_2,j}\right\}$ are rapidly decreasing. After this we get
\begin{equation*}
\begin{aligned}
 \sigma_{\beta,j}(D)(P_{<j-2}\partial^{\beta}a,P_jf)(x)=\sum_{n_1,n_2 \in \mathbb{Z}^d} c_{n_1,n_2,j}P_{<j-2}\partial^{\beta}a(x-n_1/C2^j)P_jf(x-n_2/C2^j),
\end{aligned}
\end{equation*}
and so we arrive at
\begin{equation*}
\begin{aligned}
&|P_k\pi(lh)(x)| \\
&\hspace{0.5cm}\lesssim \sum_{|j-k|\leq 2} \sum_{|\beta|=|\alpha|} \sum_{n_1,n_2 \in \mathbb{Z}^d} |c_{n_1,n_2,j}||P_{<j-2}\partial^{\beta}a(x-n_1/C2^j)P_jf(x-n_2/C2^j)|.
\end{aligned}
\end{equation*}
To control the above expression, we use Lemma \ref{bound} to find
\begin{equation*}
|P_{<j-2}\partial^{\beta}a(x-n_1/C2^j)| \lesssim (1+|n_1|)^{d}\mathcal{M}(\partial^{\beta}a)(x),
\end{equation*}
and writing $\psi_k=\phi^1_k\psi_k$, 
\begin{equation*}
|P_jf(x-n_2/C2^j)| \lesssim (1+|n_2|)^{d}\mathcal{M}(P_j f)(x).
\end{equation*}
Gathering the above estimates with the decay of the coefficients $\left\{c_{n_1,n_2,j}\right\}$ yield
\begin{equation*}
\begin{aligned}
&|P_k\pi(lh)(x)| \lesssim \sum_{|j-k|\leq 2} \sum_{|\beta|=|\alpha|} \mathcal{M}(\partial^{\beta}a)(x)\mathcal{M}(P_j f)(x).
\end{aligned}
\end{equation*}
In this manner, the above display, Fefferman-Stein inequality (see \cite{FefermStein}) and the Littlewood-Paley inequality show
\begin{equation*}
\begin{aligned}
\left\|\pi(lh)(f,g)\right\|_{L^p}  \lesssim \sum_{|\beta|=|\alpha|} \left\|\mathcal{M}(\partial^{\beta}a)\big(\mathcal{M}(P_kf)\big)_{l^2}\right\|_{L^p} \lesssim \sum_{|\beta|=|\alpha|} \left\|\partial^{\beta}a\right\|_{L^{\infty}}\left\|f\right\|_{L^p}.
\end{aligned}
\end{equation*}
It remains to derive a bound for $\pi(hl+hh)$. Notice that our previous considerations cannot be adapted to this case, since the support in frequency of $\pi(hl+hh)$ lies in the region $|\eta|\lesssim |\xi |$, where the line segment $\eta+\nu\xi$ can pass through the origin. Instead, we estimate separately each term in $\pi(hl+hh)$. 
\\
Using that $D^{2|\alpha|}=\sum_{|\gamma|=|\alpha|} c_{\gamma} \partial^{\gamma}\partial^{\gamma}$ for some constants $c_{\gamma} \in \mathbb{R}$, we can write
\begin{equation*}
    \begin{aligned}
         \pi(hl+hh) &=\sum_j\sum_{|\gamma|=|\alpha|}c_{\gamma}  \mathcal{R}_1 \big( (D^{-2|\alpha|}\partial^{\gamma}P_j \partial^{\gamma}a ) (P_{\leq j+2}\partial^{\alpha} f)\big) -c_{\gamma}(D^{-2|\alpha|}\partial^{\gamma}P_j \partial^{\gamma}a)(P_{\leq j+2} \partial^{\alpha}\mathcal{R}_1 f) \\
    &\hspace{0.8cm}-c_{\gamma} \sum_{1 \leq |\beta| < |\alpha|}\frac{1}{\beta!}\big(\partial^{\beta}(D^{-2|\alpha|}\partial^{\gamma}P_j \partial^{\gamma}a)( P_{\leq j+2}D_{R_1}^{\beta}\partial^{\alpha}f)\big)\\
    &=: \sum_{|\gamma|=|\alpha|} c_{\gamma}\mathcal{R}_1 \sigma^{\ast}_{1,\gamma}(D)(\partial^{\gamma}a,f)+c_{\gamma}\sigma^{\ast}_{1,\gamma}(D)(\partial^{\gamma}a,\mathcal{R}_1f)+c_{\gamma}\sigma^{\ast}_{2,\gamma}(D)(\partial^{\gamma}a,f)
    \end{aligned}
\end{equation*}
where the operators $\sigma^{\ast}_{1,\gamma}(D)$ are defined through the symbols
\begin{equation*}
    \begin{aligned}
         \sigma_{1,\gamma}^{\ast}(\xi,\eta)=\sum_j i^{|\gamma|+|\alpha|}\frac{\xi^{\gamma}}{|\xi|^{2|\alpha|}} \eta^{\alpha} \psi_j(\xi)\psi_{\leq j+2}(\eta),
    \end{aligned}
\end{equation*}
and 
\begin{equation*}
    \begin{aligned}
         \sigma_{2,\gamma}^{\ast}(D)(\partial^{\gamma}a,f) \hspace{0.2cm}=\sum_j\sum_{1\leq|\beta|<|\alpha|}\int \frac{(-1)^{|\alpha|+1}i}{\beta!}\partial^{\beta}\left(\frac{\eta_1}{|\eta|}\right)\frac{\xi^{\beta} \xi^{\gamma}}{|\xi|^{2|\alpha|}}\eta^{\alpha}  \psi_j(\xi)\psi_{\leq j+2}(\eta)\widehat{\partial^{\gamma}a}(\xi)\widehat{f}(\eta)\, d\xi d\eta,
    \end{aligned}
\end{equation*}
for each $|\gamma|=|\alpha|$. Using that $\sigma_{1,\gamma}(\xi,\eta)$ is supported in the region $|\eta|\lesssim |\xi|$, it is easily seen that this operator satisfies the hypothesis of Proposition \ref{coifmay}. Consequently, the $L^p$ boundedness of the Riesz transform yields 
\begin{equation*}
\begin{aligned}
  \left\|\mathcal{R}_1 \sigma_{1,\gamma}^{\ast}(D)(\partial^{\gamma}a,f)+\sigma_{1,\gamma}^{\ast}(D)(\partial^{\gamma}a,\mathcal{R}_1f)  \right\|_{L^p} \lesssim \left\|\partial^{\gamma}a\right\|_{L^{\infty}}\left\|f\right\|_{L^p},
  \end{aligned}
\end{equation*}
for all $|\gamma|=|\alpha|$. 
On the other hand, we divide the operator $\sigma^{\ast}_{2,\gamma}(D)$ by choosing (fixed) multi-indexes $\alpha(k)$ with $1\leq k < |\alpha|$ satisfying, $\alpha(k) \leq \alpha$ and $|\alpha(k)|=k$. Then we write
\begin{equation*}
 \begin{aligned}
        \sigma_{2,\gamma}^{\ast}(D)(\partial^{\gamma}a,f)=\sum_{1\leq |\beta|<|\alpha|}\sigma_{2,\gamma,\beta}^{\ast}(D)(\partial^{\gamma}a,T_{\beta}f),
    \end{aligned}
\end{equation*}
where for each $|\beta|=k$, $k=1,\dots,|\alpha|-1$ we have set
\begin{equation*}
    \sigma_{2,\gamma,\beta}^{\ast}(\xi,\eta)=\sum_j \frac{(-1)^{|\alpha|+1}i}{\beta!}\frac{\xi^{\beta}\xi^{\gamma}}{|\xi|^{2|\alpha|}}\eta^{\alpha-\alpha(k)}  \psi_j(\xi)\psi_{\leq j+2}(\eta)
\end{equation*}
and the operators
\begin{equation*}
    T_{\beta}(f)(x)=\int \eta^{\alpha(k)}\partial^{\beta}\left(\frac{\eta_1}{|\eta|}\right)\widehat{f}(\eta)e^{ix\cdot \eta}\, d \eta.
\end{equation*}
One can verify that $\sigma_{2,\gamma,\beta}^{\ast}(\xi,\eta)$ satisfies the hypothesis of Proposition \ref{coifmay} for each $1\leq |\beta|<|\alpha|$. Additionally,  the classical Mikhlin multiplier theorem establishes that $T_{\beta}$ defines a bounded operator from $L^p(\mathbb{R}^d)$ to $L^p(\mathbb{R}^d)$, whenever $1<p<\infty$. Summarizing we conclude
\begin{equation*}
\begin{aligned}
\left\|\sigma_{2,\gamma}^{\ast}(D)(\partial^{\gamma}a,f)\right\|_{L^p} \lesssim \sum_{1\leq |\beta|<\alpha} \left\|\sigma_{2,\gamma,\beta}^{\ast}(\partial^{\gamma}a,T_{\beta}f)\right\|_{L^{p}} &\lesssim \sum_{1\leq |\beta|<\alpha} \left\| \partial^{\gamma} a \right\|_{L^{\infty}}\left\|T_{\beta} f\right\|_{L^p}\\
&\lesssim \left\| \partial^{\gamma} a \right\|_{L^{\infty}}\left\|f\right\|_{L^p}.
\end{aligned}
\end{equation*}
This completes the estimate for $\pi(hl+hh)$ and in consequence the proof of Proposition \ref{coifmay}.


\subsection*{Acknowledgements}

This work was supported by  CNPq Brazil. The author wishes to express his gratitude to Prof. F. Linares and A. Mendez for several helpful comments improving this document.


\bibliographystyle{acm}
\bibliography{bibli}

\begin{thebibliography}{10}

\bibitem{A}
{\sc Abramyan, L.~A., Stepanyants, Y.~A., and Shrira, V.~I.}
\newblock {Multidimensional solitons in shear flows of the boundary-layer
  type}.
\newblock {\em Sov. Phys. Dokl 37}, 12 (1992), 575--578.

\bibitem{BS}
{\sc Bona, J.~L., and Smith, R.}
\newblock {The Initial-Value Problem for the Korteweg-De Vries Equation}.
\newblock {\em Philosophical Transactions of the Royal Society of London.
  Series A, Mathematical and Physical Sciences 278}, 1287 (1975), 555--601.

\bibitem{cald}
{\sc Calder\'on, A.~P.}
\newblock {Commutators of Singular Integral Operators}.
\newblock {\em Proc. Natl. Acad. Sci. USA. 53}, 5 (1965), 1092--1099.

\bibitem{CoifmanMeyer}
{\sc Coifman, R., and Meyer, Y.}
\newblock {On Commutators of Singular Integrals and Bilinear Singular
  Integrals}.
\newblock {\em Transactions of the American Mathematical Society 212\/} (1975),
  315--331.

\bibitem{coifman}
{\sc Coifman, R., and Meyer, Y.}
\newblock {\em Au-del{\`a} des op{\'e}rateurs pseudo-diff{\'e}rentiels}.
\newblock Ast\'erisque 57, Soci{\'e}t{\'e} Math{\'e}matique de France, 1978.

\bibitem{CUNHA}
{\sc Cunha, A., and Pastor, A.}
\newblock {The IVP for the Benjamin-Ono-Zakharov-Kuznetsov equation in weighted
  Sobolev spaces}.
\newblock {\em Journal of Mathematical Analysis and Applications 417}, 2
  (2014), 660 -- 693.

\bibitem{FefermStein}
{\sc Fefferman, C., and Stein, E.~M.}
\newblock {Some Maximal Inequalities}.
\newblock {\em American Journal of Mathematics 93}, 1 (1971), 107--115.

\bibitem{FLinaPonceWeBO}
{\sc Fonseca, G., Linares, F., and Ponce, G.}
\newblock {The IVP for the Benjamin-Ono equation in weighted Sobolev spaces
  II}.
\newblock {\em Journal of Functional Analysis 262}, 5 (2012), 2031 -- 2049.

\bibitem{FLinaPioncedGBO}
{\sc Fonseca, G., Linares, F., and Ponce, G.}
\newblock {The IVP for the dispersion generalized Benjamin-Ono equation in
  weighted Sobolev spaces}.
\newblock {\em Annales de l'Institut Henri Poincare (C) Non Linear Analysis
  30}, 5 (2013), 763 -- 790.

\bibitem{FonPO}
{\sc Fonseca, G., and Ponce, G.}
\newblock {The IVP for the Benjamin-Ono equation in weighted Sobolev spaces}.
\newblock {\em Journal of Functional Analysis 260}, 2 (2011), 436 -- 459.

\bibitem{grafakos}
{\sc Grafakos, L.}
\newblock {\em Modern Fourier Analysis}.
\newblock Graduate Texts in Mathematics. Springer New York, 2014.

\bibitem{Kenig}
{\sc Ionescu, A.~D., and Kenig, C.~E.}
\newblock {Global Well-Posedness of the Benjamin-Ono Equation in Low-Regularity
  Spaces}.
\newblock {\em Journal of the American Mathematical Society 20}, 3 (2007),
  753--798.

\bibitem{Dli}
{\sc Li, D.}
\newblock {On Kato-Ponce and Fractional Leibniz}.
\newblock {\em Revista Matem\'atica Iberoamericana 35}, 1 (2019), 23--100.

\bibitem{linaO}
{\sc {Linares}, F., {Ria{\~n}o}, O.~G., {Rogers}, K.~M., and {Wright}, J.}
\newblock {On a higher dimensional version of the Benjamin--Ono equation}.
\newblock {\em arXiv e-prints\/} (2019), arXiv:1901.04817.

\bibitem{M}
{\sc Mari\c{s}, M.}
\newblock {On the Existence, Regularity and Decay of Solitary Waves to a
  Generalized Benjamin-Ono Equation}.
\newblock {\em Nonlinear Anal. 51}, 6 (2002), 1073--1085.

\bibitem{molinetPilodBO}
{\sc Molinet, L., and Pilod, D.}
\newblock {The Cauchy problem for the Benjamin-Ono equation in $L^2$
  revisited}.
\newblock {\em Anal. PDE 5}, 2 (2012), 365--395.

\bibitem{NahPonc}
{\sc Nahas, J., and Ponce, G.}
\newblock {On the Persistent Properties of Solutions to Semi-Linear
  Schr\"odinger Equation}.
\newblock {\em Communications in Partial Differential Equations 34}, 10 (2009),
  1208--1227.

\bibitem{PS}
{\sc Pelinovsky, D.~E., and Shrira, V.~I.}
\newblock {Collapse transformation for self-focusing solitary waves in
  boundary-layer type shear flows}.
\newblock {\em Physics Letters A 206}, 3 (1995), 195 -- 202.

\bibitem{sharpRiesz}
{\sc Petermichl, S.}
\newblock {The sharp weighted bound for the Riesz transforms}.
\newblock {\em {Proc. Amer. Math. Soc} 136}, 4 (2008), 1237 -- 1249.

\bibitem{Ponce1991}
{\sc Ponce, G.}
\newblock {On the global well-posedness of the Benjamin-Ono equation}.
\newblock {\em {Differential Integral Equations} 4}, 3 (1991), 527--542.

\bibitem{RobertS}
{\sc {Schippa}, R.}
\newblock {On the Cauchy problem for higher dimensional Benjamin-Ono and
  Zakharov-Kuznetsov equations}.
\newblock {\em arXiv e-prints\/} (2019), arXiv:1903.02027.

\bibitem{SteinThe}
{\sc Stein, E.~M.}
\newblock {The characterization of functions arising as potentials}.
\newblock {\em Bull. Amer. Math. Soc. 67}, 1 (01 1961), 102--104.

\bibitem{TaoBO}
{\sc Tao, T.}
\newblock {Global well-posedness of the Benjamin-Ono equation in
  $H^1(\mathbb{R})$}.
\newblock {\em Journal of Hyperbolic Differential Equations 01}, 01 (2004),
  27--49.

\bibitem{VS}
{\sc Voronovich, V.~V., and Shrira, V.~I.}
\newblock Internal wave--shear flow resonance and wave breaking in the
  subsurface layer.
\newblock In {\em Nonlinear instability analysis}, vol.~II of {\em Advances in
  fluid mechanics}. 28 WIT Press, 2001, pp.~133--177.

\end{thebibliography}

\end{document}